\documentclass[12pt,a4paper]{amsart}

\usepackage{amsfonts,amsmath,amsthm,amsfonts,amssymb,amstext}
\usepackage{enumitem}
\usepackage[normalem]{ulem}    
\usepackage{mathrsfs}
\usepackage[margin=1.4in,marginpar=1in]{geometry}
\usepackage[all]{xy}    
\usepackage[dvipsnames]{xcolor}
\usepackage{graphicx}      
\usepackage{epstopdf} 
\usepackage{verbatim}
\usepackage[ansinew]{inputenc}
\usepackage{hyperref}  
   
\newcommand{\foot}[1]{\mbox{}\marginpar{\raggedleft\hspace{0pt}\tiny #1}}
\renewcommand{\foot}[1]{}

\newcommand{\todelete}[1]{{\color{red} \sout{#1}}}
\renewcommand{\todelete}[1]{}

\newcommand{\eps}{\epsilon}

\newcommand{\diam}{\operatorname{diam}}

\newcommand{\supp}{\operatorname{supp}}

\newcommand{\wQ}{\widehat{Q}}
\newcommand{\Gpq}{\Gamma_{pq}}

\newcommand{\barx}{{\bar{x}}}
\newcommand{\barl}{{\bar{\ell}}}

\newcommand{\quand}{\quad\text{ and } \quad }
\newcommand{\qand}{\quad\text{and}\quad}

\def \RR {{\mathbb R}}
\def \ZZ {{\mathbb Z}}
\def \NN {{\mathbb N}}

\DeclareMathOperator{\Id}{Id}

\DeclareMathOperator{\Int}{Int}

\DeclareMathOperator{\inj}{inj}
\DeclareMathOperator{\spn}{span}

 \makeatletter
 \def\part{\@startsection{part}{0}%
  \z@{2\linespacing\@plus\linespacing}{.5\linespacing}%
  {\large\normalfont\bfseries\raggedright}}
 \makeatother

\newcommand{\dem}{\begin{proof}}
\newcommand{\cqd}{\end{proof}}

\newtheorem*{Theorem*}{Theorem}
\newtheorem*{Definition*}{Definition}
\newtheorem{maintheorem}{Theorem}

\newtheorem{maincorollary}{Corollary}[maintheorem]

\newtheorem{Proposition}{Proposition}[section]
\newtheorem{Lemma}[Proposition]{Lemma}
\newtheorem{theorem}[Proposition]{Theorem}
\newtheorem{corollary}[Proposition]{Corollary}
\newtheorem{Sublemma}[Proposition]{Sublemma}

\newtheorem*{Conjecture}{Conjecture}

\theoremstyle{definition}
\newtheorem{Remark}[Proposition]{Remark}
\newtheorem{Definition}[Proposition]{Definition}
\newtheorem{Example}[Proposition]{Example}

\numberwithin{equation}{section}
\numberwithin{figure}{section}

\allowdisplaybreaks

\begin{document}

\title
[SRB measures and Young towers on surfaces]
{SRB measures and Young Towers \\ for surface diffeomorphisms}

\author{Vaughn Climenhaga}
\address{Department of Mathematics, University of Houston, Houston, Texas 77204, USA}
\email{climenha@math.uh.edu}

\author{Stefano Luzzatto}
\address{Abdus Salam International Centre for Theoretical Physics (ICTP), Strada Costiera 11, Trieste, Italy}
\email{luzzatto@ictp.it}

\author{Yakov Pesin}
\address{Department of Mathematics, McAllister Building, Pennsylvania State University, University Park, PA 16802, USA}
\email{pesin@math.psu.edu}

\date{\today}

\keywords{Nonuniform hyperbolicity, Young towers, SRB measures, Pesin theory, decay of correlations}

\subjclass[2010]{37D25 (primary); 37C40, 37E30 (secondary)}

\begin{abstract}
We give geometric conditions that are necessary and sufficient  for the existence of Sinai--Ruelle--Bowen (SRB) measures for  $C^{1+\alpha}$ surface diffeomorphisms, thus proving a version of the Viana conjecture.   As part of our argument we give an original method for constructing first return Young towers, proving that every hyperbolic measure, and in particular every SRB measure,  can be lifted to such a tower. This method relies on a new general result on hyperbolic branches and shadowing for pseudo-orbits in nonuniformly hyperbolic sets which is of independent interest. 
\end{abstract}

\maketitle

\setcounter{tocdepth}{1}
\tableofcontents

\part{Statements of Results}
\label{part:results}

The purpose of this paper is to study the connection between analytic, geometric, dynamical, and statistical properties of surface diffeomorphisms.
In particular, we are interested in the way that certain analytic properties,  such as hyperbolicity, imply non-trivial geometric structures which in turn produce non-trivial dynamics and statistical behavior. 
Although we consider only the two-dimensional case, our results on pseudo-orbits and shadowing (Theorems \ref{thm:pseudo-orbit} and \ref{thm:inf-po}) should extend to higher dimensions as well.

In \S\ref{sec:intro} we discuss the general philosophy and theoretical framework of our study,  define   \emph{Sinai-Ruelle-Bowen (SRB) measure} and  recall the \emph{Viana conjecture} on the existence of SRB measures. In 
\S\ref{sec:PgivesSRB} we state Theorem~\ref{thm:existsSRB} which, roughly speaking, says that under some mild recurrence condition,
\begin{center}
\textbf{\emph{a fat (nonuniformly) hyperbolic set supports an SRB measure,}}
\end{center}
thus proving a version of the Viana conjecture in the two-dimensional setting. We note that, unlike most previous results in this direction, our assumptions are also \emph{necessary}, thus giving an interesting  geometric characterization of SRB measures. We give a more detailed review of existing results in \S\ref{sec:histsrb}. 

Our construction of the  SRB measure uses the technique of \emph{Young towers}, which gives additional  information about the geometry and structure of the measure. 
In \S\ref{sec:P'givestower}  we state Theorem~\ref{thm:existsY} which, roughly speaking, says that under some mild recurrence condition, a (nonuniformly) hyperbolic set supports a ``topological'' Young tower and, more specifically,
\begin{center}
\textbf{\emph{a fat (nonuniformly) hyperbolic set supports a Young tower}}.
\end{center}
This result implies Theorem~\ref{thm:existsSRB} but is of independent interest. In \S\ref{sec:hypmeas} we state Theorem \ref{thm:existsnice1}, which says that the assumptions of Theorems \ref{thm:existsSRB} and \ref{thm:existsY} are necessary for the existence of an SRB measure. 
The following consequences are worth highlighting here:
\begin{itemize}
\item every SRB measure, and more generally every \emph{hyperbolic} measure, is \emph{liftable} to a topological Young tower (Corollary \ref{thm:lifts});
\item the towers we produce have the first return property (Theorem \ref{thm:existsY});
\item we formulate explicit conditions under which the decay rate of the tail of the tower can be controlled (Corollary \ref{thm:edc}).
\end{itemize}

Our construction of a Young tower works in a general setting and differs from other constructions  in the literature.\footnote{These typically use
specific geometric characteristics of the system under consideration.} 
The starting point is a measurable subset $A$ of a (non-invariant) ``uniformly'' hyperbolic set bounded by a \emph{nice domain}.  Using an abstract argument we extend $A$ to a \emph{rectangle} $\Gamma$ -- a subset  with product structure of local stable and unstable curves -- which is maximal in a sense, allowing us to build a tower. The key step in producing $\Gamma$ is Theorem \ref{thm:genkatok2} in  \S\ref{sec:hyplyap}, which states that to every \emph{almost return} to $A$ one can associate a \emph{hyperbolic branch}; the total collection of such branches ``saturates'' $A$ to the desired rectangle $\Gamma$. 
The proof of Theorem~\ref{thm:genkatok2} is based on two general results, which we state as Theorems \ref{thm:pseudo-orbit} and \ref{thm:inf-po} in 
\S\ref{sec:reg-branches} and \S\ref{sec:pseudocurves} respectively. Theorems \ref{thm:genkatok2}, \ref{thm:pseudo-orbit}, and \ref{thm:inf-po} are new results in non-uniform hyperbolicity theory of independent interest, with Theorem \ref{thm:pseudo-orbit} providing a new version of Katok's closing lemma and Theorem \ref{thm:inf-po} giving a new version of the shadowing property.

In Part \ref{part:results} of the paper we state all our results.  In Part~\ref{part:hyp} we state and prove  Theorem \ref{thm:pseudo-orbit} and \ref{thm:inf-po} which, as mentioned above, are general results in the theory of nonuniform hyperbolicity. Part~\ref{part:Y} is devoted to the proofs of the remaining results in our more specific setting. These results have a  clear logical interdependence as follows: 
\[
\text{\ref{thm:pseudo-orbit} }  \Longrightarrow 
\text{ \ref{thm:genkatok2} } \Longrightarrow 
\text{ \ref{thm:existsY} } \Longrightarrow 
\text{ \ref{thm:existsSRB}\ref{A1} }
\quad\text{ and } \quad 
\text{ \ref{thm:inf-po} } \Longrightarrow 
\text{ \ref{thm:existsnice1} } \Longrightarrow 
\text{ \ref{thm:existsSRB}\ref{A2}.}
\]
The letters above refer to the corresponding Theorems. Theorem \ref{thm:existsSRB} has two parts: \ref{thm:existsSRB}\ref{A1} states our sufficient conditions for the existence of an SRB measure, and \ref{thm:existsSRB}\ref{A2} states that these conditions are necessary. More details on organization and the relations between the various results are given at the beginning of Parts \ref{part:hyp} and \ref{part:Y}.
In an appendix we provide a list of terminology together with references to the relevant definitions.

See \S\ref{sec:histsrb} and \S\ref{sec:tower-history} for a discussion of related prior work, especially that of Young \cite{lY98,lY99} and Sarig \cite{oS13}.

\subsection*{Acknowledgments}
The authors wish to thank Stefano Bianchini who helped keep this project alive with many useful comments and suggestions when we had almost given up, Dima Dolgopyat for pointing out a missing argument in a previous version, and Vilton Pinheiro for helping us fill in the missing argument. 
We are also grateful to the anonymous referee for a very careful reading and many useful suggestions, which led to corrections and clarifications that have substantially improved the paper.
V.~C.\ is partially supported by NSF grants DMS-1362838 and DMS-1554794. Ya.~P.\ was partially supported by NSF grant DMS-1400027.

\section{SRB measures and the Viana conjecture: Theorem \ref{thm:existsSRB}}
\label{sec:intro}

Throughout this paper,  let  $M$ be a surface -- by which we mean a compact boundaryless smooth 
$2$-dimensional Riemannian manifold -- and let $f\colon M\to M$ be a $C^{1+\alpha}$ diffeomorphism, where $\alpha\in (0,1]$. Let $d(\cdot,\cdot)$ denote the distance function on $M$, and let $m$ denote Lebesgue measure on $M$; that is, the area form induced by the Riemannian metric.  Given a curve $W\subset M$, we write $m_W$ for the one-dimensional Lebesgue measure on $W$ defined by the induced Riemannian metric.  By ``measurable'' we always mean ``Borel measurable''.

\subsection{Physical measures} 
The first step in the statistical description of the diffeomorphism \( f \) is the notion of the ``statistical basin of attraction'' of a probability measure \( \mu \):
\[
\mathcal{B}_\mu := \bigg\{x\in M : \lim_{n\to\infty} \frac 1n \sum_{k=0}^{n-1} \phi(f^k x) = \int \phi\,d\mu \text{ for all continuous } \phi\colon M\to \RR\bigg\}.
\]
Equivalently, $\mathcal{B}_\mu$ consists of all points for which $\frac 1n \sum_{k=0}^{n-1} \delta_{f^k(x)}$ converges to $\mu$ in the weak* topology, where $\delta_y$ is the Dirac delta measure on $y$.
If $\mu$ is $f$-invariant and ergodic, then $\mu(\mathcal{B}_\mu)=1$ by the Birkhoff ergodic theorem, but since Lebesgue measure is the most natural reference measure, we are most interested in finding $\mu$ for which $\mathcal{B}_\mu$ is large in the following sense.

\begin{Definition}\label{def:physical}
 \( \mu \) is a \emph{physical measure} for \( f \) if \( m(\mathcal B_{\mu})>0 \). 
\end{Definition}

Thus a physical measure is a probability measure which describes the asymptotic statistical behavior of a significant (positive Lebesgue measure) subset of the phase space. Not all dynamical systems admit physical measures,\footnote{Consider the identity map, or see  \cite[p.\ 140]{aK80} for a more interesting example, sometimes referred to as ``Bowen's eye''.}
so it is a basic problem to establish the class of dynamical systems which have physical measures before going on to investigate further questions related to the possible number of such measures and their structure and properties. 

The simplest example of a physical measure is given by the Dirac-delta measure $\delta_p$ at an attracting fixed point $p$. This easily generalizes to the case when $p$ is an attracting periodic point.  At the other extreme, if $\mu\ll m$ is ergodic, then $\mu(\mathcal{B}_\mu)=1$ gives $m(\mathcal{B}_\mu)>0$, hence $\mu$ is physical.  Unfortunately it is relatively rare for such absolutely continuous measures to exist, and thus the problem of  the existence of a physical measure is quite non-trivial.

\subsection{Hyperbolic measures}

In the 1970's, Sinai, Ruelle, and Bowen established existence, as well as geometric and statistical properties, of physical measures for \emph{uniformly hyperbolic} systems. The theory of
\emph{Sinai--Ruelle--Bowen measures}, or \emph{SRB measures},
has since been extended to non-uniform hyperbolicity, in which setting we need the following definition.

\begin{Definition}[Hyperbolic measures and non-zero Lyapunov exponents]
\label{def:nonzero}
An invariant probability measure \( \mu \) is \emph{hyperbolic} if there exists a set \( \Lambda\subseteq M \) with \( f(\Lambda)=\Lambda \) and \( \mu(\Lambda)=1 \) which has \emph{non-zero Lyapunov exponents}, i.e. there exists a measurable \( Df \)-invariant  decomposition \( T_{x}M=E^{s}_{x}\oplus E^{u}_{x} \) such that for every \( x\in \Lambda \) and unit vectors \( e^{s}\in E^{s}_{x}, e^{u}\in E^{u}_{x} \) we have:
\begin{enumerate}
\item \( \displaystyle{\lim_{n\to\pm\infty}\frac1n\log \|Df_{x}^{n}(e^{s})\| =: \lambda^{s}_{x} < 0 < \lambda^{u}_{x} :=  \lim_{n\to\pm\infty}\frac1n\log \|Df_{x}^{n}(e^{u})\|}
\); 
\item \( \displaystyle{\lim_{n\to\pm\infty} \frac1n\log\measuredangle (E^{s}_{f^{n}(x)}, E^{u}_{f^{n}(x)})=0}
\).
\end{enumerate}
\end{Definition} 
The heart of this definition is that the \emph{Lyapunov exponents} $\lambda_x^s$ and $\lambda_x^u$ are \emph{non-zero} and have \emph{opposite signs};\footnote{If both Lyapunov exponents are negative or both are positive, then it can be shown that the corresponding ergodic component of the measure \( \mu \) is supported on an attracting or repelling periodic orbit respectively; we exclude this trivial situation.} the fact that the limits exist is guaranteed by Oseledets' Multiplicative Ergodic Theorem, which also guarantees that
although the angle between the two subspaces is not in general  bounded away from zero, it cannot degenerate at an exponential rate along any given orbit, as stated in  condition (2).    
We point out that in the $2$-dimensional case the Margulis--Ruelle inequality implies that every measure with positive entropy is hyperbolic.

\subsection{Sinai-Ruelle-Bowen (SRB) measures}\label{sec:sinai-ruelle-bowen-measures}
A fundamental and crucial property of sets with non-zero Lyapunov exponents is that every point \( x\in \Lambda \) has a \emph{local stable curve} \( V^{s}_{x} \) and a \emph{local unstable curve} \( V^{u}_{x} \) satisfying certain properties which we describe  in Definition \ref{def:localman} below.\footnote{In fact, the existence of local stable and unstable curves can be proved under  weaker conditions than those of non-zero Lyapunov exponents, see Definition \ref{def:hypset} and Theorem~\ref{thm:HP}.} 
For the moment we use these curves to give the formal definition of SRB measure. 

\begin{Definition}[Fat sets]\label{def:fat}
A set \( A \subseteq \Lambda\) is \emph{fat} if 
\begin{equation}\label{eq:fat}
m\bigg(\bigcup_{x\in A}V^{s}_{x}\bigg)>0.
\end{equation}
\end{Definition}

We can now give the definition of SRB measure that we use in this paper.

\begin{Definition}[SRB measures]\label{srb-mes}
An invariant probability  measure \( \mu \) is an  \emph{SRB measure}
if it is hyperbolic and every set $X\subset \Lambda$ with $\mu(X)=1$ is fat.
\end{Definition}

One of the key properties of  \( V^{s}_{x} \) is that \( d(f^{n}(y), f^{n}(x)) \to 0 \) as \( n\to \infty \) for every \( y\in V^{s}_{x} \). This implies that if \( x\in \mathcal B_{\mu} \) for some measure \( \mu \) then also \( y\in \mathcal B_{\mu} \) for every \( y\in V^{s}_{x} \).  Therefore the fatness condition \eqref{eq:fat} together with Birkhoff's Ergodic Theorem implies that any ergodic SRB measure is  a physical measure.\footnote{
 The converse is not true: for example, if $p$ is a hyperbolic fixed point whose stable and unstable curves form a figure-eight, then $\delta_p$ is a hyperbolic physical measure which is not  SRB \cite[p.\ 140]{aK80}.}
In his plenary lecture at the ICM in Berlin in 1998, Viana formulated the following natural conjecture. 

\begin{Conjecture}[Viana \cite{Via98}] 
If a smooth map has only non-zero Lyapunov exponents at Lebesgue almost every point, then it admits some SRB measure.
\end{Conjecture}

\begin{Remark}\label{rem:srb}
In the non-uniform hyperbolicity theory a common definition of SRB measure $\mu$ requires only that ``the conditional measures generated by $\mu$ on unstable manifolds are absolutely continuous with respect to the leaf volume on these manifolds'' \cite{LedYou,lY02}. 
Such measures may have some zero Lyapunov exponents in directions transversal to unstable manifolds and hence need not be hyperbolic. The advantage of this more general definition is that SRB measures are characterized as the only measures that satisfy the entropy formula. On the other hand, these ``non-hyperbolic'' SRB measures may not be physical, and some authors adopt a different convention beyond uniform hyperbolicity, in which ``SRB measure'' simply means ``physical measure'' \cite[Chapter 11]{BDV05}; this appears to be the intent of the Viana conjecture.

We stress that our definition is in fact equivalent to the requirement of being ``hyperbolic with absolutely continuous conditionals along unstable manifolds'' \cite[Theorem C]{Tsu91} although we emphasize the fatness condition which is easier to state and is ultimately the crucial property for proving physicality of the measure.
\end{Remark}

In this paper we use the fatness condition to prove a version of the Viana conjecture for surface diffeomorphisms under the hyperbolicity conditions \eqref{eq:slowvar}, \eqref{eq:angle}, and \eqref{eq:hypest} (see Definition \ref{def:hypset} below; these conditions are weaker than the non-zero Lyapunov exponents condition in Definition \ref{def:nonzero}) but with the addition of a mild \emph{recurrence} condition.
In \S\S\ref{sec:hypsets}--\ref{sec:definition-of-constants} we give the exact definitions we need for the formal statement of our result in \S\ref{sec:PgivesSRB}; then in \S\ref{sec:histsrb} we discuss previous literature on the topic. 
\subsection{Hyperbolic sets}\label{sec:hypsets}
The requirement that the limits 
in Definition \ref{def:nonzero}  exist is somewhat unnatural in a setting where we do not \emph{a priori} have an invariant probability measure,
and also obscures
a crucial feature of sets with non-zero Lyapunov exponents, which is the fact that the convergence to the limit can be  very \emph{non-uniform}. We give here an alternative  formulation of hyperbolicity, which is slightly more technical, but is more general and explicit about the intrinsic non-uniformity.

\begin{Definition}[Hyperbolic set]\label{def:hypset}
Given $\chi,\eps>0$,  we say that an \( f \)-invariant measurable set \( \Lambda \) is   \emph{$(\chi,\eps)$-hyperbolic} if there exists  a measurable \( Df \)-invariant  splitting
\(
T_{x}M=E^{s}_{x}\oplus E^{u}_{x} 
\)
for all $x\in \Lambda$, and  measurable positive functions  $C,K\colon \Lambda \to (0,\infty)$ satisfying  
\begin{equation}\label{eq:slowvar}\tag{\protect\(\mathbf{H1}\protect\)}
 e^{-\epsilon}  \leq K(f(x)) /K (x)\leq e^{\epsilon}
 \quad \text{ and } \quad 
 e^{-\epsilon}  \leq C (f(x)) /C (x)\leq e^{\epsilon} 
\end{equation}
such that for every \( x \in \Lambda \)
\begin{equation}\label{eq:angle}
\tag{\protect\(\mathbf{H2}\protect\)}
 \measuredangle (E^s_x, E^u_x) \geq  K(x)
\end{equation}
and  for all  unit vectors \( e_x^{s}\in E^{s}_{x}, e_x^{u}\in E^{u}_{x} \) and  for all \( n \geq 1 \),
 \begin{equation}\label{eq:hypest}\tag{\protect\(\mathbf{H3}\protect\)}
 \begin{aligned}
 \|Df^{n}_{x}(e_x^{s})\| &\leq C(x) e^{-\chi n}, 
 \qquad &\|Df^n_x (e_x^u)\| &\geq C(x)^{-1} e^{\chi n}, \\
 \|Df^{-n}_x (e_x^s)\| &\geq C(x)^{-1} e^{\chi n},
&
 \|Df^{-n}_{x}(e_x^{u})\| &\leq C(x) e^{-\chi n}.
 \end{aligned}
\end{equation}
A set $\Lambda$ is \emph{$\chi$-hyperbolic} if it is $(\chi,\eps)$-hyperbolic for all $\eps>0$, and \emph{hyperbolic} if it is a union of \( \chi \)-hyperbolic sets over all \( \chi>0 \).
\end{Definition}

We will always assume that both \( E^{s}_{x} \) and \(  E^{u}_{x} \) are non-trivial (hence one-dimensional) and we stress that our definition of hyperbolicity is inherently \emph{non-uniform} and the set \( \Lambda \) is not in general closed. Moreover, observe that if $\Lambda$ is $(\chi,\eps')$-hyperbolic for some $0<\eps'<\eps$, then it is $(\chi,\eps)$-hyperbolic.

\begin{Remark}
It can be shown that a set \( \Lambda \) with \emph{non-zero Lyapunov exponents} (as in Definition \ref{def:nonzero}) is \emph{hyperbolic} (as in  Definition~\ref{def:hypset}), Indeed, if \( \Lambda \) has non-zero Lyapunov exponents then it is a union of \( f \)-invariant sets on  which  the Lyapunov exponents \( \lambda^{s}, \lambda^{u} \) are uniformly bounded away from~0. Each such set is then \( (\chi,\epsilon) \)-hyperbolic for some \( \chi>0 \) and  for \emph{every} \( \epsilon>0 \), where the functions \( K=K_\epsilon\) and \(C=C_\epsilon \) clearly depend on \( \epsilon \),  see \cite[\S3.3]{BarPes07}.\footnote{The converse is not true; the limits in the definition of non-zero Lyapunov exponents need not exist at every point (only almost every), even in uniform hyperbolicity.  Although existence of these limits is not necessary for our results, the slow variation condition \eqref{eq:slowvar} still plays a crucial role in Theorem \ref{thm:HP}, and it seems unlikely that it can be removed.}
\end{Remark}

A first advantage of formulating hyperbolicity as above is that we can write~\( \Lambda \) as a union of nested sets on which we have uniform estimates. 

\begin{Definition}[Regular sets]\label{def:reg}
Given a $(\chi,\eps)$-hyperbolic set $\Lambda$, for each \( \ell \geq 1 \) we define the \emph{regular level set} \begin{equation}\label{lambdal}
\Lambda_{\chi,\eps,\ell}:=  \{x \in \Lambda: C(x) \leq e^{\epsilon \ell}  \text{ and } K(x) \geq e^{-\epsilon \ell}\}. 
\end{equation}
\end{Definition}

One can assume without loss of generality that the regular level sets $\Lambda_{\chi,\eps,\ell}$ are closed by replacing them with their closures, which still carry the hyperbolic structure in \eqref{eq:angle} and \eqref{eq:hypest} with $C(x) = e^{\eps\ell}$ and $K(x) = e^{-\eps\ell}$, and then taking the union over all $\ell$ to get a $(\chi,\eps)$-hyperbolic set containing $\Lambda$. When the values of $\chi,\eps$ are clear from context, we will often suppress them in the notation and simply write $\Lambda_\ell = \Lambda_{\chi,\eps,\ell}$. Clearly
\( 
\Lambda_\ell \subseteq \Lambda_{\ell+1}  \subseteq \cdots 
\) and \( 
\Lambda = \bigcup_{\ell \geq 1 } \Lambda_\ell
\) 
and, by \eqref{eq:slowvar},  
\begin{equation}\label{eq:images} 
f^{\pm k}(\Lambda_\ell) \subseteq \Lambda_{\ell+k}
\quad\text{for all } \ell,k\in\NN.
\end{equation}

Note that for a $\chi$-hyperbolic set $\Lambda$ the regular sets 
$\Lambda_\ell$ are not defined until we fix a  value of $\eps>0$; changing the value of $\eps$ changes the functions $C,K$ and hence changes the sets $\Lambda_\ell$. On the other hand, we can introduce a useful notation for regular sets independent of an a-priori choice of~\( \Lambda \).

\begin{Definition}[$(\chi,\eps,\ell)$-regular sets]\label{def:chi-reg}
A set $\Gamma\subset M$ is \emph{$(\chi,\eps,\ell)$-regular} if there exists a $(\chi,\eps)$-hyperbolic set $\Lambda$ such that $\Gamma \subset \Lambda_\ell$.  
\end{Definition}

\subsection{Local stable and unstable curves}\label{sec:stablecurves}

One of the fundamental  consequences of hyperbolicity is the existence of stable and unstable curves.

\begin{Definition}[Local stable and unstable  curves]\label{def:localman} 
Given $C,\lambda>0$, we say that a $C^1$ curve $V^s$ in $M$ is a \emph{local $(C,\lambda)$-stable curve} if for every $y,z \in V^s$ and $n\geq 0$, we have 
\begin{equation}\label{eqn:Vs-contract}
d(f^n(y), f^n(z)) \leq C e^{-\lambda n} d(y,z)
\end{equation}
and if in addition the curve can be written as $V^s = \exp_x(\gamma)$ for some $x\in M$ and $\gamma \subset T_x M$ satisfying the following conditions:
\begin{enumerate}
\item there is a splitting $T_x M = E\oplus F$, an interval $B \subset E$, and a $C^1$ function $\psi \colon B\to F$ such that $\gamma := \{v + \psi(v) : v\in B\}$;
\item $\gamma$ lies in the ball around the origin of radius $\inj(M)$, the injectivity radius of the manifold.
\end{enumerate}
Replacing $f^n$ with $f^{-n}$ in \eqref{eqn:Vs-contract} gives the definition of a \emph{local $(C,\lambda)$-unstable curve} $V^u$. We will sometimes omit $C,\lambda$ from the notation when their precise values are unimportant, and refer simply to \emph{local stable and unstable curves}.
\end{Definition}

\begin{Definition}[Brackets]\label{def:brackets}
A point $z\in M$ is said to be the \emph{bracket} of two points $x,y\in M$ if there is a local stable curve $V_x^s$ containing $x$ and a local unstable curve $V_y^u$ containing $y$  and such that $z \in V_x^s \cap V_y^u$, and if moreover $z$ is the \emph{only} point with this property and the intersection is transversal. In this case we write $[x,y] = z$.
\end{Definition}

The above definitions are slightly more general than those commonly used in the literature because we make no a priori assumption that the points involved lie in a $(\chi,\eps)$-hyperbolic set; this makes certain bookkeeping tasks more convenient (see Remark \ref{rem:Gamma-reg}).
With that said, we will still restrict our attention to points satisfying the usual picture provided by 
the following classical result; see \cite[\S7.1]{BarPes07} or \cite{yP76} for a more  precise and technical statement and the proof.

\begin{theorem}[Local Stable and Unstable Manifold Theorem]\label{thm:HP}
For all \(  \chi > \lambda > 0 \)   there exists \(  \epsilon_{0}>0  \) 
such that,  for any \( \epsilon\in (0, \epsilon_{0}) \) and any \( \ell\in \mathbb N \) there exist constants  $C_\ell,\delta_\ell>0$ such that
if $\Lambda$ is a $(\chi,\epsilon)$-hyperbolic set, then
there are families $\{V_x^s \ni x\}_{x\in \Lambda_\ell}$ and $\{V_x^u\ni x\}_{x\in \Lambda_\ell}$ of $C^{1+\text{H\"older}}$ local $(C_\ell,\lambda)$-stable and $(C_\ell,\lambda)$-unstable curves, respectively, such that 
\begin{enumerate}
    \item $V_x^s$ and $V_x^u$ depend continuously on $x\in \Lambda_\ell$ in the $C^1$ topology;
    \item if $x,y \in \Lambda_\ell$ satisfy $d(x,y) < \delta_\ell$, then the bracket $[x,y]$ exists and is the unique point in $V_x^s \cap V_y^u$.
\end{enumerate}
\end{theorem}

\begin{Remark}\label{rmk:HP}
In fact Theorem \ref{thm:HP} follows from Theorem \ref{thm:inf-po} and Corollary \ref{cor:bracket-reg} below, with the small caveat that there we only guarantee that $V_x^{s/u}$ are local $(C_\ell, \lambda/3)$-stable and unstable curves. The proof of Theorem \ref{thm:HP} could be modified to improve $\lambda/3$ to $\lambda$ by being more careful with the bounds in \eqref{eqn:greek0}, \eqref{eqn:eps}, and \eqref{eqn:omega} below, but this is not necessary for our purposes.
\end{Remark}

\subsection{Definition of constants}\label{sec:definition-of-constants}
Before proceeding further, we give an explicit bound on how small $\eps$ must be in the $(\chi,\eps)$-hyperbolic sets we consider.  The precise form of this bound is technical and can be omitted at a first reading; the important thing is that as with $\eps_0$ in Theorem \ref{thm:HP}, the quantity $\eps_1$ depends only on $f$, $\chi$, and $\lambda$. Fix constants $c_1<0 < c_2,c_3$ such that
\begin{equation}\label{eqn:bc}
c_1  = - c_2
\leq \min_{x\in M}\{\log\|D_{x}f^{-1}\|^{-1}\} \leq \max_{x\in M}\{\log \|D_{x}f\|\} \leq c_2 < \frac{c_{3}}{1+\alpha}.
\end{equation}
Given $\chi>\lambda>0$, let $\eps_0$ be as in Theorem \ref{thm:HP}. Define the following auxiliary constants: 
\begin{equation}\label{eqn:greek0}
\begin{aligned}
\gamma := \frac{\chi - c_1}{2\chi}, 
\quad 
\beta &:= \frac{2\chi}{c_3 + \chi} \alpha,
\quad& 
\iota &:= \frac{2(\chi-\lambda)}{6\eps_0\gamma\alpha + (2+\alpha\beta)c_2 + 2\chi},
\\
\eta &:= 6\gamma\alpha\iota + 2,
\quad&
\zeta &:= \alpha\beta\iota.
\end{aligned}
\end{equation}
Notice that $\gamma,\eta>1$ and $\beta,\iota,\zeta\in (0,1)$. Let
$\eps_1 = \eps_1(f,\chi,\lambda)$ be given by
\begin{equation}\label{eqn:eps}
\eps_1 := \min\left\{\frac{\lambda\alpha}{18}, \frac{\lambda\beta}{7\gamma},   \frac{\lambda\zeta}{\eta-1},
\frac{\lambda}{2(1+1/\alpha)}, 
\frac{\lambda^3 \alpha}{9(\lambda - 4c_1)(\lambda+4c_3)},
\epsilon_{0} \right\}.
\end{equation}

\begin{center}
\smallskip
\fbox{\parbox{.9\textwidth}
{\textbf{Throughout the paper, we will consider $(\chi,\eps)$-hyperbolic sets where $\eps \in (0,\eps_1)$ and $\eps_1$ is given by \eqref{eqn:eps}.}}} 
\smallskip
\end{center}

\begin{Remark}
The first bound in \eqref{eqn:eps} is used in \eqref{eqn:omega}, while the next three are  used in the proof of Theorem~\ref{thm:overlapping} in \S\ref{sec:overlapping}, when we produce a constant $\delta>0$ such that for every $\ell\in \NN$ and every $x,y\in \Lambda_\ell$ with $d(x,y) \leq \delta e^{-\lambda\ell}$, the corresponding \emph{Lyapunov charts} (\S\ref{sec:lyapcharts}) are \emph{overlapping} (Definition \ref{def:overlapping}).  More precisely, the second and third  bound in \eqref{eqn:eps} are used in Lemmas \ref{lem:4terms1} and \ref{lem:4terms2}, and the fourth is used in \eqref{eqn:small-in-bell}.
The fifth bound is used in \S\ref{sec:saturated} when we prove Theorem \ref{thm:existsY}.
The last bound, \( \epsilon_{1}\leq \epsilon_{0} \) is used to guarantee that we can apply the Local Stable and Unstable Manifold Theorem \ref{thm:HP}. 
The bounds in \eqref{eqn:eps} also imply that $\eps_1 < \frac \lambda{12}$, which is used in \eqref{eqn:ellj-0}.
\end{Remark}

\begin{Remark}
The precise value of $\lambda \in (0,\chi)$ is not important and the reader wishing to reduce the number of constants may as well consider $\lambda = \chi/2$, although choosing a different value of $\lambda$ might yield a larger value of $\eps_1$.
\end{Remark}

\subsection{Nice domains, recurrent sets, and SRB measures}\label{sec:PgivesSRB}
We are now ready to introduce the key definitions we need to state our main result. To simplify the notation, for a positive integer $T$ we let $T\mathbb N$ denote the set of positive integer multiples of $T$. We also use the notation $V_{p/q}^{s/u}$ to refer simultaneously to  $V_p^s, V_q^s, V_p^u, V_q^u$.

\begin{Definition}[Nice Domain]\label{def:nice}
Given $\chi, \eps>0$, $\ell \in \NN$, and \( r\in (0, \delta_{\ell}) \), a \emph{\((\chi, \eps, \ell, r)\)-nice domain}  is  a  topological disk  \( \Gpq \) satisfying \( \diam (\Gpq)<r \) whose boundary  is formed by (pieces of) the local stable and unstable curves   of \( (\chi, \epsilon, \ell) \)-regular  periodic points $p,q$; see Figure~\ref{fig:nice}. We let \( T=T(\Gpq) \) denote the smallest positive integer such that:
\begin{enumerate}
\item $T$ is a common multiple of the periods of \( p, q \), so that  \( f^{T}(p)=p \) and  \( f^{T}(q)=q \);
\item $T$ is even, so that all eigenvalues of $Df^T_p$ and $Df^t_q$ are positive; and
\item $T$ is large enough that $8\sqrt{2} (1+e^{2(\lambda-\chi)})^{-1/2}e^{2\eps\ell}e^{-\lambda T/3} < 1$.
\end{enumerate}
\end{Definition}

\begin{figure}[tbp]
\includegraphics[width=.4\textwidth]{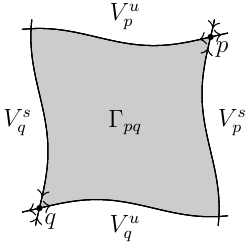}
\caption{A nice domain.}
\label{fig:nice}
\end{figure}

\begin{Remark}
The ``niceness'' condition implies  that  for all $n\in T\NN$, 
\begin{equation}\label{eq:nicerect}
f^{n}(V^{s}_{p/q}\cap \Gpq )\cap \Int \Gpq = \emptyset  \ \text{ and } \ 
f^{-n}(V^{u}_{p/q}\cap \Gpq )\cap \Int \Gpq = \emptyset;
\end{equation}
that is, the stable (respectively, unstable) boundary of \( \Gpq \) never intersects the interior of \( \Gpq \) at iterates which are multiples of \( T \) in forward (respectively, backward) time.
This can be thought of as a two-dimensional version of the notion of ``nice interval'' in one-dimensional dynamics, which refers to an interval whose boundary points never enter  the interval in forward time.\footnote{After this paper was completed we learned of recent work by Chen, Wang, and Zhang that uses a similar notion for systems with singularities; see Definition 9 in \cite[\S5.3]{CWZ}.}
Here \eqref{eq:nicerect} will play a crucial role by ensuring that certain regions are necessarily nested or disjoint, see e.g. Remark \ref{rem:nice} and Lemma~\ref{lem:nestdisj}. There is no obvious generalization of this condition to higher dimensions and this is essentially the main reason for which our main results are restricted to surface diffeomorphisms. 
\end{Remark}

\begin{Remark}
The condition that $T$ be even is used in the proof of Lemma \ref{lem:nestdisj} to guarantee that the `branches' forming our eventual Markov structure are disjoint. The condition that $T$ be sufficiently large is used in \S\ref{sec:fatyoung} to guarantee that the tower we build satisfies the contraction property required in \cite{lY98}.
\end{Remark}

\begin{Definition}[Nice regular set]\label{def:nice-set}
Given $\chi,\eps>0$,  $\ell\in\NN$, and \( r\in (0, \delta_\ell] \), a \emph{\((\chi, \eps, \ell, r)\)-nice regular set} is a set $A$ that is \( (\chi, \epsilon, \ell) \)-regular and is contained in some \((\chi, \eps, \ell, r)\)-nice domain \( \Gpq   \). 
\end{Definition}

The simplest way to produce a $(\chi,\eps,\ell,r)$-nice regular set is to fix a $(\chi,\eps,\ell,r)$-nice domain $\Gpq$ and then put $A = \Lambda_{\chi,\eps,\ell} \cap \Gpq$, but we do not require that the set $A$ has this exact form.

\begin{Remark}\label{rem:implicit}
The definitions above depend on constants \( \chi, \epsilon, \ell, r \) which, for simplicity,  are not always reflected in the notation; we will sometimes refer to a `nice domain' or a `nice regular set' without naming the constants. We stress however that any reference to nice domains or nice regular sets which does not explicitly refer to a choice of constants means that such a choice is implicitly clear from the context. The same applies to the constant \( T \) which is always associated to a nice domain and therefore to a nice regular set (as for example in the following definition in which \( T \) is implicitly given by the choice of a nice regular set~\( A \)). 
\end{Remark}

\begin{Definition}[Recurrence]\label{def:Treturn}
A nice regular set   \(  A\subseteq \Lambda  \) is
\begin{enumerate}
\item  \emph{recurrent} if   for all \(   x\in A \) there exist  \( i, j \in T\mathbb N  \) such that   
\[
  f^{i}(x), f^{-j}(x)\in A;
    \]
\item \emph{Lebesgue-strongly recurrent}
if it is recurrent and there is a local unstable curve $V^u$ and a set $E \subset V^u \cap A$ of positive 1-dimensional Lebesgue measure\footnote{One could imagine studying other equilibrium measures besides SRB by replacing Lebesgue measure here with a reference measure such as those studied in \cite{CPZ19,CPZ20}, but we do not pursue this here.} such that for every $x\in E$ we have
\begin{equation}\label{eqn:pos-freq}
\limsup_{n\to \infty} \frac 1{n} {\#\{i\in T\NN, 1\leq i \leq n: f^{i}(x) \in A\}}>0.
\end{equation}
\end{enumerate}
\end{Definition}

Observe that a Lebesgue-strongly recurrent set is fat in the sense of \eqref{eq:fat}.
We are now ready to state our main result. 

\begin{maintheorem}\label{thm:existsSRB}
Let \(  f\) be a \(  C^{1+\alpha}  \) surface diffeomorphism. 
\begin{enumerate}[label=\upshape{(\arabic{*})}]
\item\label{A1} 
For every $\chi>\lambda>0$,  
$0<\eps<\eps_1(f,\chi,\lambda)$, and  $\ell\in\NN$, there exists $r>0$ 
such that if there exists a \((\chi, \eps, \ell, r)\)-nice 
regular set $A$ that is Lebesgue-strongly recurrent,
then \( f \) admits an SRB measure. 
\item\label{A2} Conversely, if \(  f  \) admits an SRB measure,
then for every sufficiently small $\chi>0$ and every $\eps>0$, there exists $ \ell\in \NN$ such that for every $r>0$, there exists an \((\chi, \eps, \ell, r)\)-nice  regular set $A$ that is Lebesgue-strongly recurrent.
\end{enumerate}
\end{maintheorem}
In addition to the explicit bounds in \eqref{eqn:eps} on $\eps$, we refer to
equation \eqref{def:r} for relatively explicit conditions on~\(  r  \). 

\begin{Remark}
Theorem \ref{thm:existsSRB} provides (in the two-dimensional case) natural geometric conditions which are necessary and sufficient for the existence of an SRB measure, which is a highly non-trivial feature of a dynamical system. Part \ref{A1} of the theorem can be interpreted as a version of Viana's conjecture.
The condition of Lebesgue-strongly recurrent includes a `fatness' requirement, which is definitely necessary as it is built into the definition of SRB measures (see also Remark \ref{rmk:fat-subtlety}). Beyond this, the theorem requires `niceness' and `(strong) recurrence'.
While our proof uses the niceness property in an essential way, we believe that 
with some new ideas it should be possible to remove this technical requirement.
In addition, part \ref{A2} of the theorem indicates that this requirement is not too strong. 
 This leaves us with the requirement of strong recurrence which is not explicit in the statement of Viana's conjecture but we believe is needed to establish existence of an SRB measure. Here again part \ref{A2} of the theorem indicates that strong recurrence is not too much to ask for.

We stress that the proof of Theorem \ref{thm:existsSRB} includes a number of new results in non-uniform hyperbolicity theory (such as results about almost returns and  shadowing) which are of independent interest. 
We will discuss these in more details in the next sections. 
\end{Remark}

\subsection{Historical background}
\label{sec:histsrb}

We review here some of the main results on the existence of SRB measures for diffeomorphisms. To avoid getting too technical we will not be overly specific about the precise technical assumptions, emphasizing instead the general ideas. We refer the reader to \cite{CliLuzPes16} for more details and a discussion of the various techniques which have been used in different settings. Most  results mentioned below hold  in arbitrary dimension. 

\subsubsection{Uniformly hyperbolic sets}

In the 1970's, Sinai, Ruelle, and Bowen constructed (in fact \emph{invented!}) SRB measures for \emph{fat, uniformly hyperbolic} sets \( \Lambda \) (attractors) under the additional assumption that \( \Lambda \) has a dense set of periodic points  (\emph{Axiom A}) or, equivalently, that \( \Lambda \) has local product structure or that \( \Lambda \) is locally maximal (see \cite{Shu78} for a proof that these three properties of uniformly hyperbolic sets are equivalent).

\subsubsection{Partially hyperbolic attractors} 

In 1982, Pesin and Sinai \cite{PS} developed a new ``push-forward'' technique for constructing what they called ``u-Gibbs'' measures (also called simply $u$-measures), which share a lot of geometric characteristics of SRB measures but are not necessarily physical measures. They applied their construction to \emph{partially hyperbolic} attractors. In 2008, Burns, Dolgopyat, Pesin, and Pollicott \cite{BDPP} showed that under some ``transitivity'' assumptions, a $u$-measure that has negative Lyapunov exponents for vectors in the central direction on a set of positive measure is a unique SRB measure. In 2000, Bonatti and Viana \cite{BV} considered a variation of this setting with a continuous splitting \( E^{s}\oplus E^{uu} \) with uniform expansion estimates in \( E^{uu} \) and nonuniform contraction estimates in \( E^{s} \), and proved the existence of genuine SRB measures.

Around the same time, Alves, Bonatti, and Viana \cite{ABV} considered the more difficult setting of a continuous splitting \( E^{ss}\oplus E^{u} \) with uniform contraction and nonuniform expansion estimates. The construction of the SRB measure in this case required a significantly more sophisticated version of the push-forward technique. An alternative construction of the SRB measure using Young towers was carried out more recently in \cite{AlvDiaLuzPin17} under some slightly weaker expansivity assumptions. 

\subsubsection{Nonuniformly hyperbolic sets} 
Relaxing the continuity of the splitting \( E^{s}\oplus E^{u} \) and the uniform lower bound on the angle \(  \measuredangle (E^s_x, E^u_x) \) seems to bring the level of difficulty of the problem to another level. The first result for a system with this kind of hyperbolic set is due to Benedicks and Young  who constructed SRB measures for certain two-dimensional ``H\'enon'' maps, first using the push-forward technique \cite{BenYou93} and then using Young towers \cite{lY98,BenYou00}, in both cases taking significant advantage of the specific geometric and analytic properties of the maps. These results were extended to more general ``H\'enon-like'' maps in \cite{YouWan} but still only apply to some quite restrictive classes of systems. 

More recently,  \cite{CliDolPes16} significantly generalizes the techniques of \cite{ABV} to successfully construct SRB measures for systems in which the splitting \(E^{s}\oplus E^{u}\) is only measurable, with angle \(\measuredangle (E^s_x, E^u_x)\) not bounded away from zero, and with non-uniform contraction and expansion estimates which, however, need to satisfy a non-trivial ``synchronization'' assumption.

\subsubsection{Necessary and sufficient conditions}

Theorem \ref{thm:existsSRB} is in some sense an optimal result, because the conditions stated there are both necessary and sufficient for the existence of an SRB measure.  In contrast, the prior results mentioned above all established sufficient conditions for existence of SRB measures, without addressing the question of whether the conditions are also necessary.\footnote{As we will see in the next section, for surfaces Young's tower conditions from \cite{lY98} turn out to be necessary as well as sufficient, but this was not proved in that paper.}  

We mention three notable exceptions, where some set of conditions was demonstrated to be both necessary and sufficient for existence of an SRB measure. First is work of Tsujii \cite{Tsu91}, in which a point $x$ is called `regular' if its empirical measures $\mu_n = \frac 1n \sum_{k=0}^{n-1} \delta_{f^k x}$ weak*-converge to an ergodic measure $\mu$ whose Lyapunov exponents agree with the exponents of $x$; then existence of an SRB measure is shown to be equivalent to existence of a positive volume set of regular points with nonzero Lyapunov exponents. More recently, Burguet \cite{dB} has proved a similar result, using a different definition of `regular' that is related to the `tempered' property, and that does not require one to determine in advance whether empirical measures from a point converge to a given measure.

Finally, we mention a preprint of Snir Ben Ovadia \cite{BO20} that appeared after our results were completed. This work is most closely related to ours as it establishes existence of an SRB measure under conditions of hyperbolicity, fatness and strong recurrence which are very similar to ours. Moreover, Ben Ovadia's proof works in any dimension, not just for surfaces, and does not require the use of nice domains. However, while our method is pretty much geometrical, his approach is completely different and is based on Markov coding which he developed in his earlier work on countable Markov partitions \cite{BO19}, building on work of Sarig \cite{oS13}. This Markov coding is in general, non-uniformly bounded-to-one, and thus the tower which can be obtained by inducing on a single state need not be a first $T$-return Young tower for any $T$, in contrast to the one we build; see \S\ref{sec:sat-disc} for a discussion of how the first return property is related to the \emph{saturation} property of the rectangle that we build as the base of the Young tower.
Moreover,
as we will see in Corollary \ref{thm:edc}, our construction permits us to formulate a condition under which the tower has exponential tails, which appears to have no analogue in the abstract Markov partition approach.
We believe that this opens up a possibility to study various properties of SRB measures such as the decay of correlations and the Central Limit Theorem using recently available techniques involving Young towers.

\section{Young towers: Theorem \ref{thm:existsY}}\label{sec:FRTYT}

Our strategy for producing the SRB measure in Theorem \ref{thm:existsSRB} consists of a new technique for the  construction of a \emph{Young tower}.  
The latter has a non-trivial geometric and dynamical structure  and it is rather remarkable that its existence can be deduced using only 
a nice regular set that is Lebesgue-strongly recurrent. 

\subsection{Rectangles}

Recall that Theorem \ref{thm:HP} guarantees that given 
$\chi>\lambda>0$ and a $(\chi,\eps)$-hyperbolic set $\Lambda$, any two sufficiently close points $x,y\in \Lambda_\ell$ have a unique bracket $[x,y] = V_x^s \cap V_y^u$. This point need not be contained in $\Lambda_\ell$, or even in $\Lambda$; for example, this occurs when $\Lambda$ consists of two hyperbolic periodic orbits that pass sufficiently close to each other.

Nevertheless, we will see in Theorem \ref{thm:inf-po} and Corollary \ref{cor:bracket-reg} below that $z= [x,y] $ always belongs to \emph{some} hyperbolic set, and indeed more precisely to a $(\lambda/4,2\eps, \ell+\ell')$-regular set, where $\ell'$ depends only on $\chi,\lambda,\eps$. One could then produce local stable and unstable curves $V_z^s$ and $V_z^u$ by applying Theorem \ref{thm:HP}; however, \emph{a priori} these may be shorter than $V_x^s$ and $V_y^u$.
We will generally write $V_z^s = V_x^s$ and $V_z^u = V_y^u$ for the longer curves, even though these may not be ``centered'' at $z$.
Indeed, we can write $V_z^s = V_x^s$ for any $z\in V_x^s$, and similarly for unstables.

Brackets play a very important role in describing the geometry of hyperbolic sets. In particular we will be interested in sets which are \emph{closed} with respect to the bracket operation.

\begin{Definition}[Rectangles]\label{def:rectangle}
Given $C,\lambda>0$, a \emph{$(C,\lambda)$-rectangle} is a set $\Gamma$ such that for every $x,y\in \Gamma$ there is a $(C,\lambda)$-local stable curve $V_x^s$ containing $x$ and a $(C,\lambda)$-local unstable curve $V_y^u$ containing $y$ with the property that $[x,y] := V_x^s \cap V_y^u$ is a single point, and $[x,y] \in \Gamma$. When the precise values of $C,\lambda$ are unimportant, we will refer to $\Gamma$ simply as a \emph{rectangle.}
\end{Definition}

\begin{Remark}
Rectangles are also sometimes referred to in the literature as sets with \emph{local product structure} or \emph{hyperbolic product structure}, though this is usually in the  more restrictive uniformly hyperbolic setting. 
The discussion above shows that rectangles are very natural structures also in our more general (nonuniformly) hyperbolic setting. 
Indeed, if \( A \) is a \( (\chi, \epsilon, \ell) \)-regular set of sufficiently small diameter then the bracket \( [x,y] \) is well defined and consists of a single point for every pair of points \( x,y\in A \).  In this case then $[[x,y],[x',y']] = [x,y'] \in \Gamma$ for all $x,y,x',y'\in A$ and therefore the set 
\begin{equation}\label{eq:bracketclosure}
\Gamma := \{[x,y] : x,y\in A \}
\end{equation}
 is closed under the bracket operation, so $\Gamma$ is a $(C_\ell,\lambda)$-rectangle.
\end{Remark}

\subsection{Topological Young Towers}

To describe the notion of Young tower that we use,
we first recall some standard and some slightly non-standard definitions.

\begin{Definition}[Nice Rectangles]\label{def:nicerectangle}
A rectangle $\Gamma$ is a \emph{nice rectangle} if $\Gamma \subset \Gpq$ for some nice domain $\Gpq$ and if the local stable and unstable curves of every point $x\in \Gamma$ are ``full length'' in $\Gpq$ in the following sense: $V_x^u$ intersects both $V_p^s$ and $V_q^s$, and $V_x^s$ intersects both $V_p^u$ and $V_q^u$.
\end{Definition}

A nice rectangle is not assumed a priori to be a nice regular set (although it is assumed to have stable and unstable curves through every point); the word `nice' in both cases emphasizes the geometry of the situation.
We  want to restrict our attention to the pieces of local stable and unstable curves which are inside the nice domain and so, for any $x\in \Gamma \cup \{p,q\}$, let 
\[
W_x^{s/u} := V_x^{s/u} \cap \Gpq. 
\]
Then a nice  rectangle $\Gamma\subset \Gpq$ has the structure 
\begin{equation}\label{eq:rectanglestruct}
\Gamma = C^{s} \cap C^{u} 
\quad\text{ where } \quad 
C^{s}:= \bigcup_{x\in \Gamma} W^{s}_{x} \quad \text{ and } \quad 
C^{u}:= \bigcup_{x\in \Gamma} W^{u}_{x}. 
\end{equation}

\begin{Definition}[\( s \)-subsets and \( u \)-subsets]\label{def:su-sets}
If  \( \Gamma \) is a rectangle, we say that  
$\Gamma^s \subset\Gamma$ is an \(s\)\emph{-subset} of \( \Gamma \) if $x\in\Gamma^s$ implies $V^s_{x}\cap\Gamma\subset\Gamma^s$ and 
$\Gamma^u \subset\Gamma$ is a  \(u\)\emph{-subset} of \( \Gamma \) if $x\in\Gamma^u$ implies $V^u_{x}\cap\Gamma\subset\Gamma^u$. 
\end{Definition}

\begin{Definition}[T-return times]\label{def:firstTret}
For a  $T$-recurrent set \( A \subset \Gpq \) and  \( x \in A \), let
\[
\tau(x) := \min\{i\in T\NN: f^{i} (x) \in A\}
\]
be the  \emph{first T-return time} to \( A \).  
\end{Definition}

\begin{Definition}[Topological   Young Tower]\label{def:FRYT}
A   nice $T$-recurrent  rectangle \( \Gamma \) supports a \emph{First T-Return Topological  Young Tower} if
for each \( i\in T\NN \) 
we can subdivide 
\(
\Gamma^s_i :=\{x\in\Gamma: \tau(x)=i\} 
\)
into pairwise disjoint \(  s  \)-subsets
\begin{equation}\label{eq:youngtower}\tag*{\textbf{(Y0)}}
\Gamma^{s}_{i1},\dots,\Gamma^{s}_{im_i} 
\quad \text{ such that } \quad    \Gamma^{u}_{ij}:= f^{i}(\Gamma^{s}_{ij}) 
\end{equation}
 is a \(  u  \)-subset of \(  \Gamma  \)    for every \(  j=1,\dots,m_i  \).
 If $\Gamma_i^s$ is empty we put $m_i = 0$.
\end{Definition}

This definition in particular requires that $\Gamma_i^s$ is an $s$-subset, so that for each $x\in \Gamma$ with $\tau(x) = i$ we have $f^i(W_x^s \cap \Gamma) \subset W_{f^i x}^s \cap \Gamma$; this Markov condition, and the analogous one for $u$-subsets, is often formulated explicitly as part of the definition of a Young tower.

\begin{Remark}
We remark that the ``First T-Return'' part of the definition comes from the fact that \( \tau(x) \) is the first T-return time.  A similar definition can be used for 
a general return time function that is constant on local stable leaves to give a more general Topological Young Tower. 
\end{Remark}

 \begin{Remark}
As mentioned in Remark \ref{rem:implicit} above, several constants, including \( T \), are implicit in these definitions. 
We will sometimes include $T$ explicitly in the associated terminology, for example in the notions of ``T-return times'' and  of  ``First T-Return'' in Definitions \ref{def:firstTret} and \ref{def:FRYT} above, when it helps maintain clarity. 
 \end{Remark}
 
\subsection{Young Towers}
We call the above a \emph{Topological} Young Tower because \ref{eq:youngtower} captures only the topological structure of a Young tower which is often, including in the present setting,  the most difficult part of the construction. To state the other properties of a Young tower we need to introduce the induced map to~\(  \Gamma  \).

\begin{Definition}[Induced map]
For a rectangle \( \Gamma \) which supports a First T-Return Topological Young Tower, 
we define the \emph{induced map} 
 \[
 F\colon \Gamma \to \Gamma \quad \text{ by } \quad 
F|_{\Gamma^{s}_{i}}:= f^{i}
\]
and refer to \(  \Gamma  \) as the \emph{base} of the tower. 
We also let 
\(
\text{Jac}^uF(x):=|\det Df^{\tau(x)}|_{E^u_{x}}|
\) 
denote the  unstable Jacobian of this induced map,
where $E_x^u = T_x W_x^u$.
\end{Definition}

\begin{Definition}[Young Tower]\label{def:young-tower}
Let \( \Gamma \) be a rectangle which  supports a First T-Return Topological  Young Tower. We say that \(  \Gamma  \) supports a \emph{First T-Return   Young Tower} if there exist constants
\( \kappa_{1}, \kappa_{2}\in (0,1) \) and \( c>0 \) such that \begin{enumerate}[label=\upshape{\textbf{(Y\arabic{*})}}]\setcounter{enumi}{0}
\item\label{Y1}  for every \(  i \in T\NN \),  \( j\in \{1,...,m_i\} \),  \(  x\in \Gamma^{s}_{ij}  \): 
\begin{enumerate}[ref=\upshape{\textbf{(Y1)}(\alph{*})}]
\item\label{Y2a} 
$
d(F(x), F(y))\leq\kappa_{1}  d(x,y)
$  for  every \(  y\in V^{s}_{x}  \)
\item\label{Y2b} 
$
d(x,y)\leq\kappa_{1} d(F(x), F(y))
$
for every $y\in V^u_{x} \cap \Gamma^{s}_{ij}$.
\end{enumerate}

\item\label{Y2}
for every $x\in  \Gamma$ and $n\ge 0$: 
\begin{enumerate}
[ref=\upshape{\textbf{(Y2)}(\alph{*})}]
\item\label{Y3a} For all  $y\in V^s_{x}$, we have
\[
\left|\log\frac{\text{Jac}^uF(F^{n}(x))}{\text{Jac}^uF(F^{n}(y))}\right|\le c\kappa_{2}^n;
\]
\item\label{Y3b} 
If $y\in V_x^u$ has the property that for all $0\leq k\leq n$ there are $i_k \in T\NN$ and $j_k\in \{1,\dots, m_{i_k}\}$ such that $F^k(x), F^k(y) \in \Gamma^s_{i_k j_k}$, then we have
\[
\left|\log
\frac{\text{Jac}^u F(F^{n-k}(x))}{\text{Jac}^uF(F^{n-k}(y))}\right|\le c\kappa_{2}^k
\quad\text{for all }0\leq k\leq n.
\]
\end{enumerate}
\end{enumerate}
\end{Definition}

\begin{Remark}
The additional geometric estimates required for a Topological Young Tower to be a Young Tower are non-trivial and are not consequences of the topological structure. Indeed, it is easy to construct a topological Young Tower which has a  \emph{neutral fixed point}, i.e., a fixed point at which the derivative has both eigenvalues equal to 1. An easy and intuitive model of this, albeit in the one-dimensional non-invertible setting, is the map \( f(x) = x+ x^{2} \) mod 1 which is topologically a full branch map with two branches but does not satisfy uniform expansion or bounded distortion. Similarly, a Topological Young Tower with a neutral fixed point cannot satisfy conditions \ref{Y1} and \ref{Y2}. 
\end{Remark}

\begin{Definition}[Integrable return times]\label{def:integrability}
A \emph{fat} rectangle \( \Gamma \) which supports a First T-Return  Young Tower has \emph{integrable return times} if there is  \(  x\in \Gamma  \) such that
\[
\int_{V^{u}_{x}\cap \Gamma} \tau\, d m_{V_x^u}< \infty.
\] 
\end{Definition}

\subsection{Existence of Young Towers}\label{sec:P'givestower}

Young towers are non-trivial geometric structures and their  construction generally requires substantial work. A key part of our argument is to show that they are  part of the intrinsic structure of hyperbolic sets.

\begin{maintheorem}\label{thm:existsY}
Let \(  f\colon M \to M  \) be a \(  C^{1+\alpha}  \) surface diffeomorphism. 
For every $\chi>\lambda>0$, every $0<\eps<\eps_1(f,\chi,\lambda)$, and every
$\ell\in \NN$ there exists $r>0$ such that  if \( A \) is a \((\chi, \eps, \ell, r)\)-nice regular set that is recurrent, then
\begin{enumerate}[label=\upshape{(\arabic{*})}]
\item\label{B1}  \( A \) is contained in a nice rectangle $\Gamma\subset\Gpq$ that is recurrent and supports a First $T$-Return Topological Young Tower;
\item\label{B2}  if   the set $\Gamma$ of part \ref{B1} is  Lebesgue-strongly recurrent, then $\Gamma$ supports a First $T$-Return  Young Tower with integrable return times.
\end{enumerate}
\end{maintheorem}

\begin{Remark}\label{rmk:fat-subtlety}
Since the rectangle $\Gamma$ provided by Theorem \ref{thm:existsY}\ref{B1} is recurrent and contains $A$, it follows from Definition \ref{def:Treturn} that if $A$ is Lebesgue-strongly recurrent, then $\Gamma$ is Lebesgue-strongly recurrent as well.
However, it is possible that the set $A$ is not fat (and hence not Lebesgue-strongly recurrent) but still produces a Lebesgue-strongly recurrent (hence fat) rectangle $\Gamma$.
In particular the assumptions in part \ref{A1} of Theorem \ref{thm:existsSRB} can  be relaxed 
to require that the construction which we will use in the proof of  part \ref{B1} of Theorem \ref{thm:existsY} produces a Lebesgue-strongly recurrent rectangle. A simple setting in which this distinction may be seen to  be potentially relevant is that of a two-dimensional Anosov diffeomorphism. Then we can easily find a small nice domain \( \Gpq \) and may choose \( A  \) as a countable dense set of points in \( \Gpq \) which means in particular that \( A \) is not fat, but  we shall see from the proof of Theorem ~\ref{thm:existsY} that this gives rise to a  Lebesgue-strongly recurrent rectangle \( \Gamma \). 
\end{Remark}

\begin{Remark}\label{rmk:weaker}
We will see in the proof (see Proposition \ref{thm:nicehyp}) that the rectangle $\Gamma$ produced by Theorem \ref{thm:existsY}\ref{B1} is $(\lambda/4,2\eps,\ell+\ell')$-regular for a value of $\ell'$ that depends only on $\chi,\lambda,\eps$.
\end{Remark}

\subsection{Historical background}\label{sec:tower-history}

The Young tower approach for constructing invariant measures is a particular case of a general and classical method in ergodic theory, that of \emph{inducing}. This is  based on the construction of a return map to some subset of the phase space which is simpler to study and more amenable to the construction of an invariant measure;  this invariant measure can then be extended  to the whole phase space by an elementary argument.  

The specific inducing structure defined above,  which is pertinent to the study of systems with hyperbolic behavior, was introduced by Young \cite{lY98} as a framework for studying the existence and ergodic properties of SRB measures.\footnote{A more general inducing structure was introduced in \cite{PSZ16}; it can be used to study the existence and ergodic properties of equilibrium measures, which include SRB measures.}
Since then, it has been applied to a variety of cases, including billiards \cite{BalTot08, Che99, CheZha05, Mar04,  lY98}, certain H\'enon maps \cite{BenYou00}, and partially hyperbolic diffeomorphisms \cite{AlvDiaLuzPin17, AL15,  AlvPin08, AlvPin10}. A non-invertible version of a Young tower was also introduced by Young in \cite{lY99} and has proved extremely powerful in studying non-invertible maps satisfying  nonuniform expansivity conditions \cite{AlvFreLuzVai11, AlvLuzPin05, BruLuzStr03,  CP17, DiaHolLuz06,   Gou06, Hol05, Pin11}.

A quite  remarkable feature of the  method of Young towers is that it associates, by construction, a non-trivial geometric structure  to the measure; a structure which moreover  can be used  effectively to study several statistical and ergodic properties of the measure, see \cite{AlvCarFre10b, AlvCarFre10,  Chu11,  Dem10, GupHolNic11, HayPsi14,  HolNicTor12, vMD01, MN05, MN08, ReyYou08, lY98, lY99} as well as the other references mentioned above.  This leads naturally to the question of the domain of applicability of this method: if it implies so much structure then maybe it can only be applied to a limited number of special cases which have this structure? This   legitimate doubt  is partly supported by the observation that all the constructions of Young towers so far have relied heavily on specific, assumed a-priori,  geometric properties of the systems under consideration. This question has therefore led to the so-called \emph{liftability} problem, which is the question of which measures have, or ``lift to'', a Young tower structure, see Definition \ref{def:liftability} below,
and therefore can in principle be obtained by the method of Young towers.  
 
 For the non-invertible non-uniformly expanding case this has been addressed in several papers and in particular it is shown in \cite{AlvDiaLuz13} that essentially every invariant probability measure which is absolutely continuous  with respect to Lebesgue lifts to a Young tower. 
One consequence of our results, stated formally as Corollary \ref{thm:lifts} below, 
is that for surface diffeomorphisms, every hyperbolic measure (in particular, every SRB measure) lifts to a Young tower, which means that the geometric structure of Young towers is intrinsic to all hyperbolic (and in particular, all SRB) measures. Moreover, we show that every hyperbolic measure lifts to a \emph{first return} Young tower for some iterate of the map, which is perhaps surprising because in most other  applications of Young's work the towers do not have the first return property. Using this, we obtain a way  of controlling the tail of the tower, see Corollary \ref{thm:edc}.  

A related result in this direction is Sarig's result \cite{oS13} that every hyperbolic measure for a surface diffeomorphism can be lifted to a countable-state topological Markov shift, which is closely related to a Young tower.  
The results of \cite{oS13} have been extended to higher dimensions \cite{BO19}, to three-dimensional flows \cite{LS19} and to non-uniformly hyperbolic surface maps with discontinuities \cite{LM18}. It seems natural to expect that our results admit similar extensions, 
although for the extension to higher dimensions one must confront a major technical obstruction, which is that the nice domains we use are inherently two-dimensional.

\section{Lifting hyperbolic measures to towers: Theorem \ref{thm:existsnice1}}
\label{sec:hypmeas}

Statement \ref{A1} of Theorem \ref{thm:existsSRB} follows immediately from Statement \ref{B2} of Theorem \ref{thm:existsY} and the results in \cite{lY98} on the existence of SRB measures for Young towers. Statement \ref{A2} of Theorem \ref{thm:existsSRB} follows immediately from our next result, which states that 
nice Lebesgue-strongly recurrent  rectangles are \emph{necessary} for the existence of  hyperbolic (SRB) measures. 

\begin{maintheorem}\label{thm:existsnice1}
Let \(  f \) be a \(  C^{1+\alpha}  \) surface diffeomorphism. If $\mu$ is a non-atomic ergodic hyperbolic  measure for $f$,  then for every sufficiently small $\chi>0$ and every $\eps>0$, there is an integer $\ell\in \NN$ such that for every $r>0$, 
there exists a \((\chi, \eps, \ell, r)\)-nice  regular set \( A \) with $\mu(A)>0$ such that every $x\in A$ has positive frequencies of returns as in \eqref{eqn:pos-freq}, and continues to satisfy this property with $f^i$ replaced by $f^{-i}$. In particular, $A$ is recurrent, and if $\mu$ is an SRB measure, then $A$ is Lebesgue-strongly recurrent.
\end{maintheorem}

\begin{Remark}\label{rmk:large-ell}
As will be seen in Proposition \ref{prop:get-pq}, the condition on $\ell$ in Theorem \ref{thm:existsnice1} is merely that it be sufficiently large that $\mu(\Lambda_{\chi,\eps,\ell}) > 0$ for some $(\chi,\eps)$-hyperbolic set $\Lambda$.
\end{Remark}

The rest of the paper is therefore devoted to the proofs of Theorems \ref{thm:existsY} and \ref{thm:existsnice1}, which we will carry out through the development of some non-trivial technical results of independent interest. 
We first conclude this section with  two almost immediate corollaries of Theorems  \ref{thm:existsY} and  \ref{thm:existsnice1}, also  of independent interest. The first is  an observation concerning the  geometric structure of hyperbolic measures, formalized in the notion of \emph{liftability} which is relevant in many applications and studies of the ergodic properties of invariant measures, see e.g. \cite{PSZ08}. 

\begin{Definition}[Liftable measures]\label{def:liftability}
An invariant probability measure \( \mu \) is \emph{liftable}  (to a Topological Young Tower) if there exists a recurrent rectangle \( \Gamma \) which supports a Topological Young Tower and a
probability measure  $\hat\mu$ on \( \Gamma \) which is invariant for the corresponding induced 
map \( F: \Gamma\to \Gamma \), such that 
\begin{equation}\label{eq:lifts}
E_{\hat\mu}:= \int_{\Gamma} \tau\, d\hat\mu < \infty 
 \quad \text{ and }\quad 
 \mu=\frac{1}{E_{\hat \mu}} \sum_{i=0}^{\infty} f^{i}_{*}(\hat\mu | \{\tau>i\}).
\end{equation}
\end{Definition}

The following is an immediate consequence of Theorem \ref{thm:existsnice1} and \cite{Zweim}.

\begin{maincorollary}\label{thm:lifts}
Let $f $ be a $C^{1+\alpha}$ surface diffeomorphism. Then every ergodic non-atomic hyperbolic (invariant) probability measure  lifts to a  First T-Return Topological Young Tower for some \(  T>0  \). \end{maincorollary}

Before stating our second corollary, we emphasize that we have a first T-return tower which is, to all intents and purposes, essentially as good as if it was a first return time tower. Existence of a first T-return tower in this generality is a rather surprising result, and has important consequences.  For example, because every return to 
$\Gpq \cap \Lambda_\ell$ is a return to the base of the tower, we can relate the size of the tail of the tower to the return rate to regular level sets.

\begin{maincorollary}\label{thm:edc}
Let $f$ be a $C^{1+\alpha}$ surface diffeomorphism, $\mu$ an ergodic non-atomic $\chi$-hyperbolic  measure, and $\Lambda$ a $\chi$-hyperbolic set with $\mu(\Lambda)>0$.  Suppose that for every $\eps>0$ there is $\ell\in \NN$ such that for every open set $U\subset M$ with $\mu(U\cap \Lambda_\ell)>0$ and every $\ell', T\in \NN$, we have
\begin{equation}\label{eqn:exp-return}
\varlimsup_{n\to\infty} \frac 1n \log \mu \big\{ x\in U\cap \Lambda_{\ell'} : f^{kT}(x) \notin U\cap \Lambda_\ell \text{ for every } 1\leq k \leq n \big\} < 0.
\end{equation}
Then $\mu$ lifts to a First $T$-Return Topological Young Tower for some $T>0$, whose tail decays exponentially in the sense that
\begin{equation}\label{eqn:exp-tails}
\varlimsup_{n\to\infty} \frac 1n \log \hat\mu\{ x\in \Gamma : \tau(x) > n \} < 0.
\end{equation}
\end{maincorollary}

\begin{proof}
Without loss of generality, assume that $\chi>0$ is sufficiently small that Theorem \ref{thm:existsnice1} applies. Fix $\lambda \in (0,\chi)$ and let $\eps_1$ be given by \eqref{eqn:eps}. Fix $\eps \in (0,\eps_1)$, and let $\ell\in \NN$ be as in the hypotheses of Corollary \ref{thm:edc}, so that \eqref{eqn:exp-return} holds for all open $U\subset M$ with $\mu(U\cap \Lambda_\ell)>0$ and every $\ell',T\in \NN$. Note that $\ell$ can also be chosen large enough to guarantee that $\mu(\Lambda_\ell)>0$.

With this value of $\ell$, let $r>0$ be given by Theorem \ref{thm:existsY}, and let $U'\subset M$ be an open set with diameter less than $r$ such that $\mu(U'\cap \Lambda_\ell)>0$. Note that by Remark \ref{rmk:large-ell} this value of $\ell$ works for the conclusion of Theorem \ref{thm:existsnice1} (and in particular Proposition \ref{prop:get-pq}), which guarantees the existence of a nice domain $\Gpq$ such that $\mu(\Int \Gpq \cap \Lambda_\ell)>0$. Let $U = \Int(\Gpq)$ and let 
\[
Z^+ = \{x\in U\cap\Lambda_\ell : \#\{k\in\NN : f^{kT}(x) \in U\cap \Lambda_\ell\} < \infty \}.
\]
Define $Z^-$ similarly with $f^{-kT}$ in place of $f^{kT}$, and let $Z = Z^- \cup Z^+$. Note that $\mu(Z)=0$ by the Poincar\'e recurrence theorem.
Let $A = (U\cap \Lambda_\ell) \setminus Z$, and observe that
$A$ is a $(\chi,\eps,\ell,r)$-nice regular set that is recurrent, so by Theorem \ref{thm:existsY} it is contained in a nice rectangle $\Gamma\subset \Gpq$ that is recurrent and supports a First $T$-Return Topological Young Tower.

Just as in Corollary \ref{thm:lifts}, the measure $\mu$ lifts to this tower by \cite{Zweim}. It remains to prove \eqref{eqn:exp-tails}. 
Observe that
\begin{equation}\label{eqn:inclusions}
\begin{aligned}
\{x\in \Gamma : \tau(x) > n\}
&\subset
\{x\in \Gamma : f^{kT}(x) \notin \Gamma \text{ for every } 1\leq k\leq n \} \\
&\subset \{x\in \Gamma : f^{kT}(x) \notin A \text{ for every } 1\leq k \leq n\}.
\end{aligned}
\end{equation}
where the first inclusion uses the first $T$-return property, and the second uses the fact that $A\subset \Gamma$. If $f^{kT}(x) \notin A$, then we either have $f^{kT}(x) \notin U\cap \Lambda_\ell$ or $f^{kT}(x) \in Z$, and thus the sets in \eqref{eqn:inclusions} are contained in
\[
\{x\in \Gamma : f^{kT}(x) \notin U\cap \Lambda_\ell\text{ for every } 1\leq k\leq n\} \cup \bigcup_{k=1}^n f^{-kT}(Z).
\]
Now we can use the fact that $\mu(Z)=0$ and that $\Gamma \subset U \cap \Lambda_{\ell'}$ for some $\ell'\in\NN$ (recall Remark \ref{rmk:weaker}) to deduce that
\[
\mu\{x\in \Gamma : \tau(x) > n \}
\leq \mu \{x\in U\cap \Lambda_{\ell'} : f^{kT}(x) \notin U\cap \Lambda_\ell \text{ for every } 1\leq k\leq n\}.
\]
From \eqref{eq:lifts} we see that $\mu(B) \geq \hat\mu(B)/E_{\hat\mu}$ for all Borel sets $B$. Combining this observation and the previous inequality with \eqref{eqn:exp-return} yields \eqref{eqn:exp-tails}.
\end{proof}

\begin{Remark}
A similar result can easily be formulated for rates of decay other than exponential. What the proof above actually shows is the following: let
\[
\mathcal{X} = \{ (a_n)_{n=1}^\infty : a_n > 0 \text{ for all } n  \text{ and } a_n\to 0 \text{ as } n\to\infty\}
\]
and suppose that $\mathcal{A} \subset \mathcal{X}$ has the property that for every open set $U\subset M$ with $\mu(U\cap \Lambda_\ell)>0$ and every $\ell',T\in \NN$, there is $(a_n)_n \in \mathcal{A}$ such that
\[
\mu \{x\in U\cap \Lambda_{\ell'} : f^{kT}(x) \notin U\cap \Lambda_\ell \text{ for every } 1\leq k\leq n\} \leq a_n \text{ for all } n.
\]
Then $\mu$ lifts to a First $T$-Return Topological Young Tower for some $T>0$ whose tail is governed by one of these sequences $(a_n)_n \in \mathcal{A}$ in the sense that
\[
\hat\mu\{x\in \Gamma : \tau(x) > n \} \leq E_{\hat\mu} a_n \text{ for all } n.
\]
When $\mathcal{A} = \{ (a_n)_{n=1}^\infty \in \mathcal{X} : \varlimsup_{n\to\infty} \frac 1n \log a_n < 0 \}$, this yields Corollary \ref{thm:edc}.
\end{Remark}

\begin{Remark}
If $\mu$ is a hyperbolic measure whose log Jacobian along unstable leaves is sufficiently regular, then control of the tail of the tower implies results on decay of correlations and other statistical properties; see \cite{lY98,lY99,vMD01,CS09,PSZ16,SZ} for upper bounds on correlations, \cite{oS02,sG04,SZ} for lower bounds, and \cite{MN05,MN08} for some other statistical properties (this list is far from comprehensive).

The condition on regularity of the log Jacobian is often satisfied when $\mu$ is an equilibrium measure for a sufficiently regular potential function.  We remark that the results in \cite{CFT,CFT2,BCFT} provide techniques for studying equilibrium states for some classes of non-uniformly hyperbolic systems, which suggest a method for establishing the hypothesis of Corollary \ref{thm:edc}.  (Note that those systems are not surface diffeomorphisms, so the results here do not apply directly.)
\end{Remark}

\section{Hyperbolic branches from almost returns: Theorem \ref{thm:genkatok2}}\label{sec:hyplyap}
 
A remarkable feature of Theorem \ref{thm:existsY} is that   its hypotheses  contain only a minimal amount of structure whereas the conclusions  produce a great deal of non-trivial structure. The fundamental ingredient which we will use to build this structure is that of a \emph{hyperbolic branch} which we proceed to define. The main result of this section, Theorem \ref{thm:genkatok2}, is then a statement on the existence of hyperbolic branches under very mild recurrence conditions.

We assume throughout this section that  \( \Gpq \) is a nice domain  as in Definition \ref{def:nice} and Figure \ref{fig:nice},
and that $A\subset \Gpq$ satisfies the following condition (which in particular is true if $A$ is a nice regular set or a nice rectangle).

\begin{Definition}
Say that $A\subset \Gpq$ \emph{has full length local stable and unstable curves}
if every $x\in A$ is contained in a local stable curve $W_x^s \subset \Gpq$ whose endpoints are on $W_p^u$ and $W_q^u$, and in a local unstable curve $W_x^u$ whose endpoints are on $W_p^s$ and $W_q^s$.
\end{Definition}

The following definitions will allow us to discuss analogues of stable and unstable curves through points in $\Gpq$ that do not necessarily lie in $A$.

\begin{Definition}[Cones]\label{def:cones1}
Given a normed vector space $V$, a \emph{cone} in $V$ is any subset $K \subset V$ for which there is a decomposition $V = E\oplus F$ and $\omega>0$ such that
\[
K = \{0\} \cup \{ v + w : v\in E, w\in F, \|w\| < \omega \|v\| \}.
\]
Say that two cones $K$ and $K'$ are \emph{transverse} if $K\cap K' = \{0\}$.  Note that we do not require $K$ and $K'$ to be defined in terms of the same decomposition $E\oplus F$.
\end{Definition}

\begin{Definition}[Conefields]\label{def:conefield}
A \emph{conefield} over $\Gpq$ is a family of cones $\mathcal{K} = \{K_x \subset T_x M\}_{x\in \Gpq}$; it is continuous if $K_x$ can be defined via decompositions $E_x \oplus F_x$ and widths $\omega_x$ such that $E_x,F_x,\omega_x$ vary continuously in $x$. 
\end{Definition}

\begin{Definition}[$\mathcal{K}$-admissible curves]\label{def:kadmissible}
Given a conefield $\mathcal{K} = \{K_x\}_{x\in \Gpq}$, we say that a $C^1$ curve $\gamma \subset \Gpq$ is \emph{$\mathcal{K}$-admissible}
if for every $x\in \gamma$ we have $T_x\gamma \subset K_x$.
\end{Definition}

\begin{Definition}[Adapted conefields]\label{def:conefields}
Let $A\subset \Gpq$ have full length stable and unstable curves. We say that two transverse continuous conefields
\( \mathcal K^s = \{K^s_x\}\) and \( \mathcal K^u=\{K^u_x\} \) over \(  \Gpq\) are \emph{adapted} for $A$ if $W_x^s$ is $\mathcal{K}^s$-admissible and $W_x^u$ is $\mathcal{K}^u$-admissible for every $x\in A \cup \{p,q\}$.
\end{Definition}

Now we can define the broader family of curves that generalizes the ``true'' stable and unstable curves through points in $A$.

\begin{Definition}[Admissible curves in a nice domain]\label{def:nice-adm}
Given adapted conefields $\mathcal{K}^s$ and $\mathcal{K}^u$ for a set $A \subset \Gpq$ as in the previous definition, we refer to a $\mathcal{K}^s$-admissible curve as a \emph{stable admissible curve}, and a $\mathcal{K}^u$-admissible curve as an \emph{unstable admissible curve}. We say that a stable admissible curve is \emph{full length} if its endpoints lie on $W_p^u$ and $W_q^u$, and an unstable admissible curve is full length if its endpoints lie on $W_p^s$ and $W_q^s$.
\end{Definition} 

\begin{Definition}[Strips in a nice domain]\label{def:nice-strips}
Given $A\subset \Gpq$ and $\mathcal{K}^{s/u}$ as above,
a closed region \( \widehat C^s \subset \Gpq\) is a \emph{stable strip} if it is bounded by two pieces of \( W^u_p, W^u_q\) and two full length stable admissible curves; see Figure \ref{fig:topbranch}. Similarly, a closed region \( \widehat C^u \subset \Gpq\) is an \emph{unstable strip} if it is bounded by two pieces of \( W^s_p, W^s_q\) and two full length unstable admissible curves.
\end{Definition}

\begin{figure}[tbp]
\includegraphics[width=150pt]{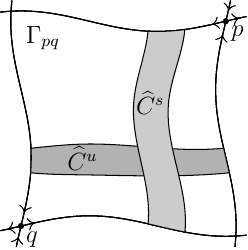}
\caption{Stable strips and unstable strips  in \( \Gpq \). }
\label{fig:topbranch}
\end{figure}

The stable and unstable strips we consider will typically be bounded by admissible curves of the form $W_x^{s/u}$ for some $x$, but we do not require this as part of the definition.

\begin{Definition}[Hyperbolic  Branches]\label{def:hypbranch}
Let $\mathcal{K}^s,\mathcal{K}^u$ be adapted conefields for $A\subset \Gpq$.
Let \(  \widehat C^s , \widehat C^u  \subseteq \Gpq  \) be a stable and an unstable strip respectively and suppose that there exists \(  i>0  \) such that \(  f^{i}(\widehat C^s ) = \widehat C^u  \).  The map 
 \begin{equation}\label{eqn:branch}
f^i\colon \widehat C^s \to \widehat C^u
\end{equation}
is a \( (C, \kappa)\)-\emph{hyperbolic branch} if 
for every \( x\in \widehat C^s\) and \(y=f^i(x)\in \widehat C^u\) we have 
\begin{equation}\label{eqn:inv-cones}
 Df^i_x(  K^u_x) \subset  K^u_{y}
\quad \text{ and } \quad 
 Df^{-i}_y( K^s_y) \subset  K^s_{x}
 \end{equation}
and if for every
$v^u\in K_x^u$ and $v^s \in K_y^s$, the vectors defined for $0\leq j\leq i$ by $v_j^u := Df_x^j(v^u)$ and $v_j^s := Df_y^{-(i-j)}(v^s)$
satisfy
\begin{equation}\label{eqn:branch-contract}
\|v^u_j\| \leq C e^{-\kappa(i-j)} \|v^u_i\| \quad \text{ and }   \quad
\|v^s_j\| \leq C e^{-\kappa j} \|v^s_0\|.
\end{equation}
We call \( i \) the \emph{order} of the hyperbolic branch.
\end{Definition}

\begin{Remark}\label{rem:boundaries}
Stable and unstable strips which form a hyperbolic branch have the special property that their boundaries are pieces of the global stable and unstable manifolds of the points \(p, q\). 
\end{Remark}

The following result plays no role in Theorem \ref{thm:genkatok2}, but is natural to state at this point and is important in \S\ref{sec:sat-rec}.
See \S\ref{sec:inf-po} and Remark \ref{rmk:hyp-branches-Vs} for its proof.

\begin{Proposition}\label{prop:intersect-branches}
Suppose that $I\subset \mathbb{N}$ is infinite and that $\{f^{i} \colon \widehat{C}_i^s \to \widehat{C}_i^u\}_{i\in I}$ is a family of $(C,\kappa)$-hyperbolic branches whose stable strips are nested in the sense that $\widehat{C}_i^s \supset \widehat{C}_j^s$ for all $i,j\in I$ with $j\geq i$. Then $\bigcap_{i\in I} \widehat{C}_i^s$ is a local $(C,\kappa)$-stable curve that has full length in $\Gpq$.

Similarly, if the unstable strips are nested in the sense that $\widehat{C}_i^u \supset \widehat{C}_j^u$ for all $i,j\in I$ with $j\geq i$, then $\bigcap_{i\in I} \widehat{C}_i^u$ is a local $(C,\kappa)$-unstable curve that has full length in $\Gpq$.
\end{Proposition}

\begin{Definition}\label{def:concat}
If $f^i \colon \widehat{C}_1^s \to \widehat{C}_1^u$ and $f^j \colon \widehat{C}_2^s \to \widehat{C}_2^u$ are hyperbolic branches, then $\widehat{C}^s := f^{-i} (\widehat{C}_1^u \cap \widehat{C}_2^s)$ is the stable strip for a hyperbolic branch $f^{i+j} \colon \widehat{C}^s \to \widehat{C}^u := f^j (\widehat{C}_1^u \cap \widehat{C}_2^s)$.  
We refer to this new branch as the \emph{concatenation} of the original two branches. Observe that
$\widehat{C}^s$ consists of the points in the stable strip of the first branch that are mapped by $f^i$ to the stable strip of the second branch, as shown in Figure \ref{fig:concat}.
\end{Definition}

\begin{figure}[tbp]
\includegraphics[width=.9\textwidth]{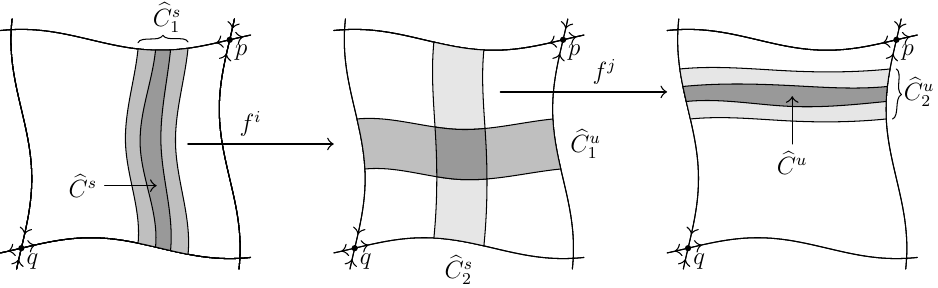}
\caption{Concatenating hyperbolic branches.}
\label{fig:concat}
\end{figure}

If we start with two \( (C, \kappa) \)-hyperbolic branches, then the concatenation of the two has the property that for every $x\in \widehat{C}^s$ and 
$v^{s} \in Df^{-(i+j)}_{f^{i+j}(x)} K_{f^{i+j}(x)}^{s}$, the vectors $v_k^{s} := Df_x^k(v^{s})$ for $0\leq k \leq i+j$ satisfy
$\|v_k^s\| \leq C e^{-\kappa k} \|v_0^s\|$ when $0\leq k \leq i$, and
\[
\|v_k^s\| \leq C e^{-\kappa(k-i)} \|v_i^s\| \leq C^2 e^{-\kappa k} \|v_0^s\|
\]
when $k>i$, with similar estimates on $v_k^u$; thus 
$f^{i+j} \colon \widehat{C}^s \to \widehat{C}^u$ is a $(C^2,\kappa)$-hyperbolic branch.
The fact that $C$ is replaced with $C^2$ can be a problem as the constant would continue to grow with each repeated concatenation. 

In the uniformly hyperbolic setting, one can deal with this problem by passing to an \emph{adapted metric}, or \emph{Lyapunov metric}, in which $C=1$. In non-uniform hyperbolicity this procedure does not give a continuous metric on $M$, and one must restrict to regular level sets. In the next section we use this idea to produce a family of branches that has the following property.

\begin{Definition}[Concatenation property of hyperbolic branches]\label{def:concatenation}
We say that a collection of  \( (C, \kappa) \)-hyperbolic branches has the \emph{concatenation property} if any finite concatenation of these branches is still a \( (C, \kappa) \)-hyperbolic branch. 
\end{Definition}

Given a nice regular set $A \subset \Gpq$, it would be a fairly routine exercise in non-uniform hyperbolicity theory (and the geometry of nice domains) to show that for every $x\in \Gamma$ and $i\in \NN$ such that $f^i(x) \in \Gamma$, there is a hyperbolic branch $f^i \colon \widehat{C}^s \to \widehat{C}^u$ such that $x\in \widehat{C}^s$, and that the collection of branches associated to all returns to $A$ has the hyperbolic branch property. However, when we prove Theorem \ref{thm:existsY} in \S\ref{sec:C->B}, it will become important to work with a larger collection of branches associated to a weaker kind of return.\footnote{We need the larger collection to guarantee that the rectangle we build is ``saturated'', see Remark \ref{rmk:why-almost} and \S\ref{sec:sat-disc}.}

\begin{Definition}[Almost Returns]\label{def:almostreturns}
Let $A\subset \Gpq$ have full length stable and unstable curves.
A point $x\in A$
has an \emph{almost return} to \(  A  \)   at  time \(  i \in T
\mathbb Z   \)  (see Figure \ref{fig:Gamma-p-q}) if
\( f^{i}(x)\in \Gpq \)  and there is  \( y\in A \) such that 
\[
[x,y,i]\neq \emptyset 
\quad \text{ where } \quad [x,y,i] :=  
\begin{cases} 
f^{i}(W^{s}_{x})\cap W^{u}_{y} \quad &\text{ if } i > 0
\\
 f^{i}(W^{u}_{x})\cap W^{s}_{y}\quad &\text{ if } i < 0.
\end{cases}
\]
\end{Definition}

Note that actual returns, where \(x, f^{i}(x) \in A  \), are special cases of almost returns and thus if  \( A \)  is recurrent, then 
every $x\in A$ has infinitely many almost returns in both forward and backward time. 
In \S\ref{sec:C->B} we will construct a rectangle using hyperbolic branches associated to almost returns, for which we need the following definition and result. (These become vacuous for a set without almost returns.)

\begin{figure}[tbp]
\includegraphics{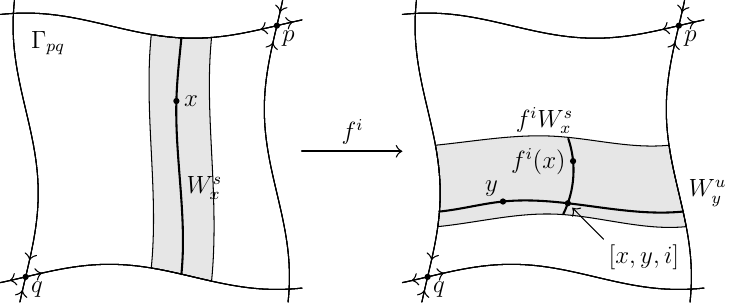}
\caption{An almost return  and the associated hyperbolic branch.}
\label{fig:Gamma-p-q}
\end{figure}

\begin{Definition}[Hyperbolic branch property]\label{def:hyp-branch-prop}
The set $A\subset \Gpq$ has the \emph{$(C,\kappa)$-hyperbolic branch property} if there exist adapted conefields $\mathcal{K}^{s/u}$ such that the following are true:
\begin{enumerate}
\item whenever $x\in A$ has an almost return to $A$ at a time $i\in T\mathbb N$,
there exists a $(C,\kappa)$-hyperbolic branch $f^i \colon \widehat{C}^s \to \widehat{C}^u$ such that $x\in \widehat{C}^s$; 
\item the collection of such branches has the concatenation property.
\end{enumerate}
\end{Definition}

\begin{maintheorem}\label{thm:genkatok2}
Let \(  f\) be a \(  C^{1+\alpha}  \) surface diffeomorphism.  
For every $\chi>\lambda>0$, every $0<\eps<\eps_1 {(f,\chi,\lambda)}$, and every $\ell\in\NN$,
there is $r>0$ such that every \((\chi, \eps, \ell, r)\)-nice regular set has the $(C, \lambda/3) $-hyperbolic branch property for $C = 8\sqrt{2} (1-e^{2(\lambda-\chi)})^{-1/2} e^{2\eps\ell}$.
\end{maintheorem}

The constant $C$ in Theorem \ref{thm:genkatok2} is equal to $\wQ e^{2\eps\ell}$ where $\wQ$ is as in \eqref{def:qhat} below.
The constant \( r \)  determined by Theorem \ref{thm:genkatok2} is the same constant which appears in the statements of Theorem \ref{thm:existsY} and Theorem \ref{thm:existsSRB}, and is given relatively explicitly in \eqref{def:r}. 
In~\S \ref{sec:C->B} we 
use Theorem \ref{thm:genkatok2}, together with some regularity estimates from \S\ref{sec:pseudocurves} and \S\ref{sec:holder}, to prove
Theorem \ref{thm:existsY}, and thus part \ref{A1} of Theorem \ref{thm:existsSRB}. 
Most of that proof only relies on Theorem \ref{thm:genkatok2} and can be read independently of the other sections; the regularity results are only needed for the proof of part \ref{B2} of Theorem \ref{thm:existsY}.

Theorem \ref{thm:genkatok2} implies that every almost return corresponds to a hyperbolic branch. Such branches can be concatenated as in Definition \ref{def:concat}, and we point out that the concatenated branches themselves may or may not correspond to almost returns; see Remark \ref{hyp-br}.

\begin{Remark}
Our construction of a first $T$-return topological Young tower in \S\ref{sec:C->B} is solely based on the hyperbolic branch property. Theorem \ref{thm:genkatok2} ensures this property in the setting of a global invariant hyperbolic set $\Lambda$ although one could imagine verifying this property in other situations using alternate arguments. In this case one would still require some extra information on bounded distortion in order to get the analogue of Theorem \ref{thm:existsY}\ref{B2} about a genuine Young tower.
\end{Remark}

\begin{Remark}\label{rem:nice}
The proof of Theorem \ref{thm:genkatok2} in \S\ref{sec:almost-return} relies on the notion of ``regular branch'' that we introduce in the next section. The hyperbolic branches we seek are restrictions of these regular branches to a nice domain $\Gpq$ (see Figure \ref{fig:get-branch}), and Lemmas \ref{lem:niceint1} and \ref{lem:images} use the niceness property in a crucial way to show that the restricted branch has stable and unstable strips that are actually contained in $\Gpq$. One could attempt to mimic the restriction procedure for a domain $\Gpq$ without the niceness property, but in that case there would be no way to guarantee this containment.
\end{Remark}

\section{Pseudo-orbits: Theorems \ref{thm:pseudo-orbit} and \ref{thm:inf-po}}\label{sec:po}
 
The proofs of Theorems~\ref{thm:existsnice1} and \ref{thm:genkatok2} are  ultimately based on a  new and non-trivial result, Theorem \ref{thm:pseudo-orbit} below,  in the general theory of (nonuniformly) hyperbolic sets, 
which is a generalization of the Katok closing lemma. This result is of independent interest and we expect it to have further applications beyond those presented here. For simplicity, and in view of the setting of our main theorems, we state it for two-dimensional diffeomorphisms but it should generalize in a relatively straighforward way to arbitrary dimension.  

In \S \ref{sec:lyapcharts} we introduce the basic notion of a Lyapunov chart, following  the approach  of Sarig in \cite{oS13},  and use this to define the notion of stable and unstable strips in \S\ref{sec:strips} and of \emph{regular} branch for pseudo-orbits in \S \ref{sec:reg-branches}.  In \S\ref{sec:reg-branches} we state our main result about regular branches for pseudo-orbits, and in \S\ref{sec:pseudocurves} a useful consequence about shadowing orbits.

\subsection{Lyapunov charts}\label{sec:lyapcharts}
Let \(  \Lambda  \) be a \(  (\chi, \epsilon)  \)-hyperbolic set. We fix some    
\[
 b>0 
 \]
sufficiently small to be determined in the course of the proof. This is the only constant which we cannot give \emph{explicitly} in terms of properties of \( f \) and the constants associated to the hyperbolic set \( \Lambda \), however see equations \eqref{eqn:norm-Psix}, \eqref{eqn:Q5b},  \eqref{eq:smallb1},  and Lemma 
\ref{lem:overlapping} for the key places in which conditions on \( b \) appear. 
For   \( x\in \Lambda \),   let $e^s_{x}\in E^s_{x}$, $e^u_{x}\in E^u_{x}$ be unit vectors and define $s,u\colon \Lambda\to [1,\infty)$ by
\begin{equation}\label{def:sxux}
\begin{gathered}
s(x)^2 := {2\sum_{n=0}^\infty e^{2n\lambda} \|Df^n_x e_x^s\|^{2}}, \\
u(x)^2 := {2\sum_{n=0}^\infty e^{2n\lambda} \|Df^{-n}_x e_x^u\|^{2}}.
\end{gathered}
\end{equation}
By \eqref{eq:hypest}, for $\lambda<\chi$,  the sums above converge and therefore  \( s(x), u(x) \) are well defined, though note that they are not uniformly bounded in \( x \). 
 Letting   $e_1 = (1,0), e_2 =(0,1)$ denote  the standard basis vectors in \( \mathbb R^{2} \), define the  linear map \( L_{x}\colon \RR^2 \to T_x M \),   by letting 
\begin{equation}\label{def:L}
L_{x}(e_1) := u(x)^{-1} e^u_x \qand L_{x}(e_2) := s(x)^{-1} e^s_x
\end{equation}
and extending to  \( \mathbb R^{2} \) by linearity. We call \( L_{x} \)
the \emph{Lyapunov change of coordinates} at  \( x \). 
\begin{figure}[tbp]
 \includegraphics[width=\textwidth]{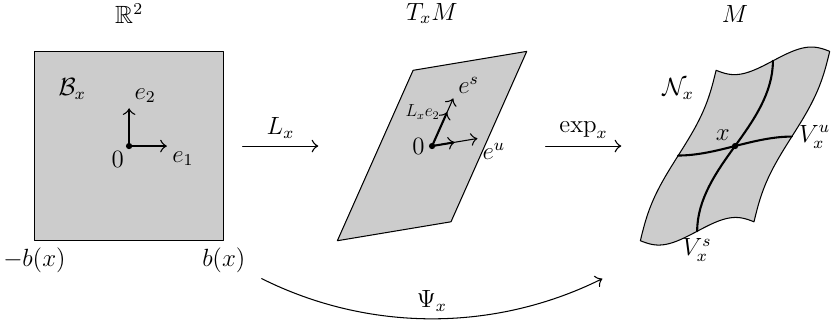}
 \caption{The Lyapunov chart at $x$.}
 \label{fig:lyap-chart}
 \end{figure}
In Lemma \ref{lem:Lbound} we prove the following standard relation between the Riemannian metric and the metric induced by the Lyapunov coordinates: 
for any \(  0< \lambda < \chi  \) let\footnote{We could of course replace \( Q_{0} \) in the expression for \(  \wQ  \) by its explicit value but  various calculations to be given below will be easier and clearer by keeping track of \( Q_{0} \) as an independent constant.}
\begin{equation}\label{def:qhat} 
Q_{0}:=1/8
    \quand  \wQ:=  Q_{0}^{-1}\left({2\sum_{i=0}^{\infty} e^{2(\lambda-\chi) i}} \ \right)^{1/2}
    = \frac{8\sqrt{2}}{\sqrt{1 - e^{2(\lambda-\chi)}}},
\end{equation}
then,
for every $\ell\in \NN$ and $x\in \Lambda_\ell$, we have 
\begin{equation}\label{eqn:norm-Lx}
1\leq \|L_x^{-1}\| \leq 3Q_0 \wQ e^{2\eps\ell}.
\end{equation}
For every    \( x\in \Lambda \)  let 
\begin{equation}\label{eq:bx}
 b(x) := b \left(\sum_{k=-\infty}^{\infty} e^{- 3 |k|\epsilon} \|L_{f^{k}(x)}^{-1}\|\right)^{-\frac 1\alpha}.
\end{equation}
Notice that  the sum converges by \eqref{eq:images} and \eqref{eqn:norm-Lx}. 
For  \( \ell \geq 1 \)  let 
\begin{equation}\label{eq:bell}
b_{\ell}:= b \left(3 Q_{0}\wQ e^{2 \epsilon \ell} \sum_{k=-\infty}^{\infty} e^{- |k|\epsilon}\right)^{- \frac{1}{\alpha}}
\end{equation}
It follows immediately from the definition that 
 \(
 b_{\ell+1}/b_\ell = e^{-2\eps/\alpha}
\) 
and 
\( 
b \geq b(x) \geq b_\ell > 0,
\) 
and it  follows from \eqref{eqn:norm-Lx} (we give a formal proof in Lemma \ref{lem:Pesin}) that 
\( 
e^{-3\epsilon/\alpha} <   b(x)/  b(f(x))< e^{3\epsilon/\alpha}
\).
Then, for every  \( x \in \Lambda_{\ell} \) we define 
\begin{equation}\label{eq:regB}
\mathcal B_{x}^{(\ell)}:= [-b_{\ell}, b_{\ell}]^{2} \subseteq  [-b(x), b(x)]^{2} =:\mathcal B_{x} \subset \RR^2.
\end{equation}
Letting $\exp_{x}\colon T_x M \to M$ be  the exponential map, define \( \Psi_x \colon \mathcal B_{x} \to M \) by
\[
\Psi_x := \exp_x \circ \ L_{x}.
\]
The map $\Psi_x$ is called a \emph{Lyapunov chart} at $x$; see \cite{BarPes07} for a more general notion.\footnote{In \cite{BarPes07} the Lyapunov change of coordinates $L_x$ is required to be \emph{tempered}, but we do not require this condition.}  We write
\begin{equation}\label{eq:regN}
 \mathcal N^{(\ell)}_{x} :=\Psi_{x}(\mathcal B^{(\ell)}_{x}) \subseteq 
\Psi_{x}(\mathcal B_{x}) =:\mathcal N_{x}.
\end{equation}
Notice that \( \Psi_{x}(0)=x \) and therefore \( \mathcal N_{x}^{(\ell)}, \mathcal N_{x} \) are neighbourhoods of  \( x \), which we call  respectively the \emph{regular neighbourhood of level} \(\ell \)  of  \( x \) and  the \emph{regular neighbourhood} of  \( x \). 

\subsection{Stable and unstable strips}\label{sec:strips}
We want to define stable and unstable cones and other objects related to these regular neighbourhoods and Lyapunov charts,
these will be analogous to the corresponding definitions for nice domains in \S\ref{sec:hyplyap}.
For this we  need to introduce  an additional small constant \(  \omega  \) which, for completeness, we define precisely.

Let \( \chi > \lambda > 0 \) and \( 0< \epsilon < \epsilon_{1}(f,\chi,\lambda) \), and \( \Lambda \) a \( (\chi, \epsilon) \)-hyperbolic set.  Notice  that  \eqref{eqn:eps} gives $\eps < \alpha\lambda/18$ and so $3\eps/\alpha < \lambda/6$, which gives $e^{-\lambda/2}e^{3\eps/\alpha} < e^{-\lambda/3}$.  Thus we can choose $\omega>0$ sufficiently small that the following inequalities all hold: 
\begin{equation}\label{eqn:omega}
\begin{aligned}
e^{-\lambda/2} e^{3\eps/\alpha} + e^{-\lambda}\omega &< e^{-\lambda/3},
\\
e^\lambda / \sqrt{1+\omega^2} &\geq e^{2\lambda/3},
\\
(e^{-\lambda/24} - \omega e^{\lambda/24})/\sqrt{1+\omega^2} &> e^{-\lambda/4},
\end{aligned}
\qquad 
\begin{aligned}
(1-\omega) &\geq \sqrt{2(1+\omega^2)} /2,
\\
2\omega &< 1- e^{\lambda/24}e^{-\lambda/3},
\\
e^{-\lambda/24} &< 1/\sqrt{1+e^{-2\lambda}\omega^2}.
\end{aligned}
\end{equation}
From now on we always assume that \( \omega \) satisfies \eqref{eqn:omega}.
The following three definitions are analogues of Definitions \ref{def:conefields}--\ref{def:nice-strips} for nice domains.

\begin{Definition}[Cones in regular neighbourhoods]\label{def:conesize}
For any \(  \lambda>0  \)  
we fix the following cones defined in terms of  standard Euclidean coordinates with  \( v=(v_1, v_2) \in \mathbb R^{2}\): 
\begin{equation}\label{def:cones}
\begin{aligned}
\tilde K^s:= \{v: |v_1| < e^{-\lambda}\omega |v_2|\} \ \subset \ 
&  \{v: |v_1| < \omega |v_2|\}:= K^s ,
\\
\tilde K^u:= \{v: |v_2| < e^{-\lambda}\omega |v_1|\} \ \subset \ 
&\{v: |v_2| <  \omega|v_1|\} := K^u. 
\end{aligned}
\end{equation}
Given $x\in \Lambda$ and $y\in \mathcal{N}_x$, we write
\begin{equation}\label{eqn:M-cones}
K_{x,y}^{s/u} := D_{\Psi_x^{-1}(y)} \Psi_x(K^{s/u})
\quand
\tilde K_{x,y}^{s/u} := D_{\Psi_x^{-1}(y)} \Psi_x(\tilde K^{s/u})
\end{equation}
for the cones in $T_yM$ that correspond to the cones over $\mathcal{B}_x$.
\end{Definition}

\begin{Definition}[Admissible curves in regular neighbourhoods]\label{def:localmanreg}
Let \( x \in \Lambda \).  A \( C^1 \) curve 
\(
\gamma^s\subset \mathcal B_x
\) 
is a \emph{stable admissible curve} (resp.\ \emph{strongly stable admissible curve}) if its tangent directions lie in $K^s$ (resp.\ $\tilde K^s$); similarly for unstable admissible curves \(\gamma^u\subset\mathcal B_x
\). 
In particular, the horizontal and vertical boundaries of \(  \mathcal B_{x}  \) or  \(  \mathcal B_{x}^{(\ell)}  \), are stable and unstable admissible curves respectively and so we   denote them by \(  \partial^{s}\mathcal B_{x},  \partial^{s}\mathcal B_{x}^{(\ell)}  \) and \(  \partial^{u}\mathcal B_{x},  \partial^{u}\mathcal B_{x}^{(\ell)} \) respectively. A stable admissible curve \( \gamma^s \) is a \emph{full length stable admissible curve}, with respect to \( \mathcal B_x \) or \( \mathcal B_x^{(\ell)} \)   if its endpoints lie on distinct components of \( \partial^{u}\mathcal B_x \) or \( \partial^{u}\mathcal B_x^{(\ell)} \) respectively. We define \emph{full length unstable admissible curves} similarly. For a \( C^{1} \) curve \( \gamma^{s/u}\subset \mathcal N_{x} \) we use the same terminology according to the geometry  of the corresponding curve \( \Psi_{x}^{-1}(\gamma^{s/u}) \subset \mathcal B_{x} \). 
\end{Definition}

\begin{Definition}[Strips in regular neighbourhoods] \label{def:strips}
Let \( x \in \Lambda \).
A region \(
 \mathcal B^s_x \subseteq \mathcal B_x^{(\ell)} 
 \)
is a \emph{(strongly) stable strip}   if its boundary is  formed  by two full length (strongly) stable admissible curves and two pieces of \( \partial^u \mathcal B_x^{(\ell)} \). Full length unstable admissible curves and \emph{(strongly) unstable strips} 
\(\mathcal B^u_x \subseteq \mathcal B_x^{(\ell)}\) are defined similarly. Moreover, for subsets \( \mathcal N^{s/u}_{x}\subseteq \mathcal N_{x}^{(\ell)} \), we use the same terminology  depending on the geometry of the corresponding sets \( \Psi_{x}^{-1}(\mathcal N^{s/u}_{x}) \subseteq \mathcal B_{x}^{(\ell)}  \). 
\end{Definition}

\subsection{Pseudo-orbits and regular branches}\label{sec:reg-branches}

\begin{Definition}[Pseudo-orbits]\label{def:po}
Given constants \(  \delta, \lambda>0  \) and a finite sequence \(  \barl = (\ell_0,\dots, \ell_k) \) of positive integers satisfying $|\ell_j - \ell_{j-1}| \leq 1$ for all $1\leq j\leq k$, we  say that a finite sequence of points $\barx = (x_0,\dots,x_k)$ is a \emph{finite ($\barl, \delta, \lambda$)-pseudo-orbit} if for every $0\leq j\leq k$ we have
\[
 x_j \in \Lambda_{\ell_j} \quand d(f(x_{j-1}),x_j) \leq  \delta e^{-\lambda\ell_j}.
 \]
\end{Definition}
Notice that the true orbit of the point \( x\in \Lambda_{\ell} \) is trivially an \( (\barl, \delta, \lambda) \)-pseudo-orbit for \( \delta=0 \) and for the sequence \( \ell_{j}=\ell+j \), recall \eqref{eq:images}. The results we present here are essentially already known for true orbits, but their generalizations to pseudo-orbits is non-trivial and constitute a crucial step in our arguments. 
Given a finite (\( \barl, \delta, \lambda \))-pseudo-orbit we write
\[
\mathcal{N}_\barx^0 := \bigcap_{i=0}^k f^{-i} \mathcal{N}_{x_i}^{(\ell_i)}
\quad\text{and}\quad
\mathcal{N}_\barx^j := f^j(\mathcal{N}_\barx^0)  \text{ for } 1\leq j\leq k
\]
for the sets of points 
corresponding to orbit segments that go through all the Lyapunov neighborhoods of the points $x_i$;   the size of these sets depends on the choice of $\barl$, but we suppress this in the notation.

\begin{Remark}\label{rmk:superscript-notation}
We point out
that one must carefully distinguish between $\mathcal{N}_{\barx}^j$, where the subscript denotes a finite pseudo-orbit and the superscript indicates which point in the pseudo-orbit to consider, and $\mathcal{N}_x^{(\ell)}$, where the subscript denotes a point and the superscript denotes the level of a regular set containing that point. A similar distinction should be made when $\mathcal{N}$ is replaced by $\mathcal{B}$ as in the following.
\end{Remark}

In Lyapunov coordinates, for every \( 0\leq j\leq k \), we write 
\begin{equation}\label{eqn:Bjx}
\mathcal{B}^j_{\barx} := 
\Psi_{x_j}^{-1} 
(\mathcal{N}_\barx^j)
\subseteq \mathcal B_{x_{j}}^{(\ell_{j})}
\end{equation}
and for every  $0\leq i, j \leq k$, 
\begin{equation}\label{eq:Bjxmap}
f_{\barx}^{i,j} := \Psi_{x_j}^{-1} \circ f^{j-i} \circ \Psi_{x_i}\colon \mathcal{B}^i_{\barx} \to \mathcal{B}^j_{\barx},
\end{equation}
which is a diffeomorphism between $\mathcal{B}^i_{\barx}$ and $\mathcal{B}^j_{\barx}$ by definition.
We will show that such maps are hyperbolic branches in a suitable sense. 

\begin{Definition}
[$\barl$-regular branch]\label{def:reg-branch}
A finite (\( \barl, \delta, \lambda \))-pseudo-orbit \emph{$\barx$ determines an $\barl$-regular branch for $f$} if the following are true:
\begin{enumerate}[label=(\roman{*})]
\item $\mathcal{B}_{\barx}^0$, $\mathcal{B}_{\barx}^k$ are stable and unstable strips in $\mathcal{B}_{x_0}^{(\ell_0)}$, $\mathcal{B}_{x_k}^{(\ell_k)}$, respectively;
\item given $0\leq i < j \leq k$, $\underline{y} \in \mathcal{B}_{\barx}^i$, $\underline{z} \in \mathcal{B}_{\barx}^j$, $\underline{v}^u \in K^u$, and $\underline{v}^s \in K^s$,
we have 
\begin{equation}\label{eq:regbranch}
\begin{gathered}
D_{\underline{y}} f_{\barx}^{i,j} \underline{v}^u \in K^u,
\qquad \|D_{\underline{y}} f_{\barx}^{i,j} \underline{v}^u\| \geq e^{\lambda(j-i)/3} \|\underline{v}^u\|, \\
D_{\underline{z}} f_{\barx}^{j,i} \underline{v}^s \in K^s,
\qquad \|D_{\underline{z}} f_{\barx}^{j,i} \underline{v}^s\| \geq e^{\lambda(j-i)/3} \|\underline{v}^s\|.
\end{gathered}
\end{equation}
\end{enumerate}
In the specific case where $\ell_0 = \ell_k = \ell$ and $\ell_j = \min(\ell + j, \ell + k-j)$ for $0<j<k$, we refer to this as an $\ell$-regular branch.
\end{Definition}

Here and in what follows, we use underlined variables such as $\underline{y},\underline{z}$ to represent coordinates in $\RR^2$, while undecorated variables such as $y,z$ will represent points in $M$.

We now state our main result in the general setting of hyperbolic sets. 

\begin{maintheorem}\label{thm:pseudo-orbit}
Let $f\colon M\to M$ be a $C^{1+\alpha}$ surface diffeomorphism.  
For every $\chi>\lambda>0$ and every $0<\eps<\eps_1(f,\chi,\lambda)$,
there are $b,\delta>0$ such that 
if \( \Lambda \) is a  $(\chi,\eps)$-hyperbolic set then every finite \( (\barl, \delta, \lambda \))-pseudo-orbit in $\Lambda$ determines an \( \barl \)-regular branch.
\end{maintheorem}

Recall from Definition \ref{def:concat} and the discussion following it that two $(C,\kappa)$-hyperbolic branches can be concatenated, but the resulting branch is only guaranteed to be $(C^2,\kappa)$-hyperbolic. The crucial advantage we gain by considering $\barl$-regular branches is that the concatenation of two $\barl$-regular branches is again a $\barl$-regular branch; this is because the hyperbolicity estimates in \eqref{eq:regbranch}  are given at the level of the Lyapunov charts, and not on the surface itself. The relationship between the two types of estimates is given by the following computation.

Since $D_0\exp_x$ is the identity map, we can choose $b$ small enough that 
\(  \Psi_{x}\colon \mathcal B_{x}\to \mathcal N_{x}  \) is a diffeomorphism, in particular injective, for every \( x\in \Lambda \), and such that 
\eqref{eqn:norm-Lx} gives 
\begin{equation}\label{eqn:norm-Psix} 
\|D_{\underline{y}}\Psi_x\| \leq 2
\quad \text{ and } \quad 
\|D_y\Psi_x^{-1}\| \leq 4Q_0 \wQ e^{2\eps \ell}
\end{equation}
for every $\ell\in \NN$, $x\in \Lambda_\ell$, $\underline{y} \in \mathcal{B}_x$, and $y\in \mathcal{N}_x$. 
The conclusion of the following result should be compared to Definition \ref{def:hypbranch}.

\begin{Proposition}\label{prop:hyp-branch}
Suppose $\barx=(x_0,\dots,x_k)$ determines an $\barl$-regular branch. Given $y\in \mathcal{N}_\barx^0$ and $v^u \in K_{x_0,y}^u$,
let $v_j^u = D_yf^j(v^u)$ for $0\leq j\leq k$. Then $v_j^u \in K_{x_j, f^j y}^u$ and
\begin{equation}\label{eqn:vjuvku}
\|v_j^u\| \leq \wQ e^{2\eps\ell_k} e^{-\lambda(k-j)/3} \|v_k^u\|.
\end{equation}
Similarly, given $z\in \mathcal{N}_\barx^k$ and $v^s \in K_{x_k,z}^s$, let $v_j^s = D_z f^{-(k-j)}(v^s)$ for $0\leq j\leq k$; then $v_j^s \in K_{x_j, f^{-(k-j)} z}^s$ and
\begin{equation}\label{eqn:vjsvks}
\|v_j^s\| \leq \wQ e^{2\eps\ell_0} e^{-\lambda j/3} \|v_0^s\|.
\end{equation}
\end{Proposition}
\begin{proof}
The inclusions follow immediately from \eqref{eqn:M-cones} and \eqref{eq:regbranch}. For \eqref{eqn:vjuvku} it suffices to observe that
\eqref{eq:regbranch} gives 
\[
\|v_j^u\|
= \|D_{f^ky} (\Psi_{f^jy} \circ f_{\barx}^{k,j} \circ \Psi_{f^k y}^{-1})(v_k^u)\|
\leq 2 \cdot e^{-\lambda(k-j)/3} \cdot 4Q_0 \wQ e^{2\eps \ell_k} \|v_k^u\|
\]
and recall that \( Q_{0}=1/8 \) from~\eqref{def:qhat}. The proof of \eqref{eqn:vjsvks} is analogous.
\end{proof}

In \S\ref{sec:almost-return} we use Theorem \ref{thm:pseudo-orbit} and Proposition \ref{prop:hyp-branch} to prove Theorem~\ref{thm:genkatok2}.

\begin{Remark}\label{rmk:po-branch}
In the theory of \emph{uniformly} hyperbolic systems, it is well-known that every pseudo-orbit segment determines a regular branch as in Definition~\ref{def:reg-branch}; that is, there is $\delta>0$ such that $x_0,\dots, x_k$ determines a regular branch whenever $d(f(x_j), x_{j+1}) < \delta$ for all $0\leq j < k$. 
In non-uniform hyperbolicity, various versions of Theorem \ref{thm:pseudo-orbit} have been obtained. The first result of this type is the well-known Katok closing lemma \cite{aK80}. Other versions were obtained by Hirayama \cite{Hir}, and by Sarig \cite{oS13} in his construction of countable Markov partitions for surface diffeomorphisms with positive topological entropy; this latter result was generalized to higher dimensions by Ben Ovadia \cite{SbO}. What makes our Theorem \ref{thm:pseudo-orbit} different is the  \emph{explicit} relationship between $\ell$ and the pseudo-orbit scale $\delta e^{-\lambda \ell}$; this  is absolutely crucial for our arguments, particularly our construction of hyperbolic branches associated with almost $T$-returns.
\end{Remark}

\subsection{Shadowing}\label{sec:pseudocurves} 

We conclude this section by stating a consequence of Theorem \ref{thm:pseudo-orbit}  for the problem of \emph{shadowing} in the non-uniformly hyperbolic setting. It is classical problem in hyperbolic dynamics whether every pseudo-orbit is ``shadowed'' by a real orbit of the system \cite{sP99}. In uniform hyperbolicity theory a positive answer to this question is a fundamental result, with many applications and in particular,    is a key ingredient in the construction of Markov partitions. Extending it to the setting of non-uniform hyperbolicity has proved challenging and few results have been obtained, see \cite{BarPes07}. Here we give a powerful shadowing result, which follows relatively easily from Theorem \ref{thm:pseudo-orbit}. In addition to the existence of a shadowing orbit, this result gives an explicit estimate on the hyperbolicity constants associated to this shadowing orbit and also proves
 Theorem \ref{thm:HP} on the existence of  local stable and unstable curves (up to optimizing the rate of contraction, see Remark \ref{rmk:HP}).

\begin{Definition}\label{def:infpo}
 Given a bi-infinite sequence $\bar\ell = (\ell_n)_{n\in\ZZ}$ with $|\ell_{n+1} - \ell_n|\leq 1$ for all $n$, a \emph{bi-infinite $(\bar\ell,\delta,\lambda)$-pseudo-orbit} is a bi-infinite sequence $(x_n)_{n\in \ZZ}$ such that for every $n\in\ZZ$, we have $x_n\in \Lambda_{\ell_n}$ and $d(f(x_{n-1}),x_n) \leq \delta e^{-\lambda \ell_n}$. Replacing ``$n\in \ZZ$'' with ``$n\geq 0$'' and ``$n\leq 0$'' gives the definitions of \emph{forward $(\barl,\delta,\lambda)$-pseudo-orbit} and \emph{backward $(\barl,\delta,\lambda)$-pseudo-orbit}, respectively.
\end{Definition}

\begin{maintheorem}\label{thm:inf-po}
Let $f\colon M\to M$ be a $C^{1+\alpha}$ surface diffeomorphism.  
For every $\chi>\lambda>0$ and every $0<\eps<\eps_1(f,\chi,\lambda)$,
there are $b,\delta>0$ such that if \( \Lambda \) is a  $(\chi,\eps)$-hyperbolic set, then the following are true.
\begin{enumerate}[label=\upshape{(\arabic{*})}]
    \item If $(x_n)_{n\geq 0}$ is a forward $(\barl,\delta,\lambda)$-pseudo-orbit in $\Lambda$, then $\bigcap_{n\geq 0} f^{-n}( \mathcal{N}_{x_n}^{(\ell_n)})$ is a $C^{1+\text{H\"older}}$ full length local $(\wQ e^{2\eps \ell_0}, \lambda/3)$-stable curve.
    \item If $(x_n)_{n\leq 0}$ is a backward $(\barl,\delta,\lambda)$-pseudo-orbit in $\Lambda$, then $\bigcap_{n\leq 0} f^{-n}( \mathcal{N}_{x_n}^{(\ell_n)})$ is a $C^{1+\text{H\"older}}$ full length local $(\wQ e^{2\eps \ell_0}, \lambda/3)$-unstable curve.
    \item If $(x_n)_{n\in\ZZ}$ is a bi-infinite $(\bar\ell,\delta,\lambda)$-pseudo-orbit in $\Lambda$, then
there is a unique shadowing point $y\in M$ such that $f^n(y) \in \mathcal{N}_{x_n}^{(\ell_n)}$ for all $n\in \ZZ$. Moreover,   
the point $y$ is $(\lambda/4,2\eps,\ell_0 + \ell')$-regular for $\ell' =\lceil\frac 1{2\eps}\log\wQ\rceil$.
\end{enumerate}
\end{maintheorem}

Let $\sigma$ denote the left shift on $\mathbb{N}^\ZZ$ 
(which contains the possible sequences $\bar\ell$) and
$\Lambda^\ZZ$ (which contains the possible pseudo-orbits $(x_n)_{n\in\ZZ}$).  If $\bar{x}$ is a bi-infinite $(\bar\ell,\delta,\lambda)$-pseudo-orbit, then $\sigma\bar{x}$ is a bi-infinite $(\sigma\bar\ell,\delta,\lambda)$-pseudo-orbit.  It follows from uniqueness of the shadowing point that the map $\barx \mapsto y$ intertwines $\sigma$ and $f$: if $y$ is the unique shadowing point for $\barx$, then $f(y)$ is the unique shadowing point for $\sigma\barx$. Moreover, if the pseudo-orbit is periodic in the sense that $\sigma^n\bar\ell=\bar\ell$ and 
$\sigma^n\bar{x} = \bar{x}$ for some $n\in\NN$, then $f^n(y)=y$.
In particular, this allows us to deduce the Katok Closing Lemma as a specific case of Theorem \ref{thm:inf-po}: if $x,f^n(x) \in \Lambda_\ell$ and $d(x,f^n(x)) < \delta e^{-\lambda \ell}$, then we can take $x_k = f^k(x)$ and $\ell_k = \ell + \min(k,n-k)$ for $0\leq k< n$, and repeat periodically mod $n$, to get a bi-infinite $(\bar\ell,\delta,\lambda)$-pseudo-orbit whose unique shadowing point is the periodic point from the Katok Closing Lemma.

\begin{corollary}\label{cor:bracket-reg}
If $x,y\in \Lambda_\ell$ and $d(x,y)\leq \delta e^{-\lambda \ell}$, then writing
\[
\ell_n = \ell + |n| \quad\text{and}\quad
z_n = \begin{cases} f^{n}(y) & n<0, \\ f^n(x) & n\geq 0 \end{cases}
\]
gives a $(\bar\ell,\delta,\lambda)$-pseudo-orbit $\bar{z}$, whose unique shadowing point $z$ is the bracket $[x,y] = V_x^s \cap V_y^u$.  Thus $\{ [x,y] : x,y\in \Lambda_\ell, d(x,y) \leq \delta e^{-\lambda\ell} \}$ is a $(\lambda/4,2\eps,\ell + \ell')$-regular set, where  $\ell' = \lceil\frac 1{2\eps}\log\wQ\rceil$.
\end{corollary}

\begin{Remark}
The $\delta$ in Corollary \ref{cor:bracket-reg} is the same as in Theorem \ref{thm:inf-po}.  This corollary in particular shows that the bracket $[x,y]$ exists, so that $\delta e^{-\lambda \ell}$ plays the role of $\delta_\ell$ from Theorem \ref{thm:HP}.  In general the $\delta_\ell$ used there could be larger than $\delta e^{-\lambda \ell}$, because that result made no claims about regularity of $[x,y]$.
\end{Remark}

\part{Hyperbolic Theory}\label{part:hyp}

In this part of the paper we develop all the general hyperbolic theory needed for the proofs of our main results, and  prove  Theorems~\ref{thm:pseudo-orbit} and 
\ref{thm:inf-po}. For simplicity we state and prove everything in the two-dimensional setting, as required by our applications of these results.
The results in this part, as well as Theorems \ref{thm:pseudo-orbit} and \ref{thm:inf-po}, do not contain any inherently two-dimensional ideas, and we expect that they hold as stated in higher dimensions as well; we have not written down detailed proofs, but there do not appear to be any conceptual obstructions. This should be contrasted with Theorems \ref{thm:existsSRB}--\ref{thm:genkatok2} and their proofs in Part \ref{part:Y}, where a crucial role is played by the notion of ``nice domain'', which has no obvious higher-dimensional analogue.

The contents of this part are completely self-contained, with no reference to existing results in the literature, and  follow directly from  the definition of \( (\chi, \epsilon) \)-hyperbolic set. 

In  \S\ref{sec:charts} we give some basic estimates related to Lyapunov charts. In \S\ref{sec:lyap} we state and prove Theorem~\ref{thm:existsbranch}, on the hyperbolicity of \( f \) in Lyapunov charts, which is a fundamental result in the theory of hyperbolic sets and the key motivation behind the introduction of Lyapunov charts. In 
\S \ref{sec:overlapping} we state and prove Theorem \ref{thm:overlapping} which gives some conditions guaranteeing that  Lyapunov charts of nearby points are  ``overlapping'' in a suitable sense. Some qualitative versions of this result are known but we give a quantitative version which is not available in the existing literature and which is crucial for our arguments; this is the most involved technical step in the paper. In \S \ref{sec:prove-po} we combine these two results to prove Theorem~\ref{thm:pseudo-orbit}, and in \S \ref{sec:inf-po} we deduce Theorem \ref{thm:inf-po}.

Throughout the proofs, we will write $Q_j$ for various constants that depend only on $f, \chi,\lambda,\eps$ and are independent of $x,y,\ell,n,k$, etc.

\section{Lyapunov chart estimates}
\label{sec:charts}

Recall the definition of Lyapunov charts in \S\ref{sec:lyapcharts} and in particular the functions \( s(x), u(x)  \) defined in \eqref{def:sxux}, the map \( L_{x} \) defined in \eqref{def:L}, and the quantities \( b_{\ell} \) and \( b(x) \) defined in \eqref{eq:bell} and \eqref{eq:bx} respectively. We follow here the basic approach of Sarig in \cite{oS13} and prove some properties of these objects, including~\eqref{eqn:norm-Lx}. 
We start with a couple of simple estimates showing that although the functions \(  s(x), u(x)  \) are not slowly varying along orbits, they are uniformly bounded on each \( \Lambda_{\ell} \) and  satisfy a  bounded variation property along orbits. 

\begin{Lemma}\label{lem:subound}
For every \( \ell\geq 1 \) and  \( x \in \Lambda_\ell \),  we have
\begin{equation}\label{eq:subound}
\sqrt{2}\leq s(x) \leq  Q_{0} \wQ e^{\epsilon \ell}  
 \quad \text{ and } \quad 
\sqrt{2}\leq u(x) \leq   Q_{0} \wQ e^{\epsilon \ell}. 
\end{equation}
\end{Lemma}
\begin{proof}
The lower bounds follow immediately from the definition of $s(x),u(x)$ in \eqref{def:sxux}.
Using the hyperbolicity property \eqref{eq:hypest} and the definition of $Q_0$, $\wQ$ in \eqref{def:qhat}, we have
\[
s(x)^2 \leq 2\sum_{n\geq 0} e^{2n\lambda} C(x)^2 e^{-2n\chi}
= C(x)^2 (Q_0 \wQ)^2,
\]
and then the definition of $\Lambda_\ell$ in \eqref{lambdal} gives the upper bound for $s(x)$.  The upper bound for $u(x)$ is similar.
\end{proof}

\begin{Lemma}\label{lem:boundvar}
There exists a constant \(  Q_1 = Q_1(c_{1}, c_{2},\lambda)>0  \) such that for 
 every \(  x\in \Lambda  \) and unit vectors \( e^{s}_{x}\in E^{s}_{x}, e^{u}_{x}\in E^{u}_{x} \), we have
\begin{gather*}
Q_1^{-1}\leq e^{\lambda }\|D_{x}f(e^{s}_{x})\|  \leq s(x)/s(f(x))  \leq \sqrt{1+ e^{2\lambda }\|D_{x}f(e_x^{s})\|^{2}}\leq Q_1, \\
Q_1^{-1}\leq e^{\lambda }\|D_{x}f(e^{u}_{x})\|^{-1}  \leq u(f(x))/u(x)  \leq \sqrt{1+ e^{2\lambda }\|D_{x}f (e^{u}_{x})\|^{-2}}\leq Q_1.
\end{gather*}
\end{Lemma}
\begin{proof}
For \(  s(x)  \) we have 
\[
\frac{s(x)^{2}}2 :=  \sum_{k=0}^\infty e^{2k\lambda} \|Df^k_x e_x^s\|^{2} = 1 + e^{2\lambda} \|Df_x e_x^s\|^2 \sum_{k=1}^\infty e^{2(k-1)\lambda} \|Df_{f(x)}^{k-1} e_{f(x)}^s\|^2,
\]
and
\[
\sum_{k=1}^\infty e^{2(k-1)\lambda} \|Df^{k-1}_{f(x)} e_{f(x)}^s\|^{2} 
= \sum_{k=0}^\infty e^{2k\lambda} \|Df^{k}_{f(x)} e_{f(x)}^s\|^{2} 
=s(f(x))^{2},
\]
which gives 
\(
s(x)^2 = 2\big( 1 + e^{2\lambda} \|Df_xe_x^s\|^2 s(f(x))^2\big).
\)
Dividing both sides by $s(f(x))^2$ and taking a square root, we obtain
\begin{equation}\label{eqn:ss}
\frac{s(x)}{s(f(x))} = \sqrt{2(s(f(x))^{-2}+ e^{2\lambda }\|Df_{x}e_x^{s}\|^{2})} 
\geq e^\lambda\|Df_x e_x^s\|,
\end{equation}
and the upper bound for $s(x)/s(f(x))$ also follows since $s(f(x))\geq \sqrt{2}$ so $2(s(f(x)))^{-2} \leq 1$.  For uniformity of $Q_1$ it suffices to note that 
$\|Df_x^{\pm 1}\|\le\max(e^{-c_1},e^{c_2})$ by \eqref{eqn:bc}. A similar computation for $u(x)$, using $f^{-1}$ in place of $f$, gives the following analogue of \eqref{eqn:ss}:
\[
  \frac{u(x)}{u(f^{-1}(x))} = \sqrt{2(u(f^{-1}(x)))^{-2}+ e^{2\lambda }\|Df_{x}^{-1}e_x^{u}\|^{2}} \geq e^\lambda \|Df_x^{-1}e_x^u\|,
\]
with the upper bound again coming from $u(f(x)) \geq \sqrt 2$, so that
\[
e^{\lambda }\|Df_{x}^{-1}e^{u}_{x}\|  \leq   \frac{u(x)}{u(f^{-1}(x))}  \leq 
\sqrt{1+ e^{2\lambda }\|Df_{x}^{-1}e^{u}_{x}\|^{2}}.
\]
Applying this to $f(x)$ and $f^{-1}(f(x)) = x$, we get
\[
e^{\lambda }\|Df_{f(x)}^{-1}e^{u}_{f(x)}\|  \leq   \frac{u(f(x))}{u(x)}  \leq 
\sqrt{1+ e^{2\lambda }\|Df_{f(x)}^{-1}e^{u}_{f(x)}\|^{2}}. 
\]
Using that \( \|Df_{f(x)}^{-1}e^{u}_{f(x)}\|  = \|Df_{x}e^{u}_{x}\|^{-1}  \) we get the bounds for \( u \), and uniformity via $Q_1$ comes as it did for $s$.
\end{proof}

An almost immediate, but extremely important, consequence of Lemma \ref{lem:boundvar}, and therefore of the way the functions \(  s(x), u(x)  \) are defined, is the following fundamental result originally proved in \cite{yP76}.\footnote{In \cite{yP76} it is required that $L_x$ is tempered, but this is not necessary for our formulation.} 

\begin{theorem}[Oseledets--Pesin Reduction Theorem]\label{prop:OselPes}
For every $x\in \Lambda$, 
\begin{equation}\label{eq:Dfx1}
L_{f(x)}^{-1}\circ D_{x}f \circ L_{x} = \begin{pmatrix} A_{x} & 0 \\ 0 & B_{x} \end{pmatrix},
\end{equation}
where $L_x$ is given by \eqref{def:L} and $A_x,B_x\in \RR$ satisfy
\begin{equation}\label{eq:OsPes1}
  0 <Q_1^{-1}< B_{x} \leq e^{-\lambda} <1 < e^{\lambda} \leq  A_{x} < Q_1. 
 \end{equation}
\end{theorem}
\begin{proof}
The diagonal form is a  consequence of the invariance of the stable and unstable subspaces \( E^{s}_{x}, E^{u}_{x} \) and the fact that \( L_{x}, L_{f(x)} \) map the coordinate axes to these subspaces, recall \eqref{def:L}. Thus, by linearity of \(  L_{x}  \) we have 
\begin{align*}
A_{x}e_{1} &= L_{f(x)}^{-1}\circ D_{x}f\circ L_{x}(e_{1}) = L_{f(x)}^{-1}\circ D_{x}f \Big(\frac{e^{u}_{x}}{u(x)}\Big) =
\frac{u(f(x))}{u(x)} \|D_{x}f(e^{u}_{x})\|,   \\
B_{x}e_{2} &= L_{f(x)}^{-1}\circ D_{x}f\circ L_{x}(e_{2}) = L_{f(x)}^{-1}\circ D_{x}f \Big(\frac{e^s_x}{s(x)}\Big)=
\frac{s(f(x))}{s(x)} \|D_{x}f(e^{s}_{x})\|,
\end{align*}
and the statement then follows from Lemma \ref{lem:boundvar}. 
\end{proof}

\begin{Lemma}\label{lem:Lbound}
For every \( x\in \Lambda \), we have
\begin{equation}\label{eqn:Lx-1}
1\leq \frac{\sqrt{s(x)^{2}+u(x)^{2}}}{\sqrt 2 \sin \measuredangle(E^{s}_{x}, E^{u}_{x})}\leq \|L_{x}^{-1}\|\leq \frac{ \sqrt{s(x)^{2}+u(x)^{2}}}{\sin \measuredangle(E^{s}_{x}, E^{u}_{x})}.
\end{equation}
In particular, for every \( \ell\geq 1 \), \( x\in \Lambda_\ell\), and $k\in \ZZ$, we have
\begin{equation}\label{eq:Pesin1}
1\leq  \|L_{f^{k}(x)}^{-1}\|  \leq 3 Q_{0}\wQ e^{2 \epsilon \ell} e^{2 \epsilon |k|}.
\end{equation}
\end{Lemma}
\begin{proof}

Let $\theta(x) = \measuredangle(E_x^s,E_x^u)$.
Consider the orthonormal basis $\{e_x^u,(e_x^u)^\perp\}$ in $T_x M$, oriented so that $e_x^s = \cos\theta(x) e_x^u + \sin\theta(x) (e_x^u)^\perp$.  From \eqref{def:L}, we have $L_x e_1 = u(x)^{-1} e_x^u$ and $L_x e_2 = s(x)^{-1} e_x^s$, so the matrices of $L_x^{\pm 1}$ relative to the orthonormal bases $\{e_1,e_2\}$ and $\{e_x^u,(e_x^u)^\perp\}$ have the form
\[
L_x = \begin{pmatrix} u(x)^{-1} & s(x)^{-1} \cos \theta(x) \\ 0 & s(x)^{-1} \sin\theta(x) \end{pmatrix}
\text{ and }
L_x^{-1} = \begin{pmatrix} u(x) & -u(x)/\tan\theta(x) \\ 0 & s(x)/\sin\theta(x) \end{pmatrix}.
\]
The norm of $A=L_x^{-1}$ is the square root of the largest eigenvalue of 
\[
A^T A = \frac 1{\sin^2\theta(x)} \begin{pmatrix} u(x)^2 & -s(x)u(x)\cos \theta(x) \\ -s(x)u(x)\cos\theta(x) & s(x)^2 \end{pmatrix},
\]
and thus a routine computation with the quadratic formula gives
\[
\|L_x^{-1}\|^2 = \frac{u(x)^2 + s(x)^2 + \sqrt{ \big(u(x)^2 + s(x)\big)^2 - 4s(x)^2u(x)^2 \sin^2\theta(x)}}{2\sin^2\theta(x)}.
\]
The square root term lies between $0$ and $u(x)^2 + s(x)^2$, which proves \eqref{eqn:Lx-1}.

For \eqref{eq:Pesin1}, we first observe that $\sin\theta \geq \frac{2\theta}\pi$ for all $\theta \in [0,\frac \pi 2]$.  From \eqref{eq:images}, we see  that if  \( x\in \Lambda_{\ell} \) then  \( f^{k}(x)\in \Lambda_{\ell+|k|} \)  and so  \(  \measuredangle(E^{s}_{f^{k}(x)}, E^{u}_{f^{k}(x)}) \geq e^{-\epsilon (\ell+|k|)} \) and, by \eqref{eq:subound},   \( s(f^{k}(x))\leq Q_{0}\wQ e^{\epsilon(\ell+|k|)} \) and \( u(f^{k}(x))\leq Q_{0}\wQ e^{\epsilon(\ell+|k|)} \). Substituting these bounds into the upper bound from \eqref{eqn:Lx-1} gives
\[
\|L_{f^k(x)}^{-1}\| \leq \frac{\sqrt{2} Q_0 \wQ e^{\eps(\ell+|k|)}}{2 e^{-\eps(\ell+|k|)}/\pi}
= \frac\pi{\sqrt 2} Q_0 \wQ e^{2\eps\ell} e^{2\eps|k|},
\]
which proves \eqref{eq:Pesin1} since $\pi/\sqrt{2} < 3$.
\end{proof}

\begin{Lemma} \label{lem:Pesin}
For  every \( \ell\geq 1 \) and  \( x\in \Lambda_{\ell} \)
\[
b \geq b(x) \geq b_\ell > 0
\quand
e^{-3\epsilon/\alpha} <   b(x)/  b(f(x))< e^{3\epsilon/\alpha}. 
\]
\end{Lemma}

\begin{proof}
From \eqref{eq:Pesin1}, the sum in the definition of \( b(x) \) converges. Therefore, for \( \epsilon \) small and \( x\in \Lambda_{\ell} \) we have 
\[
1\leq \|L_{x}^{-1}\|\leq 
\sum_{k=-\infty}^{\infty} e^{- 3 |k|\epsilon} \|L_{f^{k}(x)}^{-1}\| \leq 
  3 Q_{0}\wQ e^{2 \epsilon \ell} \sum_{k=-\infty}^{\infty} e^{- |k|\epsilon}
\]
and thus  \( b\geq b(x) \geq b_{\ell}>0 \)
as in the first part of the statement. Moreover 
\begin{align*}
b(f(x)) &:= b  \left(\sum_{k=-\infty}^{\infty} e^{- 3 |k|\epsilon} \|L_{f^{k+1}(x)}^{-1}\|\right)^{-1/\alpha} 
\\ &= 
b   \left(\sum_{k=-\infty}^{\infty} e^{- 3 |k-1|\epsilon} \|L_{f^{k}(x)}^{-1}\| e^{3(-|k|+|k|)\epsilon}\right)^{-1/\alpha} 
\\ & = b  \left(\sum_{k=-\infty}^{\infty} e^{- 3 |k|\epsilon} \|L_{f^{k}(x)}^{-1}\|  e^{3(-|k-1|+|k|)\epsilon}\right)^{-1/\alpha}. 
\end{align*}
Notice that \( -|k-1|+|k|  \) can only take the values \( +1 \) or \( -1 \), depending on the value of \( k \), and therefore 
\[
b(x) e^{-3\epsilon/\alpha} 
\leq b  \left(\sum_{k=-\infty}^{\infty} e^{- 3 |k|\epsilon} \|L_{f^{k}(x)}^{-1}\|  e^{3(-|k-1|+|k|)\epsilon}\right)^{-1/\alpha}
\leq b(x) e^{3\epsilon/\alpha}.
\]
This  completes the proof.  
 \end{proof}

\section{Hyperbolicity in Lyapunov charts: Theorem \ref{thm:existsbranch}}
\label{sec:lyap}

The key motivation for introducing Lyapunov charts is to show that the map \( f \) restricted to some neighbourhood of points of \( \Lambda\) is uniformly hyperbolic in Lyapunov coordinates and to study the way \( f \) maps such neighbourhoods to each other. To state this  precisely, 
we consider $x\in \Lambda$ and recall from Figure \ref{fig:lyap-chart}
that the Lyapunov chart $\Psi_x$ maps $\mathcal{B}_x \subset \RR^2$ onto $\mathcal{N}_x \subset M$, so that a point $y\in \mathcal{N}_x$ is represented by coordinates $\underline{y} = \Psi_x^{-1}(y) \in \mathcal{B}_x$. (As in \S\ref{sec:po}, we will use underlined variables to represent coordinates in $\mathcal{B}_x \subset \RR^2$.) To represent the map $f$ in these coordinates, we write 
\[
\mathcal{B}_x^{s,1} := \Psi_x^{-1} (\mathcal{N}_x \cap f^{-1} \mathcal{N}_{f(x)})
\quand
\mathcal{B}_{f(x)}^{u,1} := \Psi_{f(x)}^{-1} (f(\mathcal{N}_x) \cap \mathcal{N}_{f(x)})
\]
and denote by 
\[
f_x := \Psi_{f(x)}^{-1} \circ f \circ \Psi_x \colon \mathcal{B}_x^{s,1} \to \mathcal{B}_{f(x)}^{u,1}
\]
the corresponding diffeomorphism;
see Figure \ref{fig:strips-and-cones}. Similarly, if \( x\in \Lambda_{\ell_{0}} \), and  \( f(x)\in \Lambda_{\ell_{1}} \) for some \( \ell_{1} \) with \( |\ell_{1}-\ell_{0}|\leq 1 \) we let  \( \bar\ell = (\ell_{0}, \ell_{1})\) and write
 \begin{equation}\label{eq:Bsuell}
\mathcal{B}_{x, \barl}^{s,1} := \Psi_x^{-1} (\mathcal{N}_x^{(\ell_{0})} \cap f^{-1} \mathcal{N}_{f(x)}^{(\ell_{1})})
\quand
\mathcal{B}_{f(x), \barl}^{u,1} := \Psi_{f(x)}^{-1} (f \mathcal{N}_x^{(\ell_{0})}\cap \mathcal{N}_{f(x)}^{(\ell_{1})})
\end{equation}

By \eqref{eq:regB},  \( \mathcal{B}_{x, \bar\ell}^{s,1} \subseteq  \mathcal{B}_x^{s,1} \) and \( \mathcal{B}_{f(x), \bar\ell}^{u,1}\subseteq \mathcal{B}_{f(x)}^{u,1} \) and therefore \( f_{x} \) restricts to a map 
\[
f_x \colon \mathcal{B}_{x, \bar\ell}^{s,1} \to \mathcal{B}_{f(x), \bar\ell}^{u,1}
\]
For simplicity we will just use the same notation \( f_{x} \) in both cases. We also mention that in the definitions of  the sets \( \mathcal{B}_{x, \bar\ell}^{s,1}, \mathcal{B}_x^{s,1},  \mathcal{B}_{f(x), \bar\ell}^{u,1}, \mathcal{B}_{f(x)}^{u,1} \)  we implicitly mean the  \emph{connected component},  containing \( x \) or \( f(x) \) respectively, of these sets which \emph{a priori} may not be connected. 
The main result of this section states that the map \( f_{x} \) is uniformly hyperbolic and the sets just defined have a certain specific geometry, see   Figure~\ref{fig:strips-and-cones}.

\begin{figure}[tbp]
\includegraphics[width=\textwidth]{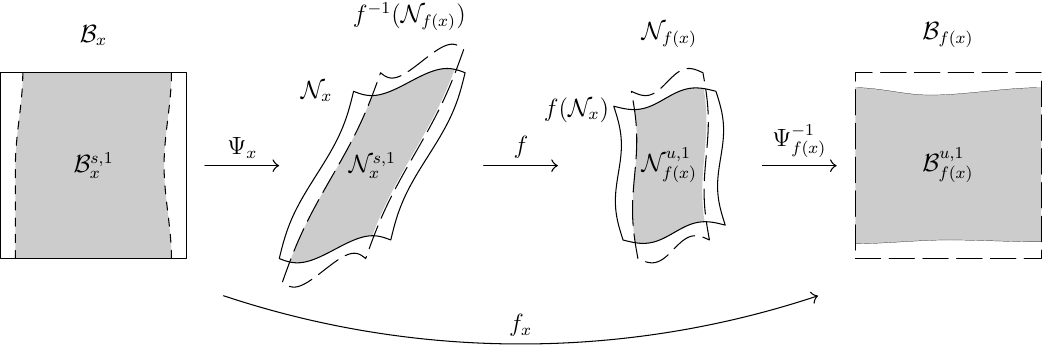}
\caption{The map $f$ in Lyapunov coordinates.}
\label{fig:strips-and-cones}
\end{figure}

A key part of the statement of Theorem \ref{thm:existsbranch} below is that certain sets are stable and unstable strips which are \emph{strictly} contained in the sets \( \mathcal B_{x}^{(\ell)} \). It is convenient to introduce the following sets. 

\begin{Definition}\label{def:tilde-BN}
Given $r>0$, consider the sets
\begin{align*}
\widetilde{\mathcal B}_r^s &:= [-e^{-\lambda/3}r,e^{-\lambda/3}r]\times [-r,r] \subset \RR^2, \\
\widetilde{\mathcal B}_r^u &:= [-r,r] \times [-e^{-\lambda/3}r,e^{-\lambda/3}r] \subset \RR^2
\end{align*}
and let  
\begin{equation}\label{eq:Ntilde}
\widetilde{\mathcal{N}}_x^{s/u} = \Psi_x(\widetilde{\mathcal B}_{b(x)}^{s/u})
\quad\text{ and } \quad 
\widetilde{\mathcal{N}}_{x,\ell}^{s/u} = \Psi_x(\widetilde{\mathcal B}_{b_\ell}^{s/u}).
\end{equation}
\end{Definition}

Recall that \( K_{x,y}^{s/u} \) and \( \widetilde{K}_{x,y}^{s/u} \) are defined in \eqref{eqn:M-cones}.

\begin{maintheorem}\label{thm:existsbranch}
Let $f\colon M\to M$ be a $C^{1+\alpha}$ surface diffeomorphism. 
Fix $\chi>\lambda>0$, $0< \eps < \epsilon_{1}(f,\chi,\lambda)$,  and \( \omega \) satisfying \eqref{eqn:omega}.  Then there is $b>0$ such that  for every $(\chi,\eps)$-hyperbolic set $\Lambda$,
 \( \bar\ell=(\ell_{0}, \ell_{1})  \) with \( |\ell_{0}-\ell_{1}|\leq 1 \) and \( x\in \Lambda_{\ell_{0}}\) with \( f(x)\in \Lambda_{\ell_{1}} \), the following holds: 
for every  $\underline{y}\in \mathcal{B}_x^{s,1}$, $\underline{z}\in \mathcal{B}_{f(x)}^{u,1}$,  $\underline{v}^{u/s} \in K^{u/s}$,
\begin{equation}\label{eqn:cones-0}
\begin{aligned}
D_{\underline{y}} f_x(\underline{v}^u) \in \widetilde K^u,\qquad
&\|D_{\underline{y}} f_x(\underline{v}^u)\| \geq e^{\lambda/2} \|\underline{v}^u\|, \\
D_{\underline{z}} f_x^{-1}(\underline{v}^s) \in \widetilde K^s,\qquad
&\|D_{\underline{z}} f_x^{-1}(\underline{v}^s)\| \geq e^{\lambda/2} \|\underline{v}^s\|.
\end{aligned}
\end{equation}
 Moreover,  the sets \( \mathcal{B}_{x, \barl}^{s,1}, \mathcal{B}_{f(x), \barl}^{u,1} \)  are strongly stable and  unstable strips, satisfying
\begin{equation}\label{eq:strips}
 \mathcal{B}_{x, \barl}^{s,1} \subseteq \widetilde{\mathcal B}_{b_{\ell_{0}}}^s 
\quad \text{ and } \quad  
\mathcal{B}_{f(x), \barl}^{u,1} \subseteq \widetilde{\mathcal B}_{b_{\ell_{1}}}^u. 
\end{equation}

\end{maintheorem}

\begin{Remark}\label{rem:existsbranch}

Notice that the  estimates \eqref{eqn:cones-0}  
hold in particular  for all $\underline{y}\in \mathcal{B}_{x, \barl}^{s,1} $, $\underline z\in \mathcal{B}_{f(x), \barl}^{u,1}$ but they 
\emph{do not depend on} \( \ell \) and give \emph{one-step hyperbolicity} in the sense that the expansion and contraction is exhibited immediately after one iteration.  This is in contrast with the fact that if \( x\in \Lambda_{\ell} \) for large \( \ell \) we have very poor hyperbolicity estimates on the surface, cf. \eqref{eq:hypest} and \eqref{lambdal}. This is of course the effect of the Lyapunov change of coordinates which has controlled, but very large, distortion, and applies to very small neighbourhoods of \( x \) when \( x\in \Lambda_{\ell} \) for large \( \ell \), recall \eqref{eqn:norm-Lx} and \eqref{eq:bx}. The crucial advantage of writing the estimates as in \eqref{eqn:cones-0} is that we can iterate the map any number of times without  loss of hyperbolicity. We only need to worry about the effect of the distortion at the beginning and end of any arbitrarily long piece of orbit in order to recover the actual hyperbolicity estimates for the original map \( f \) on the surface.
\end{Remark}

In the rest of this section we prove Theorem \ref{thm:existsbranch}. 
In \S\ref{sec:derivest} we establish the derivative estimates  \eqref{eqn:cones-0}, in  \S\ref{sec:cone-inv} we use these to prove the invariance property of the cones, and in \S\ref{sec:prop-strips} we prove that  \( \mathcal{B}_{x,\barl}^{s,1}, \mathcal{B}_{x,\barl}^{u,1} \)  are  stable and unstable strips and satisfy \eqref{eq:strips}. 

\subsection{Derivative estimates}\label{sec:derivest}
Here we prove the hyperbolicity estimates  \eqref{eqn:cones-0}. We start with  the special case \( \underline y =0\). 
\begin{Lemma}\label{lem:hyp0}
For every \(  x\in \Lambda  \), \(\underline v^{u}\in K^{u} \),  \( \underline v^{s} \in K^{s}\), 
\begin{equation}\label{eq:hyp0}
\| D_{0}f_{x}(\underline v^{u}) \| \geq e^{2\lambda/3}\|\underline v^{u} \|
\quad \text{ and } \quad 
\| D_{0}f_{x}^{-1}(\underline v^{s}) \| \geq e^{2\lambda/3}\|\underline v^{s}\|. 
\end{equation}
\end{Lemma}

Before proving Lemma \ref{lem:hyp0},  we set up some notation.  Consider the map
\begin{equation}\label{eq:fhat}
 \hat f_{x}:= \exp_{f(x)}^{-1}   \circ f\circ  \exp_{x}\colon T_x M \to T_{f(x)} M.
 \end{equation}
Then since 
  \(
f_{x}:= \Psi_{f(x)}^{-1} \circ f \circ \Psi_x = L_{f(x)}^{-1}  \circ \exp_{f(x)}^{-1}  \circ  f \circ \exp_x \circ L_{x} 
\), we have 
 \begin{equation}\label{eq:f}
f_{x}
= L_{f(x)}^{-1}  \circ \hat f_{x}\circ L_{x}.
\end{equation}
Given \(  v\in T_{x}M  \), we have
\begin{equation}\label{eq:Dfhat}
D_{v}\hat f_{x} = D_{f\circ \exp_{x}(v)}\exp^{-1}_{f(x)} \circ D_{\exp_{x}(v)}f \circ D_{v}\exp_{x}.
\end{equation}
Also  \(  f_{x}\colon \mathcal B_{x}^{s,1}\to \mathcal B_{f(x)}^{u,1}  \)  and for every  \( \underline y\in \mathcal B^{s,1}_{x} \), \( D_{\underline y}f_{x} = D_{\underline y}(L_{f(x)}^{-1}  \circ \hat f_{x}\circ L_{x}) \).  Using that \(  L_{x}, L_{f(x)}  \) are linear maps, we have 
\begin{equation}\label{eq:chart1}
D_{\underline y}f_{x}  =L_{f(x)}^{-1} \circ D_{L_{x}(\underline y)}\hat f_{x} \circ L_{x}
\end{equation}
With this notation we can easily prove Lemma \ref{lem:hyp0}. 

\begin{proof}[Proof of Lemma \ref{lem:hyp0}]
For \( \underline y=0 \) we have \( L_{x}(\underline y) = L_{x}(0) = 0 \), so \eqref{eq:chart1} gives
\begin{equation*}
D_{0}f_{x} =L_{f(x)}^{-1} \circ D_{0}\hat f_{x} \circ L_{x}.
\end{equation*}
Now \(  \exp_{x}(0)=x \) gives \(  f\circ \exp_{x}(0)=f(x) \),
and  the exponential function  is tangent to the identity at 0, i.e., \( D_{0}\exp_{x} = \Id \) and \( D_{f(x)}\exp_{f(x)}^{-1}= \Id \), so \eqref{eq:Dfhat} implies 
\( D_{0}\hat f_{x} = D_{x} f \), which means that we have 
\begin{equation}\label{eq:chart2}
D_{0}f_{x} = L_{f(x)}^{-1} \circ D_{x} f \circ L_{x}. 
\end{equation}
Writing $\underline v^u = v_1 e_1 + v_2 e_2$, \eqref{eq:Dfx1} and \eqref{eq:chart2} give
\(
D_0 f_x(\underline v^u) = A_x v_1 e_1 + B_x v_2 e_2.
\)
Since $\underline v^u\in K^u$ implies $|v_2| \leq \omega |v_1|$, we conclude that
\[
\frac{\|D_0 f_x(\underline v^u)\|^2}{\|\underline v^u\|^2}
= \frac{A_x^2 v_1^2 + B_x^2 v_2^2}{v_1^2+v_2^2} \geq \frac{A_x^2}{1+(v_2/v_1)^2} \geq \frac{e^{2\lambda}}{1+\omega^2} \geq e^{4\lambda/3},
\]
where the last two inequalities use \eqref{eq:OsPes1} and \eqref{eqn:omega}, respectively.  This proves the first half of \eqref{eq:hyp0}; the second is similar.
\end{proof}

By Lemma \ref{lem:hyp0} and continuity of the differential, the expansion estimates in \eqref{eqn:cones-0} hold in some neighbourhood of \( 0 \in \mathcal B^{s,1}_{x} \).  We  show that this neighbourhood contains \( B^{s,1}_{x} \), which is  the key part of the proof of Theorem \ref{thm:existsbranch}. The main step is the following estimate for the derivatives of \( \hat f_{x}\). 

\begin{Lemma}\label{lem:Df0}
There exists \( Q_2>0 \) such that for all \( x \in \Lambda\) and \(  \underline y, \underline z \in\mathcal B^{s,1}_{x}  \), 
\[
\|D_{L_{x}(\underline y)}\hat f_{x}   -  D_{L_{x}(\underline z)}\hat f_{x}\| \leq Q_2 \|\underline y - \underline z\|^{\alpha}.
\]
\end{Lemma}
\begin{proof}
For simplicity we write \(  v:= L_{x}(\underline y), u:= L_{x}(\underline z)  \). Since \(  L_{x}  \) is a contraction and 
\(  \underline y, \underline z \in \mathcal B^{s,1}_{x}\subseteq  [-b(x), b(x)]^{2}\subset [-b, b]^{2}    \), 
we have \(  \|v-u\|\leq \|\underline y - \underline z\| \leq 2\sqrt{2} b  \). In particular  it is sufficient to prove that 
\(
\|D_{v}\hat f_{x}   -  D_{u}\hat f_{x}\| \leq Q \|v - u\|^{\alpha}
\)
for all \(  u, v \in T_{x}M  \) with \(  \|u\|, \|v\|\leq b  \). 
Since $M$ is a $C^2$ Riemannian manifold, there is $Q_3>0$ such that 
\(
 \|D_{v}\exp_{x} - D_{u}\exp_{x}\|\leq Q_3\|u-v\|
 \), 
 \(
 \|D_{f\circ \exp_{x}v}\exp_{x}^{-1} - D_{f\circ \exp_{x}u}\exp_{x}^{-1}\|\leq Q_3\|u-v\|
\)
for all $x\in M$ and $u,v\in T_x M$ with $\|u\|,\|v\|\leq 1$.  Moreover, there is $Q_4>0$ such that $d(\exp_x(u),\exp_x(v)) \leq Q_4 \|u-v\|$ for all such $x,u,v$, and hence
since $Df$ is H\"older continuous we have
\(
\|D_{\exp_{x}(v)}f - D_{\exp_{x}(u)}f\|\leq |Df|_\alpha Q_4^\alpha \|u-v\|^{\alpha}.
\)
Then the definition of 
\(  D_{v}\hat f_{x} , D_{u}\hat f_{x}  \) in  \eqref{eq:Dfhat} gives the result. 
\end{proof}

\begin{Lemma}\label{lem:neardiff}
There exists \( Q_5>0 \) such that for all \( x \in \Lambda\) and \(  \underline y, \underline z \in\mathcal B^{s,1}_{x}  \), 
\[
\|D_{\underline y} f_{x}   -  D_{\underline z} f_{x}\| \leq Q_5 b^{\alpha}.
\]
\end{Lemma}
\begin{proof}
By \eqref{eq:chart1}, using  \( \|L_{x}\|\leq 1 \) and Lemma \ref{lem:Df0}, for every  \(\underline y, \underline z \in \mathcal B^{s,1}_{x}\), 
\begin{equation}\label{eq:Lf}
\begin{aligned}
 \|D_{\underline y}f_{x}- D_{\underline z}f_{x}\|
&= \|L_{f(x)}^{-1}\circ D_{L_{x}(\underline y)}\hat f_{x} \circ L_{x}  - L_{f(x)}^{-1}\circ D_{L_{x}(\underline z)}\hat f_{x} \circ L_{x} \|
\\&= 
\|L_{f(x)}^{-1}\circ (D_{L_{x}(\underline y)}\hat f_{x}   -  D_{L_{x}(\underline z)}\hat f_{x}) \circ L_{x} \|
\\ & \leq \|L_{f(x)}^{-1}\| \cdot \|D_{L_{x}(\underline y)}\hat f_{x}   -  D_{L_{x}(\underline z)}\hat f_{x}\|
\\ 
&\leq Q_2 \|L_{f(x)}^{-1}\| \  \|\underline y - \underline z\|^{\alpha}
\end{aligned}
\end{equation}
Moreover, \( \underline y, \underline z \in \mathcal B^{s,1}_{x} \subset [-b(x), b(x)]^{2}\) implies 
\( \|\underline y - \underline z\| \leq 2b(x)\) and therefore, by Lemma \ref{lem:Pesin}, 
\begin{equation}\label{eq:Lf1}
 \|\underline y - \underline z\|^{\alpha} \leq 2^{\alpha} b(x)^{\alpha} \leq 2^{\alpha} e^{3\epsilon} b(f(x))^{\alpha}.
 \end{equation}
Moreover, from  \eqref{eq:bx} we have 
\[
 b(f(x))^{\alpha} := b^{\alpha} 
 \left(\sum_{k=-\infty}^{\infty} e^{- 3 |k|\epsilon} \|L_{f^{k+1}(x)}^{-1}\|\right)^{-1}
 \leq  b^{\alpha} \|L_{f^{}(x)}^{-1}\|^{-1}
\]
 and therefore, substituting into \eqref{eq:Lf1} and  then into \eqref{eq:Lf} we get 
\[
 \|D_{\underline y}f_{x}- D_{\underline z}f_{x}\|
 \leq  Q_2 \|L_{f(x)}^{-1}\| \  \|\underline y - \underline z\|^{\alpha} \leq 
 2^\alpha e^{3\eps} Q_2 b^\alpha,
\] 
which  completes the proof. 
\end{proof}
Now we can prove the expansion estimates in \eqref{eqn:cones-0} for all $\underline{y}$.
By Lemmas \ref{lem:hyp0} and \ref{lem:neardiff},  for every \(  x\in \Lambda  \), \( \underline y   \in \mathcal B^{s,1}_{x}\), and \(\underline v^{u}\in K^{u} \), we have
\[
\|D_{\underline y}f_x(\underline v^u)\|
\geq \|D_0 f_x(\underline v^u)\| - \|D_{\underline y} f_x - D_0 f_x\| \cdot \|\underline v^u\| \geq (e^{2\lambda/3} - Q_5 b^\alpha) \|\underline v^u\|.
\]
Choose $b>0$ small enough that 
\begin{equation}\label{eqn:Q5b}
e^{2\lambda/3} - Q_5 b^\alpha \geq e^{\lambda/2};
\end{equation}
then we get  \(
\| D_{\underline y}f_{x}(\underline v^{u}) \| \geq e^{\lambda /2}\|\underline v^{u} \|.
\)
A similar argument gives  $\|D_{\underline y}f_{f(x)}^{-1}(\underline v^s) \| \geq e^{\lambda/2}\|\underline v^s\|$ for every $\underline v^s\in K^s$, and so we have  the expansion estimates in~\eqref{eqn:cones-0}.

\subsection{Conefield invariance}\label{sec:cone-inv} 

We now prove the conefield invariance from \eqref{eqn:cones-0}.
Fix $\eta>0$ small enough that if $z=z_1 e_1 + z_2 e_2 \in \RR^2$ has $\|z\|=1$ and $|z_2| < e^{-2\lambda} \omega |z_1|$, then every $v\in \RR^2$ with $\|v-z\|<\eta$ is contained in $\tilde K^u$.  By homogeneity we see that if the assumption on $\|z\|$ is removed and we have $\|v-z\| < \eta \|z\|$, then once again $v\in \tilde K^u$.
Given $x\in \Lambda$ and $\underline v = v_1e_1 + v_2e_2\in K^u$, \eqref{eq:Dfx1} gives
\(
D_0 f_x(\underline v) = A_x v_1e_1 + B_x v_2 e_2,
\)
and we have $|B_x v_2| < e^{-\lambda} \omega |v_1| < e^{-2\lambda} \omega |A_x v_1|$, so $D_0 f_x(\underline v)$ satisfies the assumption on $z$  mentioned above.
Now choose $b$ small enough that 
\begin{equation}\label{eq:smallb1}
Q_5 b^\alpha < \eta.
\end{equation}  
Then for every $\underline y = \Psi^{-1}_{x}(y) \in \mathcal B^{s,1}_{x}$, Lemma \ref{lem:neardiff} gives
\(
\|D_{\underline y} f_x(\underline v) - D_0 f_x (\underline v) \| \leq Q_5 b^\alpha \|\underline v\| < \eta \|\underline v\|,
\)
and by our choice of $\eta$ we conclude that $D_{\underline y} f_x(\underline v) \in \tilde K^u$.  
A completely symmetric argument applies to the stable cones and $f^{-1}$. 

\subsection{Stable and unstable strips}\label{sec:prop-strips}

To complete the proof of Theorem  \ref{thm:existsbranch} we show that $\mathcal{B}_{x, \barl}^{s,1}$ and $\mathcal{B}_{f(x), \barl}^{u,1}$ are strongly stable and unstable strips in $\widetilde{\mathcal{B}}^{s}_{b_{\ell_{0}}}$, $\widetilde{\mathcal{B}}^{u}_{b_{\ell_{1}}}$ respectively, recall \eqref{eq:Bsuell} and \eqref{eq:strips}.  
We begin by proving the statement for  \( \mathcal{B}_{x, \barl}^{s,1}  \). Let 
\(
\gamma^{u,1}_{0}:=\{(v_{1}, 0)\in \mathcal B_{x}^{(\ell_{0})}: f_{x}(v_{1}, 0) \in 
\mathcal B_{f(x)}^{(\ell_{1})}\}\subseteq \mathcal B^{s,1}_{x, \barl}.
\)
Notice that \( f_{x}(0)=0 \) and therefore \(\mathcal B^{s,1}_{x, \barl} \) contains a neighbourhood of \( 0 \) and therefore  \( \gamma^{u,1}_{0}\) is  a non-trivial horizontal segment, and in particular its tangent vectors are contained in the unstable cones \( K^{u} \). Therefore, 
by \eqref{eqn:cones-0},  the images of the tangent vectors to \( \gamma^{u,1}_{0}\) are contained in the strong unstable cones \( \tilde K^{u} \) and  in particular  the slope of the curve \( f_{x}(\gamma^{u,1}_{0}) \) always has absolute value \( < e^{-\lambda}\omega< 1 \).

\begin{figure}[tbp]
\includegraphics[width=\textwidth]{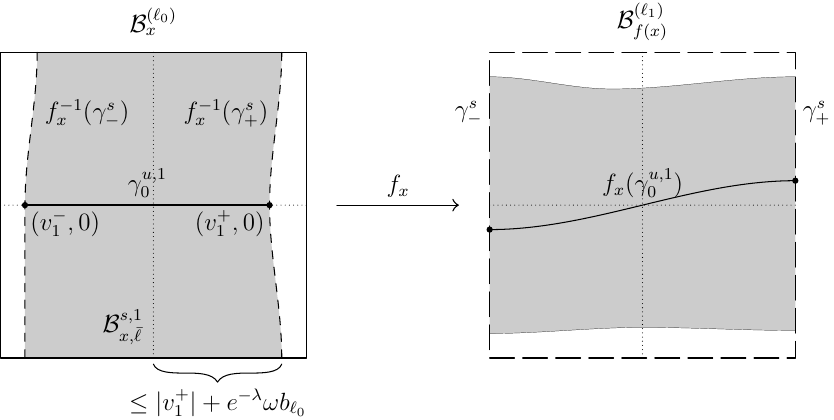}
\caption{$\mathcal{B}_x^{s,1}$ is a strongly stable strip.}
\label{fig:Bxs1}
\end{figure}

Since $f_x(\gamma_0^{u,1})$ goes through the origin, has slope $<1$ in absolute value, and has both endpoints on the boundary of the square 
$\mathcal{B}_{f(x)}^{(\ell_{1})} = [-b_{\ell_{1}},b_{\ell_{1}}]^2$  these endpoints must both lie on the stable boundaries
\(
\gamma^{s}_{\pm}:= \{(\pm b_{\ell_{1}}, v_{2}), v_{2}\in [-b_{\ell_{1}}, b_{\ell_{1}}]\}
\) 
of $\mathcal{B}_{f(x)}^{(\ell_{1})}$, so \( f_{x}(\gamma^{u,1}_{0}) \) is a \emph{full length} strongly unstable admissible curve in $\mathcal{B}_{f(x)}^{(\ell_{1})}$ as shown in Figure \ref{fig:Bxs1}. Therefore the preimages \( f^{-1}_{x}(\gamma^{s}_{\pm})\cap \mathcal B_{x}^{(\ell_{0})} \) are, by \eqref{eqn:cones-0}, strongly stable admissible curves through the endpoints of \(\gamma^{u,1}_{0}\), which are points on the horizontal axis with coordinates \( (v_{1}^{\pm}, 0 )\) with \(| v_{1}^{\pm}| \leq e^{-\lambda/2}b_{\ell_{1}} \).  Note that this last estimate comes from the fact that every tangent vector to $f_x(\gamma_0^{u,1})$ is contracted by a factor of at least $e^{-\lambda/2}$ under the action of $f_x^{-1}$. Since $f^{-1}_x(\gamma^s_\pm)$ are strongly stable admissible curves in 
$\mathcal{B}_x^{(\ell_{0})}$, we conclude as shown in Figure \ref{fig:Bxs1} that their horizontal distance from the $y$-axis in $\RR^2$ is always bounded above by
\[
|v_1^\pm| + e^{-\lambda}\omega b_{\ell_{0}} \leq e^{-\lambda/2}b_{\ell_{1}}+ e^{-\lambda}\omega b_{\ell_{0}}
< (e^{-\lambda/2} e^{3\eps/\alpha} + e^{-\lambda}\omega) b_{\ell_{0}},
\]
where the last inequality uses Lemma \ref{lem:Pesin}.  By \eqref{eqn:omega}, this is $< e^{-\lambda/3} b_{\ell_{0}}$, which proves that 
$\mathcal{B}_{x, \barl}^{s,1}$ is a  strongly stable   strip in $\widetilde{\mathcal{B}}^{s}_{b_{\ell_{0}}}$. 
Similarly, $\mathcal{B}_{f(x), \barl}^{u,1}$ is a strongly unstable strip in $\widetilde{\mathcal{B}}^{u}_{b_{\ell_{1}}}$. 

\section{Overlapping charts: Theorem \ref{thm:overlapping}}
\label{sec:overlapping} 

\subsection{Overlapping charts}\label{sec:overlapping-charts}

The parameters defining Lyapunov charts vary slowly along orbits but in general only measurably with the point \( x\in \Lambda \).
 On each regular level set $\Lambda_\ell$, the dependence is continuous, and it is well-known that ``if points $x,y \in \Lambda_\ell$ are close, then their Lyapunov charts are close''.  The condition on how close $x,y$ need to be depends on $\ell$; we need an explicit quantitative estimate, which is provided by \eqref{eqn:overlapping} in Theorem \ref{thm:overlapping} below.
This is the core technical result of the paper, whose proof demands the largest share of our efforts.

First we make precise what it means for two Lyapunov charts to be close.
Let $\chi>\lambda>0$ be fixed, $\eps_1$ given by \eqref{eqn:eps}, and $\eps\in (0,\eps_1)$.
Let $\Lambda$ be a $(\chi,\eps)$-hyperbolic set.
Given \(  \ell\geq 1  \) and \(  x,y\in \Lambda_{\ell}  \), recall that 
 $\mathcal{N}_x^{(\ell)}, \mathcal{N}_y^{(\ell)}$
are defined in \eqref{eq:regN} and \( \widetilde{\mathcal N}_{x,\ell}^{s/u}, \widetilde{\mathcal N}_{y,\ell}^{s/u} \) in \eqref{eq:Ntilde}, and that the notion of ``full length (un)stable admissible curves'' is defined in Definition \ref{def:localmanreg}.

\begin{Definition}[Overlapping charts]\label{def:overlapping}
 We   say that  
 $\mathcal{N}_x^{(\ell)}$  and  $\mathcal{N}_y^{(\ell)}$ \emph{are overlapping}
 if \( x, y\in  \mathcal{N}_x^{(\ell)} \cap \mathcal{N}_y^{(\ell)}\)
and the following conditions hold: 
\begin{description}
\item[A) Overlapping derivative estimates]
for every $\underline z\in \Psi_{x}^{-1}(\mathcal{N}_x^{(\ell)} \cap \mathcal{N}_y^{(\ell)})$ and every \( \underline v^{u}\in \tilde K^{u} \),  \( \underline v^{s}\in \tilde K^{s} \), we have
\begin{gather}\label{eq:overlap}
D_{\underline z}(\Psi^{-1}_{y}\circ \Psi_{x}) (\underline v^{u})\subset K^{u}
\ \text{ and } \ 
 \|D_{\underline z} (\Psi_y^{-1}\circ \Psi_x)(\underline v^{u})\|\geq e^{-\lambda/24} \|\underline v^{u}\|,  \\
\label{eq:overlaps}
D_{\underline z}(\Psi^{-1}_{x}\circ \Psi_{y}) (\underline v^{s})\subset K^{s}
\ \text{ and } \
 \|D_{\underline z} (\Psi_x^{-1}\circ \Psi_y)(\underline v^{s})\|\geq e^{-\lambda/24} \|\underline v^{s}\|,  
\end{gather}
and similarly with the roles of $x$ and $y$ reversed.
\item[B) Overlapping stable and unstable strips] 
Every full length strongly stable  admissible curve $\gamma^s \subset \widetilde{\mathcal N}_{x,\ell}^s$ (resp. \( \widetilde{\mathcal N}_{y,\ell}^s \)) completely crosses $\widetilde{\mathcal N}_{y,\ell}^u$ (resp. \( \widetilde{\mathcal N}_{x,\ell}^s \)) and
every full length strongly unstable  admissible  curve $\gamma^u \subset \widetilde{\mathcal N}_{x,\ell}^u$ (resp. \(  \widetilde{\mathcal N}_{y,\ell}^u\)) completely crosses $\widetilde{\mathcal N}_{y,\ell}^s$ (resp. \(  \widetilde{\mathcal N}_{x,\ell}^u\)).
\end{description}
\end{Definition} 

\begin{maintheorem}\label{thm:overlapping}
Let $f$ be a $C^{1+\alpha}$ surface diffeomorphism.  
For every $\chi>\lambda>0$, and every $0<\eps<\eps_1(f,\chi,\lambda)$,
there exists $\delta>0$ such that
given any $(\chi,\eps)$-hyperbolic set $\Lambda$, any integer $\ell\in \NN$, and $x,y\in \Lambda_\ell$ with 
\begin{equation}\label{eqn:overlapping}
d(x,y) \leq \delta e^{-\lambda \ell},
\end{equation}
 the Lyapunov charts $\mathcal N_{x}^{(\ell)}$ and $\mathcal N_{y}^{(\ell)}$ are overlapping. 
\end{maintheorem}

In the rest of this section we prove Theorem \ref{thm:overlapping}.  The proof  depends on two intermediate results about the H\"older continuity of $E_x^{s/u}$ and of $s(x), u(x)$ which we prove in \S \ref{sec:holder} and \ref{sec:su-holder} respectively. In \S\ref{sec:overlapnearby} and \S\ref{sec:overlappingstrips} we combine these to prove parts A) and B) respectively in the definition of overlapping charts, thus completing the proof of Theorem \ref{thm:overlapping}. 

\begin{Remark}\label{rmk:overlapping}
A similar `overlapping charts' condition plays a central role in Sarig's construction of countable Markov partitions; see \cite[Definition 3.1]{oS13} for the formal definition and \cite[Proposition 3.2]{oS13} for the key properties, which are similar to our definition above, although it is not immediately clear whether one definition implies the other. The crucial ingredient of our approach here is the explicit distance criterion \eqref{eqn:overlapping} that guarantees overlapping charts, which has no analogue that we are aware of in \cite{oS13}.
\end{Remark}

\subsection{H\"older continuity of the splitting}
\label{sec:holder}

\begin{Proposition}\label{prop:holder-on-regular}
There is $Q_6>0$ such that for  any $\ell\in \NN$ and $x,y\in \Lambda_\ell$, we have
\[
d(E_x^s,E_y^s) \leq Q_6 e^{6\eps\gamma\ell} d(x,y)^\beta\quad \text{ and } \quad d(E_x^u,E_y^u) \leq Q_6 e^{6\eps\gamma\ell} d(x,y)^\beta
\]
where $d(\cdot, \cdot)$ represents distance in the Grassmannian of $M$.
\end{Proposition}

Recall that $\beta $ and \( \gamma\) are given in \eqref{eqn:greek0}.
For generality we prove Proposition \ref{prop:holder-on-regular} as a special case of the following result which does not require \(  x,y  \) to belong to a \(  (\chi, \epsilon)  \)-hyperbolic set. More specifically the H\"older continuity only depends on the angle and hyperbolicity estimates at the points \(  x,y  \) and not on how these vary along the orbits of \(  x,y  \). 

\begin{Proposition}\label{prop:theta-holder}
Let $M$ be a compact smooth Riemannian manifold and $f\colon M\to M$ a 
$C^{1+\alpha}$ diffeomorphism. There is a constant $Q_6$, depending only on $M$, 
$\|Df^{\pm 1}\|$, $\alpha$, $|Df^{\pm 1}|_\alpha$, and $\chi$, such that if $C,K>0$ and $x,y\in M$ are such that
\begin{gather}\label{eqn:Cvs}
\|Df_x^n e_x^s\| \leq C e^{-\chi n}
\text{ and }
\|Df_y^n e_y^s\| \leq C e^{-\chi n}
\text{ for all } n\geq 0, \\
\label{eqn:Cvu}
\|Df_x^n e_x^u\| \geq C^{-1} e^{\chi n} 
\text{ and }
\|Df_y^n e_y^u\| \geq C^{-1} e^{\chi n} 
\text{ for all } n\geq 0,
\end{gather}
for some unit vectors $e_x^{s/u} \in T_x M$, $e_y^{s/u}\in T_yM$ for which the corresponding subspaces $E_{x/y}^{s/u}$ satisfy
 $\measuredangle(E_x^s,E_x^u) \geq K$, $\measuredangle(E_y^s,E_y^u) \geq K$,
then we have
\begin{equation}\label{eqn:Esu-holder}
d(E^s_x,E^s_y) \leq Q_6 (C^2K^{-1})^{2\gamma} d(x,y)^\beta,
\end{equation}
where $\gamma,\beta$ are as in \eqref{eqn:greek0}.  The same bound holds for $E_x^u,E_y^u$ if we have
\begin{gather}\label{eqn:Cvu-}
\|Df_x^{-n}e_x^u\| \leq C e^{-\chi n} 
\text{ and }
\|Df_y^{-n}e_y^u\| \leq C e^{-\chi n} 
\text{ for all } n\geq 0, \\
\label{eqn:Cvs-}
\|Df_x^{-n}e_x^s\| \geq C^{-1} e^{\chi n} 
\text{ and }
\|Df_y^{-n}e_y^s\| \geq C^{-1} e^{\chi n} 
\text{ for all } n\geq 0.
\end{gather}
In particular, if \eqref{eqn:Cvs}, \eqref{eqn:Cvu}, \eqref{eqn:Cvu-}, and \eqref{eqn:Cvs-} all hold, then
\begin{equation}\label{eqn:theta-holder}
|\measuredangle(E_x^s,E_x^u) - \measuredangle(E_y^s,E_y^u)|
\leq Q_6 (C^2K^{-1})^{2\gamma} d(x,y)^\beta.
\end{equation}
\end{Proposition}

To see that Proposition \ref{prop:theta-holder} implies Proposition \ref{prop:holder-on-regular}, it suffices to observe that given $\eps>0$ and $\ell\in \NN$, the conditions \eqref{eqn:Cvs}, \eqref{eqn:Cvu}, \eqref{eqn:Cvu-}, \eqref{eqn:Cvs-} are satisfied for all $x,y\in \Lambda_\ell$, with $C = e^{\eps \ell}$ and $K =e^{-\eps\ell}$.  

We will give the explicit calculations only for the stable subspaces, leading to the proof of \eqref{eqn:Esu-holder}, as the situation for the unstable subspaces is completely symmetrical. 
We follow Brin's approach in \cite[Appendix A]{BarPes01} (see also \cite[\S5.3]{BarPes07}; the main idea of the argument goes back to \cite{BK87}).  The only thing we need that is not given there is the computation for how much vectors in $(E_x^s)^\perp$ are expanded, depending on $C$ and $K$, given in Lemma \ref{lem:extra} below.  The rest of the proof of Proposition \ref{prop:theta-holder} is taken nearly verbatim from \cite[Appendix A]{BarPes01}, with notation adjusted to fit our current setting.  

First note that  by the Whitney embedding theorem \cite{mH94}, we can choose  $N\in\NN$ such that $M$ can be smoothly embedded in $\RR^N$.  By compactness of $M$, the Riemannian metric is uniformly equivalent to the distance induced by the embedding and therefore it suffices to prove the result under the assumption that $M\subset \RR^N$.  Then,  for each $x\in M$ write $E^\perp(x)$ for the orthogonal complement to $T_x M \subset \RR^N$; since $E^\perp$ is smooth it suffices to prove the result with $E_x^s$ replaced by $\tilde E^s_x := E^s_x \oplus E^\perp(x)$.

\begin{Definition}\label{def:Dxn}
Given $x\in M \subset \RR^N$ and $n\in \NN$, let $D_x^{(n)}$ be the $N\times N$ matrix representing the linear map that takes $v\mapsto D_x f^n(v)$ for $v\in T_x M$ and $v\mapsto 0$ for $v\in E^\perp(x)$.
\end{Definition}

Since we embed $M$ in $\RR^N$, we can treat Grassmannian distance between subspaces as follows. 
Given a subspace \(  E\subset \mathbb R^{N}  \),  we define the  distance of a non-zero vector \(  v  \) from the subspace \(  E  \) by considering the unique decomposition \(  v = v^{E}+v^{\perp}  \) where \(  v^{E}\in E  \) and \(  v^{\perp}\perp E  \) and letting 
\( 
d(v, E) := \|v^{\perp}\|/\|v\|. 
\) 
We can then define the distance between two subspaces \(  E, E'\subset \mathbb R^{N}  \) by
\begin{equation}\label{eqn:dEE'}
d(E, E') := \sup\{d(v, E): v\in E'\setminus \{0\}\}=\sup\{d(v, E'): v\in E\setminus \{0\}\}.
\end{equation}
The strategy of the proof is based on the following general result. 

\begin{Lemma}\label{lem:A1}\cite[Lemma A.1]{BarPes01}
Let \(  N\geq 2  \) and 
let $\{A_k\},\{B_k\},$  be two sequences of real $N\times N$ matrices satisfying the following properties

\begin{enumerate}
\item there are $\Delta \in (0,1)$ and $c_{3}>0$ such that 
\begin{equation}\label{eqn:D}
\|A_k - B_k\| \leq \Delta e^{c_{3}k} \text{ for all } k\geq 0;
\end{equation}
\item  there are subspaces $E_A,E_B \subset \RR^N$, $\chi>0$, and $C'>1$ such that 
\begin{equation}
\label{eqn:AkC'}
\|A_k v_{A} \| \leq C' e^{-\chi k} \|v_{A}\|
\quad\text{ and } \quad
\|A_k v_{A}^{\perp} \| \geq (C')^{-1} e^{\chi k} \|v_{A}^{\perp}\| 
\end{equation}
for every \(  v_{A}\in E_A,  v_{A}^{\perp}\perp E_A \), $k\geq 0$, and 
\begin{equation}
\label{eqn:BkC'}
\|B_k v_{B} \| \leq C' e^{-\chi k} \|v_{B}\|
\quad\text{ and } \quad
\|B_k v_{B}^{\perp} \| \geq (C')^{-1} e^{\chi k} \|v_{B}^{\perp}\| 
\end{equation}
 for every \( v_{B}\in E_B,  v_{B}^{\perp}\perp E_B \), $k\geq 0$. 
\end{enumerate}
Then 
\begin{equation}\label{eqn:EAEB}
d(E_A,E_B) \leq 3 (C')^2 e^{2\chi} \Delta^{\frac{2\chi}{c_{3}+\chi} }.
\end{equation}
\end{Lemma}

\begin{proof}
We start by fixing $q := -(\chi + c_{3})$, and let $k_{0} := \lfloor \frac{\log\Delta}q \rfloor$, so that 
\[
(k_{0}+1)q < \log\Delta \leq k_{0}q \leq \log \Delta - q
\]
(recall that \(  q<0  \)). 
In particular, 
 \begin{equation}\label{eq:k0a}
 \Delta e^{c_3k_{0}} \leq e^{k_{0}q} e^{ak_{0}} = e^{-\chi k_{0}}
\end{equation}
and,  letting  $\xi := \frac{2\chi}{c_{3}+\chi} = -\frac{2\chi}q$, 
\begin{equation}\label{eq:k0b}
 e^{-2\chi k_{0}} = (e^{qk_{0}})^{\frac{-2\chi}q} = (e^{qk_{0}})^\xi
\leq e^{-q\xi} \Delta^\xi= e^{2\chi} \Delta^\xi,
 \end{equation}
where the inequality uses $k_0 q < -q + \log \Delta$.
Then, by \eqref{eqn:D}, for every $v_{B}\in E_B$ and every \(  k\geq 1  \) we have
\[
\|A_k v_{B}\| \leq \|B_k v_{B}\| + \|A_k - B_k\| \cdot \|v_B\| \leq C' e^{-\chi k}\|v_{B}\| + \Delta e^{c_3k} \|v_{B}\|
\]
and therefore, in particular, for \(  k=k_{0}  \), by \eqref{eq:k0a}, 
\[
\|A_{k_{0}} v_{B}\| \leq  C' e^{-\chi k_{0}}\|v_{B}\| + \Delta e^{c_3k_{0}} \|v_{B}\| 
\leq 2C' e^{-\chi k_{0}} \|v_{B}\|. 
\]
This implies that 
\[
E_B \subset R_A := \{v\in \mathbb R^{N}: \|A_{k_{0}} v\| \leq 2C' e^{-\chi k_{0}} \|v\| \}.
\]
Clearly we also have \(  E_{A}\subset R_{A}  \) and therefore it is sufficient to estimate the ``width'' of \(  R_{A}  \). 
For  $v\in R_A$, write $v = v_{A} + v_{A}^\perp$, where $v_{A}\in E_A$ and $v_{A}^\perp \perp E_A$. Then by \eqref{eqn:AkC'}, for any \(  k\geq 1   \) we have
\[
\|A_k v\| \geq \|A_k v_{A}^\perp\| - \|A_k v_{A}\|
\geq (C')^{-1} e^{\chi k} \|v_{A}^\perp\| - C' e^{-\chi k} \|v_{A}\|,
\]
and therefore, for \(  k=k_{0}  \), using also that \(  \|v_{A}\|\leq \|v\|  \) since the the splitting of \(  v  \) is orthogonal, we get 
\[
\|v_{A}^\perp\| \leq C' e^{-\chi k_{0}}(\|A_{k_{0}} v\| + C' e^{-\chi k_{0}} \|v_A\|) 
\leq 
3(C')^2 e^{-2\chi k_{0}} \|v\|,
\]
which, by  \eqref{eq:k0b}, implies 
\(  d(v,E_A) \leq 3(C')^2 e^{-2\chi k_{0}}\leq 3 (C')^2 e^{2\chi} \Delta^\xi  \) and therefore, from the definition of \(  \xi  \), the conclusions of the Lemma. 
\end{proof}
The following two Lemmas give the estimates we need to apply Lemma \ref{lem:A1}.  Recall that $c_1,c_2,c_3$ are as in \eqref{eqn:bc}, that $M$ is embedded in $\RR^N$, and that $D_x^{(n)}$ are the matrices defined in Definition \ref{def:Dxn}.

\begin{Lemma}\label{lem:A2}
\cite[Lemma A.2]{BarPes01}
There is  $Q_7\geq 1$ such that for all \(  x, y \in M  \) and every   $n\geq 1$, we have
\[
\|D_{x}^{(n)}- D_{y}^{(n)}\|  \leq Q_7 e^{c_3 n} \|x-y\|^\alpha.
\]
\end{Lemma}
\begin{proof}
We prove the Lemma by induction on \( n \).  
 For \( n=1 \), since  \( f \)  is \( C^{1+\alpha} \) we have 
\(  \|D_{x}^{(1)}- D_{y}^{(1)}\|\leq |Df|_\alpha \|x-y\|^{\alpha} \)  and therefore the statement holds for any \( Q_7\geq |Df|_\alpha \). 
 Then, by the chain rule,  we have
\[
D_x^{(n+1)} - D_y^{(n+1)} = D_{f^n(x)}^{(1)} D_x^{(n)} - D_{f^n(y)}^{(1)} D_y^{(n)};
 \] 
by adding and subtracting \( D_{f^{n}(x)}^{(1)} D_{y}^{(n)} \), taking norms, using the inductive assumption and   the fact   that $\|f^n x - f^n y\| \leq e^{c_2 n} \|x-y\|$ for all $x,y\in M$, we get 
\begin{align*}
\|D_x^{(n+1)} - &D_y^{(n+1)}\|
\leq \|D_{f^n(x)}^{(1)}\| \cdot \|D_x^{(n)} - D_y^{(n)}\|
+ \|D_{f^n(x)}^{(1)} - D_{f^n(y)}^{(1)}\| \cdot \|D_y^{(n)}\| \\
&\leq e^{c_2} Q_7 e^{c_3n} \|x-y\|^\alpha + |Df|_\alpha e^{c_2 n \alpha} \|x-y\|^\alpha e^{c_2 n} \\
&\leq Q_7 e^{c_3(n+1)} \|x-y\|^\alpha \big( e^{c_2 -c_3} +
 |Df|_\alpha Q_7^{-1}e^{(1+\alpha) c_2 n} e^{-c_3(n+1)}\big).
\end{align*}
Since by (\ref{eqn:bc}), $c_3 > (1+\alpha) c_2$, we can choose  $Q_7$ sufficiently large so that  the quantity inside the brackets is less than 1 for every $n$, which completes the proof.
\end{proof}

\begin{Lemma} \label{lem:extra}
Suppose $A_k$ is a sequence of $N\times N$ matrices and $\RR^N = E^s \oplus E^u$ is a splitting such that $\measuredangle(E^s,E^u) \geq K > 0$, and for every $k\geq 1$ and $v^u\in E^u$, $v^s\in E^s$ we have
\begin{equation}\label{eqn:Akvsu}
\|A_k v^s\| \leq C e^{-\chi k} \|v^s\|
\quand
\|A_k v^u\| \geq C^{-1} e^{\chi k} \|v^u\|.
\end{equation}
Then for every $w\perp E^s$ and every $k\geq 1$, we have
\begin{equation}\label{eqn:Akw}
\|A_k w\|\geq (2C^2 K^{-1})^\gamma  e^{\chi k} \|w\|,
 \end{equation}
where $\gamma := \frac{\chi-c_1}{2\chi}$ as in \eqref{eqn:greek0}.
\end{Lemma}
\begin{proof}
Writing $w = w^u + w^s$ where $w^u \in E^u, w^s \in  E^s$, from \eqref{eqn:Akvsu} we get  
\[
\|A_k w \| \geq \|A_k w^u\| - \|A_k w^s\|
\geq C^{-1} e^{\chi k} \|w^u\| - C e^{-\chi k} \|w^s\|.
\]
Let \(  \theta= \measuredangle (w^{s}, w^{u})  \) and note that \(  \theta \geq K  \) and since $w\perp w^s$, we have $\|w\| = \|w^u\| \sin \theta \leq \|w^u\|$ and $\|w\| =\|w^s\| \tan\theta \geq \|w^s\| \tan K \geq \|w^s\|  K $.  Plugging this into the equation above, gives 
\begin{equation}\label{eqn:Akgeq}
\|D_x f^k w \|  \geq (C^{-1} e^{\chi k} - CK^{-1} e^{-\chi k}) \|w\|. 
\end{equation}
Now fix $k_{0} := \lfloor (2\chi)^{-1} \log(2C^2 K^{-1}) \rfloor$, 
Then for   $k \leq k_{0}$ we have 
\[
\frac{\|A_k w \|}{e^{\chi k} \|w\|} \geq \frac{e^{c_1k}}{e^{\chi k}}
\geq e^{(c_1-\chi) k_{0}} 
\geq e^{(c_1-\chi)(2\chi)^{-1} \log (2C^2 K^{-1})}
= (2C^2 K^{-1})^{\frac {c_1-\chi}{2\chi}} 
\]
where we recall that $c_1<0$ (see (\ref{eqn:bc})). The formula for $\gamma$ gives the required estimate.

It remains to treat $k> k_{0}$.  In this case we have
$$
e^{-2\chi k} \leq e^{-2\chi (k_{0}+1)} \leq e^{-\log(2C^2K^{-1})} = \tfrac 12 C^{-2}K,
$$
which gives
\[
CK^{-1} e^{-\chi k} \leq \frac 12 C^{-1} e^{\chi k}
\]
and hence, \eqref{eqn:Akgeq} gives
\[
\|A_k w\| \geq (2C)^{-1} e^{\chi k}\|w\|.
\]
Since $\gamma>1$, $C\geq1$, and $K\leq 1$, we have
$$
(2C^2 K^{-1})^{-\gamma} \leq (2C^2K^{-1})^{-1} = (2C)^{-1}(CK^{-1})^{-1} \leq (2C)^{-1},
$$
and thus we get the result in this case also, thus completing the proof. 
\end{proof}

To complete the proof of Proposition \ref{prop:theta-holder}, we will apply Lemma \ref{lem:A1} with
\[
A_k = D_x^{(k)},\quad  B_k = D_y^{(k)}, \quad
\Delta = Q_7 \|x-y\|^\alpha,\quad C' =  (2C^2 K^{-1})^\gamma.
\]
To verify the conditions of Lemma \ref{lem:A1}, first note that Lemma \ref{lem:A2} gives
\[
\|A_k - B_k\| = \|D_x^{(k)} - D_y^{(k)}\|
\leq Q_7 e^{c_3 k} \|x-y\|^\alpha = \Delta e^{c_k k},
\]
so that \eqref{eqn:D} holds. For \eqref{eqn:AkC'}, we observe that the hypotheses \eqref{eqn:Cvs} and \eqref{eqn:Cvu} of Proposition \ref{prop:theta-holder} give
\[
\|A_k e_x^s \| \leq C e^{-\chi k}
\quad\text{and}\quad
\|A_k e_x^u \| \geq C^{-1} e^{\chi k}
\]
for all $k\geq 0$, so that Lemma \ref{lem:extra} can be applied with $v^s = e_x^s$ and $v^u = e_x^u$ to obtain
\[
\|A_k v_A^\perp\| \geq (2C^2 K^{-1})^\gamma e^{\chi k} \|v_A^\perp\|
= C' e^{\chi k} \|v_A^\perp\|
\]
for all $v_A^\perp \perp E_x^s$. This establishes \eqref{eqn:AkC'}, and  \eqref{eqn:BkC'} follows similarly.
Thus Lemma \ref{lem:A1} applies, 
 and using \eqref{eqn:greek0} to write $\frac{2\chi}{c_3+\chi} = \frac{\beta}\alpha$, we have
\begin{equation}\label{eqn:dExys}
d(E_x^s,E_y^s) \leq 3(C')^2 e^{2\chi} \Delta^{\frac{\beta}{\alpha}}
= 3(2C^2 K^{-1})^{2\gamma} e^{2\chi} (Q_7 d(x,y)^\alpha)^{\frac{\beta}{\alpha}},
\end{equation}
which  completes the proof of Proposition \ref{prop:theta-holder}
by taking $Q_6 = 3(2)^{2\gamma} e^{2\chi} Q_7^{\beta/\alpha}$.

\subsection{H\"older continuity of Lyapunov coordinates}
\label{sec:su-holder}

In this section we prove that $s,u\colon \Lambda_\ell\to [\sqrt{2},Q_0\wQ e^{\eps\ell}]$ are H\"older continuous with exponent $\zeta$ and constant given in terms of $e^{\eps\eta\ell}$, where $\zeta,\eta$ are given in \eqref{eqn:greek0}.
Observe that the H\"older exponent $\zeta$ depends on $\chi-\lambda$, and decays to 0 as $\lambda \to \chi$ where $\chi$ is the decay rate associated to the $(\chi,\eps)$-hyperbolic set $\Lambda$ and
$\lambda<\chi$ is the rate used in the definition of $s(x)$, \(  u(x)  \). 

\begin{Proposition}\label{prop:su-holder}
There is $Q_8>0$ such that for any $\ell\in \NN$ and $x,y\in \Lambda_\ell$, we have
\begin{equation}\label{eqn:s-holder}
|s(x) - s(y)| \leq Q_8 e^{\eps\eta\ell} d(x,y)^\zeta \text{ and }   |u(x) - u(y)| \leq Q_8 e^{\eps\eta\ell} d(x,y)^\zeta.
\end{equation}
\end{Proposition}

We give the argument for the upper bound for \( |s(x) - s(y)| \);  the argument for \( |u(x) - u(y)| \) is analogous.  Recall first that by definition, 
$$
s(x)^2 = 2\sum_{n\geq 0} e^{2\lambda n} \|Df_x^n e^s_{x}\|^2.
$$
Since $s(x),s(y)\geq 1$, we have
\begin{equation}\label{eq:sum}
|s(x) - s(y)| 
\leq \frac{|s(x)^2 - s(y)^2|}{2} \leq \sum_{n\geq 0} e^{2\lambda n} \big|\|Df_x^n e_x^s\|^2 - \|Df_y^n e_y^s\|^2\big|.
\end{equation}
Notice that   $x,y\in \Lambda_\ell$ gives $\|Df_x^n e_x^s\| \leq e^{\eps\ell}e^{-\chi n}$, $\|Df_y^n e_y^s\| \leq e^{\eps\ell}e^{-\chi n}$, and so  
\begin{equation}\label{eqn:bound-1}
\Delta_n = \Delta_{n}(x,y) := \big|\|Df_x^n e^s_{x}\|^2 - \|Df_y^n e^s_{y}\|^2\big|\leq 2e^{2\eps\ell} e^{-2\chi n}.
\end{equation}
Plugging this into \eqref{eq:sum} gives a uniform bound for \(  |s(x) - s(y)|   \) but is not sufficient for our purposes since it does not include \(  d(x,y)  \) and does not therefore imply H\"older continuity. It will nevertheless be useful to bound the tail of the sum for large values of \(  n  \). For small \(  n  \) we need a more sophisticated estimate on $\Delta_n$, as follows. 

\begin{Lemma}\label{lem:bound-2}
There is $Q_9>0$ such that for all $\ell\in \NN$ and  \(  x, y \in \Lambda_{\ell}  \) we have 
\begin{equation}\label{eqn:bound-2}
\Delta_n \leq Q_9 e^{(2+\alpha\beta)c_2 n} e^{6\eps\gamma\alpha(\ell+n)} d(x,y)^{\alpha\beta}.
\end{equation}
\end{Lemma}
The proof of Lemma \ref{lem:bound-2} uses the H\"older continuity of the hyperbolic splitting from Proposition \ref{prop:holder-on-regular} and so  for clarity we isolate the specific estimate in which this property is used. 

\begin{Sublemma}
There is $Q_{10}>0$ such that for all $\ell\in \NN$, \(  x, y \in \Lambda_{\ell}  \), and \(  k\geq 0  \), we have 
\begin{equation}\label{eq:holder1}
 \big| \|Df_{f^kx} e^s_{f^kx} \| - \|Df_{f^ky} e^s_{f^ky}\| \big| 
\leq Q_{10} e^{6\eps\gamma\alpha(\ell+k)} d(f^kx,f^ky)^{\alpha \beta}.
\end{equation}
\end{Sublemma}
\begin{proof}
Since  $Df$ is H\"older on $TM$  we have 
\[
\big| \|Df_{f^kx} e^s_{f^kx} \| - \|Df_{f^ky} e^s_{f^ky}\| \big|
\leq |Df|_\alpha d(e_{f^kx}^s, e_{f^ky}^s)^\alpha 
\]
By \eqref{eq:images},  $f^kx,f^ky\in \Lambda_{\ell + k}$ and therefore, by Proposition \ref{prop:holder-on-regular},  
\[
 d(e_{f^kx}^s, e_{f^ky}^s)^\alpha 
\leq |Df|_\alpha(Q_6 e^{6\eps\gamma(\ell+k)}d(f^kx,f^ky)^\beta)^\alpha
\]
which gives the result. 
\end{proof}

\begin{proof}[Proof of Lemma \ref{lem:bound-2}]
By  \eqref{eqn:bc} the norm of \(  \|Df \|   \) is bounded above by  $e^{c_2}$, and therefore, using the formula for the difference of two squares,   
\begin{equation}\label{eqn:dif-squares}
\Delta_n 
\leq 2e^{c_2n} \big|\|Df_x^n e_{x}^s\| - \|Df_y^n e_{y}^s\| \big|. 
\end{equation}
Moreover, by  the chain rule we have 
\begin{equation}\label{eq:chain}
\big|\|Df_x^n e_x^s\| - \|Df_y^n e_y^s\| \big|
= \bigg| \prod_{k=0}^{n-1} \|Df_{f^kx} e_{f^kx}^s\| - \prod_{k=0}^{n-1} \|Df_{f^ky} e_{f^ky}^s\| \bigg| 
\end{equation}
and therefore, applying the standard equality for the difference of two products 
\(
\left|\prod_{k=0}^{n-1} a_{k}-\prod_{k=0}^{n-1} b_{k}\right|
 = \left|\sum_{k=0}^{n-1} a_{0}...a_{k-1}(a_{k}-b_{k})b_{k+1}...b_{n-1}\right|
\), 
 and using that the absolute value of each individual term is bounded by \(  e^{c_{2}}  \), we get 
\[
 \bigg| \prod_{k=0}^{n-1} \|Df_{f^kx} e_{f^kx}^s\| - \prod_{k=0}^{n-1} \|Df_{f^ky} e_{f^ky}^s\| \bigg| 
\leq e^{c_2 n} \sum_{k=0}^{n-1} \big| \|Df_{f^kx} e_{f_kx}^s \| - \|Df_{f^ky} e_{f^ky}^s\| \big|.
\]
Substituting this into \eqref{eq:chain} and \eqref{eqn:dif-squares} and using \eqref{eq:holder1}, we get 
\begin{equation}\label{eqn:Dn2}
\begin{aligned}
\Delta_n 
&\leq 2e^{2c_2n}\sum_{k=0}^{n-1} \big| \|Df_{f^kx} e_{f^kx}^s \| - \|Df_{f^ky} e_{f^ky}^s\| \big|
\\ &\leq 2Q_{10} e^{2c_2n}\sum_{k=0}^{n-1} e^{6\eps\gamma\alpha(\ell+k)} d(f^kx,f^ky)^{\alpha \beta}.
\end{aligned}
\end{equation}
Using  the bound \(  e^{c_{2}}  \) for the derivative we get  \(  d(f^kx,f^ky) \leq e^{k c_{2}}d(x,y)  \) and therefore, plugging this into \eqref{eqn:Dn2} and rearranging the terms we get 
\[
\Delta_n \leq 
2Q_{10} e^{2c_2 n} e^{6\eps\gamma\alpha\ell}d(x,y)^{\alpha\beta}
\sum_{k=0}^{n-1} e^{(6\eps\gamma+c_{2}\beta) \alpha k}.
\]
To bound the geometric sum, we write
\[  
 \sum_{k=0}^{n-1} e^{(6\eps\gamma+c_{2}\beta) \alpha k} = 
\frac{e^{(6\eps\gamma + c_2 \beta)\alpha n} - 1}
{e^{(6\eps\gamma + c_2 \beta)\alpha } - 1}  \leq \frac {e^{(6\eps\gamma + c_2 \beta)\alpha n}}
{e^{(6\eps\gamma + c_2 \beta)\alpha } - 1}
\]
and so we conclude that 
\[
\Delta_n \leq 
\frac{2Q_{10}}{e^{(6\eps\gamma + c_2\beta)\alpha} - 1}
e^{2c_2 n} e^{6\eps\gamma\alpha \ell} d(x,y)^{\alpha\beta}
e^{(6\eps\gamma + c_2\beta)\alpha n}
\]
which gives the result. 
\end{proof}

\begin{proof}[Proof of Proposition \ref{prop:su-holder}]
We want to use \eqref{eqn:bound-1} for large $n$, and \eqref{eqn:bound-2} for small $n$; the transition happens at the point where the two bounds are roughly equal.  Thus we choose $N$ such that
\[
e^{2\eps\ell} e^{-2\chi N} \approx e^{(2+\alpha\beta)c_2 N} e^{6\eps\gamma\alpha(\ell+N)} d(x,y)^{\alpha\beta};
\]
more precisely, we take
\begin{equation}\label{eqn:N}
N = \left\lfloor
\frac{2\eps\ell - 6\eps\gamma\alpha\ell - \alpha\beta\log d(x,y)}
{2\chi + (2+\alpha\beta)c_2 + 6\eps\gamma\alpha}
\right\rfloor,
\end{equation}
so that 
\begin{multline}\label{eqn:N1}
(2\chi + (2+\alpha\beta)c_2  +6\eps\gamma\alpha)N  
\leq 2\eps\ell - 6\eps\gamma\alpha\ell - \alpha\beta\log d(x,y) \\
\leq (2\chi + (2+\alpha\beta)c_2  +6\eps_0\gamma\alpha)(N+1).
\end{multline}
Note that there is a number $\rho>0$ which depends only on $\alpha$, $\beta$, $\ell$, and $\eps$ such that the numerator in \eqref{eqn:N} is positive provided $d(x,y)<\rho$. Continuing with this assumption our choice of \(  N  \) and  the bound in \eqref{eqn:bound-1} give
\begin{align*}
\sum_{n=N}^\infty e^{2\lambda n} \Delta_n
&\leq \sum_{n=N}^\infty 2e^{2\lambda n} e^{2\eps\ell} e^{-2\chi n}  \\
&\leq \frac{2}{1-e^{-2(\chi - \lambda)}} e^{2\eps\ell} e^{2(-\chi + \lambda)N}
= Q_{11} e^{2\eps\ell} e^{2(-\chi+\lambda)(N+1)}
\end{align*}
where $Q_{11} = 2(1-e^{-2(\chi-\lambda)})^{-1} e^{2(\chi-\lambda)}$.  Then 
the second inequality in \eqref{eqn:N1} and
the definitions of the constants \(  \iota, \eta, \zeta  \) in \eqref{eqn:greek0} give
\begin{equation}\label{eqn:Ninfty}
\begin{aligned}
\sum_{n=N}^\infty e^{2\lambda n} \Delta_n &\leq Q_{11} e^{2\eps\ell}e^{-2(\chi - \lambda) \frac{2\eps\ell - 6\eps\gamma\alpha\ell - \alpha\beta \log d(x,y)}{6\eps_0\gamma\alpha + (2+\alpha\beta)c_2 + 2\chi}} \\ 
&= Q_{11} e^{2\eps\ell} e^{-\iota(2\eps\ell - 6\eps\gamma\alpha\ell - \alpha\beta\log d(x,y))} \\
&= Q_{11} e^{(2(1-\iota)+6\gamma\alpha\iota)\eps\ell} d(x,y)^{\alpha\beta\iota} 
\leq Q_{11} e^{\eps \eta\ell} d(x,y)^\zeta,
\end{aligned}
\end{equation}
where the last inequality uses the fact that $2(1-\iota) + 6\gamma\alpha\iota \leq 2 + 6\gamma\alpha\iota = \eta$.
Turning our attention to the finite part of the sum, \eqref{eqn:bound-2} gives
\begin{align*}
\sum_{n=0}^{N-1} e^{2\lambda n} \Delta_n
&\leq Q_9 \sum_{n=0}^{N-1} e^{2\lambda n} e^{(2+\alpha \beta)c_2 n} e^{6\eps\gamma\alpha(\ell + n)} d(x,y)^{\alpha\beta} \\
&\leq Q_{12} e^{6\eps\gamma\alpha\ell} e^{(6\eps\gamma\alpha+(2+\alpha\beta)c_2+2\lambda)N} d(x,y)^{\alpha\beta} 
\\&
= Q_{12} e^{6\eps\gamma\alpha\ell} e^{(6\eps\gamma\alpha+(2+\alpha\beta)c_2+2\chi) (1-\iota) N} d(x,y)^{\alpha\beta} 
\end{align*}
for some constant $Q_{12}$ independent of $\ell,x,y$.
Applying   the first inequality in \eqref{eqn:N1} gives 
\[
e^{6\eps\gamma\alpha\ell} e^{(6\eps\gamma\alpha + (2+\alpha\beta)c_2+2\chi)N} d(x,y)^{\alpha\beta} \leq e^{2\eps\ell},
\]
and thus 
\begin{align*}
\sum_{n=0}^{N-1} e^{2\lambda n} \Delta_n
&\leq Q_{12} e^{2\epsilon \ell} e^{-\iota (6\eps\gamma\alpha+(2+\alpha\beta)c_2+2\chi) N} \\
&= Q_{13} e^{2\epsilon \ell} e^{-\iota (6\eps\gamma\alpha+(2+\alpha\beta)c_2+2\chi) (N+1)}
\end{align*}
for $Q_{13} = Q_{12} e^{\iota  (6\eps\gamma\alpha+(2+\alpha\beta)c_2+2\chi)}$.   Now the second inequality in \eqref{eqn:N1} gives
\begin{align*}
\sum_{n=0}^{N-1} e^{2\lambda n} \Delta_n& 
\leq  Q_{13} e^{2\epsilon \ell} e^{-\iota (2\eps\ell - 6\eps\gamma\alpha\ell - \alpha\beta\log d(x,y))}
\\& 
= Q_{13} e^{(2(1-\iota)+ 6\gamma\alpha \iota)  \epsilon \ell}   d(x,y)^{\alpha\beta\iota} \leq Q_{13}e^{\eta\epsilon \ell}d(x,y)^{\zeta},
\end{align*}
where the last inequality again uses $2(1-\iota) + 6\gamma\alpha\iota \leq \eta$.
Adding  \eqref{eqn:Ninfty} and using \eqref{eq:sum} completes the proof of Proposition \ref{prop:su-holder} in the case $d(x,y)<\rho$. 
When $d(x,y)\geq \rho$, \eqref{eqn:bound-1} gives
\begin{equation}\label{eqn:large-dxy}
\Delta_n\leq 2e^{2\eps\ell}e^{-2\chi n} \rho^{-\zeta} d(x,y)^\zeta
\end{equation}
and thus \eqref{eq:sum} gives
\[
|s(x)-s(y)| \leq \sum_{n\geq 0} e^{2\lambda n}\Delta_n
\leq \frac{2 e^{2\eps\ell} \rho^{-\zeta}}{1- e^{-2(\chi-\lambda)}}  d(x,y)^\zeta,
\]
which completes the proof because $\eta\geq 1$, so $e^{2\eps\ell} \leq e^{2\eps\eta\ell}$.
\end{proof}

\subsection{Overlapping derivative estimates}\label{sec:overlapnearby}

We are now ready to begin the proof of  Theorem  \ref{thm:overlapping}.  We consider two points  \(  x,y \in \Lambda_{\ell}  \)   with the property that $d(x,y) \leq \delta e^{-\lambda\ell}$, as in \eqref{eqn:overlapping},
and prove that the corresponding regular neighbourhoods at level \( \ell \) are overlapping, subject to certain conditions on $\delta$.  Crucially, these conditions will not depend on $\ell$.

 In this section we prove the derivative estimates involved in the definition of overlapping charts.  In fact, we will prove here a slightly stronger version of \eqref{eq:overlap} by showing  that \eqref{eq:overlap} holds for all $\underline z\in \Psi_{x}^{-1}(\mathcal{N}_x \cap \mathcal{N}_y)$.  The analogous statement \eqref{eq:overlaps} is completely symmetric. 

\begin{figure}[tbp]
\includegraphics[width=.8\textwidth]{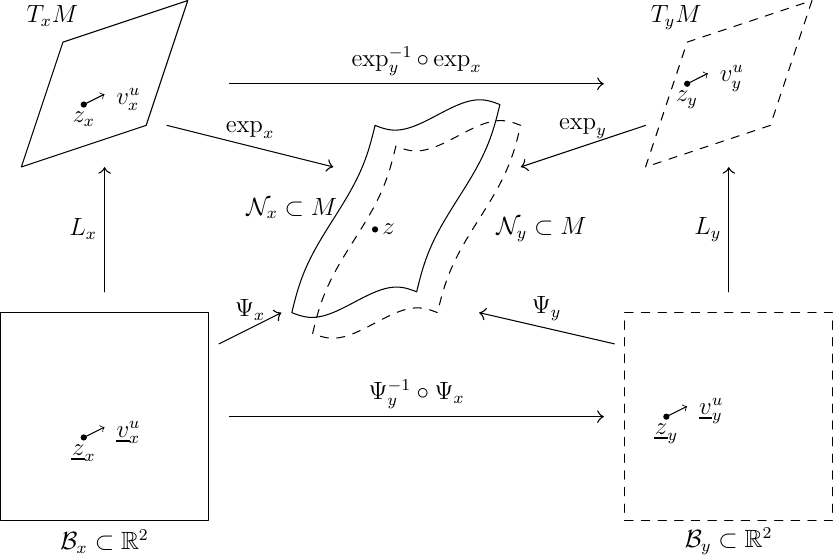}
\caption{Applying the transition map.}
\label{fig:overlapping4}
\end{figure}

For \(  x,y \in \Lambda_{\ell}  \) with \(  \mathcal N_{x}\cap \mathcal N_{y} \neq \emptyset  \),  let 
\(
z\in \mathcal N_{x}\cap \mathcal N_{y}
\)
and denote
\[
\underline z_{x}:= \Psi_{x}^{-1}(z)\in \mathcal B_{x}, \quad \text{ and } \quad 
\underline z_{y}:= \Psi_{y}^{-1}(z)\in \mathcal B_{y}.
\]
Then 
  \((\Psi_{y}^{-1}\circ \Psi_{x})(\underline z_{x}) = \underline z_{y}\)
 and  
\(
 D_{\underline z_{x}} (\Psi_{y}^{-1}\circ \Psi_{x})\colon T_{\underline z_{x}}\mathcal B_{x} \to   T_{\underline z_{y}}\mathcal B_{x}. 
\)
We 
consider the standard coordinates given by the orthogonal basis \(  (e_{1}, e_{2})  \) in \(  T_{\underline z_{x}}\mathcal B_{x}  \) and \(  T_{\underline z_{y}}\mathcal B_{x}  \) and consider 
an unstable vector  
\[
  \underline v^{u}_{x} \in \tilde K^{u}\subset T_{\underline z_{x}}\mathcal B_{x}
  \] 
which we assume is normalized, so that \(   \|\underline v^{u}_{x}\| = 1 \), and which we write as
\[
 \underline v^{u}_{x} = \underline v^{u}_{1,x}e_{1}+ \underline v^{u}_{2,x}e_{2}.
 \] 
Then let 
\[
\underline v_{y}^{u}:=  D_{\underline z_{x}} (\Psi_{y}^{-1}\circ \Psi_{x})(\underline v_{x}^{u}) = 
\underline v_{1,y}^{u}e_{1}+\underline v_{2, y}^{u}e_{2}.
\]
We will estimate the absolute values of \( \underline v_{1,y}^{u}, \underline v_{2,y}^{u}\) in order to prove \eqref{eq:overlap}.
Consider unit vectors 
 $e_x^u,e_x^s \in T_x^M$ and $e_y^u,e_y^s\in T_y^M$,
in the directions given by the hyperbolic splitting.
Throughout this section, we write  \(  d_{x,y}:= d(x,y)  \), \( u_{x}=u(x), s_{x}=s(x) \) to make our computations more compact and easier to read.
Observe that
\[
\Psi_{y}^{-1}\circ \Psi_{x} = L_{y}^{-1}\circ \exp_{y}^{-1}\circ \exp_{x}\circ L_{x}.
\]
Use $\{e_x^u, e_x^s\}$ as a basis for each tangent space to $T_xM$ in the obvious way, and similarly for $T_y^M$.  With respect to these bases and the standard basis $\{e_1,e_2\}$, the derivatives of the maps $L_y^{-1}$, $\exp_y^{-1}\circ \exp_x$, and $L_x$ are represented by the matrices
\[
\begin{pmatrix} u_y & 0 \\ 0 & s_y \end{pmatrix},\quad
\begin{pmatrix} \xi_1^u & \xi_1^s \\ \xi_2^u & \xi_2^s \end{pmatrix},\quad
\begin{pmatrix} u_x^{-1} & 0 \\ 0 & s_x^{-1} \end{pmatrix},
\]
respectively, where $\xi_{1/2}^{s/u} \in \RR$ are determined by
\begin{equation}\label{eq:newcoords}
\begin{aligned}
D_{L_{x}(\underline z_{x})}( \exp_{y}^{-1}\circ \exp_{x}) e^{u}_{x} &=\xi_{1}^{u}e^{u}_{y}+\xi_{2}^{u}e^{s}_{y}, \\
D_{L_{x}(\underline z_{x})}( \exp_{y}^{-1}\circ \exp_{x})e^{s}_{x} &=\xi_{1}^{s}e^{u}_{y}+\xi_{2}^{s}e^{s}_{y}.
\end{aligned}
\end{equation}
Thus 
$D_{\underline z_{x}} (\Psi_{y}^{-1}\circ \Psi_{x})$ has matrix (with respect to $\{e_1,e_2\}$) given by the product of these matrices, which is
\[
\begin{pmatrix} u_y & 0 \\ 0 & s_y \end{pmatrix}
\begin{pmatrix} u_x^{-1} \xi_1^u & s_x^{-1} \xi_1^s \\ u_x^{-1} \xi_2^u & s_x^{-1} \xi_2^s\end{pmatrix}
= \begin{pmatrix} u_y u_x^{-1} \xi_1^u & u_ys_x^{-1} \xi_1^s \\
s_y u_x^{-1} \xi_2^u & s_y s_x^{-1} \xi_2^s \end{pmatrix},
\]
and we conclude that
\begin{align}
\label{eq:final-1}
\underline v_{1,y}^{u} &=
u_{y}u_{x}^{-1} \xi_1^u \underline v_{1,x}^{u}
+ u_{y}s_{x}^{-1} \xi_1^s \underline  v_{2, x}^{u},
\\ 
\label{eq:final-2}
\underline v_{2, y}^{u} &=
s_{y} u_{x}^{-1} \xi_2^u \underline v_{1,x}^{u}
+ s_{y} s_{x}^{-1} \xi_2^s \underline v_{2, x}^{u}.
\end{align}
We now collect the various estimates which we will  plug into these equations to estimate the norms of   \( \underline v_{1,y}^{u}, \underline v_{2,y}^{u}\). 

\begin{Lemma} There exists a constant \( Q_{14} >0 \) such that for every $\ell\in \NN$ and \(  x, y \in \Lambda_{\ell}  \) satisfying \eqref{eqn:overlapping}, we have
\begin{gather}
\label{eq:vux}
  |\underline v^{u}_{1,x}| \geq 1/\sqrt 2 \geq 1/2 
  \quad \text{ and } \quad 
   |\underline v^{u}_{2,x}| \leq e^{-\lambda}\omega    |\underline v^{u}_{1,x}|, \\
\label{eq:components}
  |\xi_{2}^{u}|, |\xi_{1}^{s}| \leq Q_{14} e^{6\eps\gamma\ell} d_{x,y}^\beta
 \quad \text{ and } \quad 
 |1-\xi_{1}^{u}|,  |1-\xi_{2}^{s}| \leq Q_{14} e^{6\eps\gamma\ell} d_{x,y}^\beta,
\\
\label{eq:us1}
Q_{14}  e^{-\epsilon \ell} \leq u_{y}s_{x}^{-1}\leq  Q_{14} e^{\epsilon \ell}
\quand
 Q_{14}  e^{-\epsilon \ell} \leq s_{y}u_{x}^{-1} \leq Q_{14} e^{\epsilon \ell} ,
\\
\label{eq:us2}
1 - Q_{14} e^{(\eta-1)\epsilon \ell} d_{x,y}^{\zeta}  \leq {u_{y}}{u_{x}^{-1}}\leq 1 + Q_{14} e^{(\eta-1)\epsilon \ell} d_{x,y}^{\zeta},
\\
\label{eq:us3}
1 - Q_{14} e^{(\eta-1)\epsilon \ell} d_{x,y}^{\zeta}  \leq {s_{y}}{s_{x}^{-1}} \leq 1 + Q_{14} e^{(\eta-1)\epsilon \ell} d_{x,y}^{\zeta}.
\end{gather}
\end{Lemma}

\begin{proof}
Equation \eqref{eq:vux} follows immediately from the fact that  \(  \underline v^{u}_{x}  \) is a unit vector and \( \underline v^{u}_{x} \in \tilde K^{u} \). Equation \eqref{eq:components} follows from \eqref{eq:newcoords} 
and   Proposition \ref{prop:holder-on-regular} which gives quantitative control on the H\"older dependence of the stable and unstable directions on the base point in \(  \Lambda_{\ell}  \). 
Equation \eqref{eq:us1} follows immediately from  \eqref{eq:subound}. 
Finally, by \eqref{eq:subound} and  Proposition \ref{prop:su-holder}, 
\[
\frac {u_{y}}{u_{x}} = 1+ \frac{u_{y}-u_{x}}{u_{x}}\geq 1 - \frac{|u_{y}-u_{x}|}{ u_{x}} \geq 1 - 
\frac{Q_8 e^{\eps\eta\ell}d_{x,y}^\zeta}{Q_0\wQ e^{\eps\ell}},
\]
which gives the first half of \eqref{eq:us2}.  The upper bound and \eqref{eq:us3} are similar.
\end{proof}

We are now ready to start estimating the two components \(  \underline v_{1,y}^{u}, \underline v_{2,y}^{u}  \) of \(  \underline v_{y}^{u}  \). We estimate each one separately. Once we have proved Lemmas \ref{lem:4terms1} and \ref{lem:4terms2}, we will be in a position to give the conditions on $\delta>0$, which we stress will be independent of $\ell$.

\begin{Lemma} 
\label{lem:4terms1}
For every $\ell\in \NN$ and $x,y\in \Lambda_\ell$ satisfying \eqref{eqn:overlapping}, we have
\[
|\underline v_{1,y}^u| \geq (1-Q_{14} \delta^\beta)(1-Q_{14}\delta^\zeta)|\underline v_{1,x}^u| - Q_{14}^2 \delta^\beta.
\]
\end{Lemma} 
\begin{proof}
Using \eqref{eq:components} and \eqref{eq:us2}, we have the following estimate for the first term of \eqref{eq:final-1}:
\[
|u_{y}u_{x}^{-1} \underline  v_{1,x}^{u}\xi_{1}^{u}|  \geq  (1-Q_{14}
e^{6\eps\gamma\ell} d_{x,y}^\beta)(1 - Q_{14} e^{(\eta-1)\epsilon \ell} d_{x,y}^{\zeta})|\underline  v_{1,x}^{u}|.
\]
Now \eqref{eqn:overlapping} and the bounds on $\eps$ in \eqref{eqn:eps} give 
\begin{equation}\label{eqn:use-eps1}
\begin{aligned}
e^{6\eps\gamma\ell}d_{x,y}^\beta
&\leq e^{6\eps\gamma\ell} \delta^\beta e^{-\beta\lambda\ell}
= \delta^\beta e^{(6\eps\gamma - \beta\lambda)\ell} \leq \delta^\beta, \\
e^{(\eta-1)\eps\ell}d_{x,y}^\zeta
&\leq e^{(\eta-1)\eps\ell}\delta^\zeta e^{-\zeta\lambda\ell}
= \delta^\zeta e^{((\eta-1)\eps - \zeta\lambda)\ell} \leq \delta^\zeta,
\end{aligned}
\end{equation}
and thus
\begin{equation}\label{eq:4terms1}
|u_{y}u_{x}^{-1} \xi_{1}^{u} \underline  v_{1,x}^{u}| \geq (1-Q_{14}\delta^\beta)(1-Q_{14}\delta^\zeta)|\underline v_{1,x}^u|.
\end{equation}
For the second term of \eqref{eq:final-1}, by \eqref{eq:components} and \eqref{eq:us1},  and using  \(  |\underline v_{2,x}^{u}|\leq 1  \), we have 
\[
 |u_{y}s_{x}^{-1} \xi_1^s \underline v_{2,x}^{u}| \leq Q_{14}^2 e^{\epsilon \ell} e^{6\eps\gamma\ell} d_{x,y}^\beta
\leq Q_{14}^2 e^{7\eps\gamma\ell} d_{x,y}^\beta,
\]
where the second inequality uses the fact that $\gamma\geq 1$.
As in \eqref{eqn:use-eps1} above, 
\begin{equation}\label{eqn:use-eps2}
e^{7\eps\gamma\ell} d_{x,y}^\beta
\leq \delta^\beta e^{(7\eps\gamma - \beta\lambda)\ell} \leq \delta^\beta,
\end{equation}
and thus $|u_{y}s_{x}^{-1} \xi_1^s \underline v_{2,x}^{u}| \leq Q_{14}^2 \delta^\beta$.  Subtracting this from \eqref{eq:4terms1} and recalling \eqref{eq:final-1} proves the lemma.
\end{proof}

\begin{Lemma}
\label{lem:4terms2} 
For every $\ell\in \NN$ and $x,y\in \Lambda_\ell$ satisfying \eqref{eqn:overlapping}, we have
\[
|\underline v_{2,y}^u| \leq (1+Q_{14}\delta^\beta)(1+Q_{14}\delta^\zeta)|\underline v_{2,x}^u| + Q_{14}^2 \delta^\beta.
\]
\end{Lemma}
\begin{proof}
The proof is nearly identical to the proof of the previous Lemma. Let us
bound the first term in \eqref{eq:final-2} as follows:
\begin{align*}
|s_{y}s_{x}^{-1}\xi_{2}^{s} \underline v_{2,x}^{u}| &\leq (1 + Q_{14} e^{(\eta-1)\epsilon \ell} d_{x,y}^{\zeta}) (1+ Q_{14} e^{6\eps\gamma\ell} d_{x,y}^\beta)   |\underline v_{2,x}^{u}| \\
&\leq (1+Q_{14}\delta^\zeta)(1+Q_{14}\delta^\beta)|\underline v_{2,x}^u|.
\end{align*}
Similar computations for the second term of \eqref{eq:final-2} give
\[
|s_{y}u_{x}^{-1}\xi_{2}^{u} \underline v_{1,x}^{u}| 
\leq  Q_{14}^2 e^{\eps\ell} e^{6\eps\gamma\ell} d_{x,y}^\beta
\leq Q_{14}^2 \delta^\beta.
\]
Adding these estimates together proves the lemma.
\end{proof}

\begin{proof}[Proof of Theorem \ref{thm:overlapping} (derivative estimates)]
We can now prove the first part of Theorem \ref{thm:overlapping} concerning properties \eqref{eq:overlap} and \eqref{eq:overlaps} in the definition of overlapping charts. 
Start by requiring that $\delta>0$ is sufficiently small that $Q_{14} \delta^\beta<1$ and $Q_{14} \delta^\zeta < 1$.  (Further conditions will come later.)
Lemmas \ref{lem:4terms1} and \ref{lem:4terms2} give 
\begin{align*}
\frac{ |\underline v^{u}_{2,y}|}{   |\underline v_{1,y}^{u}|} &\leq  
\frac{(1+Q_{14}\delta^\beta)(1+Q_{14}\delta^\zeta)|\underline v_{2,x}^u| + Q_{14}^2\delta^\beta}
{(1-Q_{14}\delta^\beta)(1-Q_{14}\delta^\zeta)|\underline v_{1,x}^u| - Q_{14}^2\delta^\beta} \\
&\leq  
\frac{(1+Q_{14}\delta^\beta)(1+Q_{14}\delta^\zeta)e^{-\lambda}\omega|\underline v_{1,x}^u| + Q_{14}^2\delta^\beta}
{(1-Q_{14}\delta^\beta)(1-Q_{14}\delta^\zeta)|\underline v_{1,x}^u| - Q_{14}^2\delta^\beta},
 \end{align*}
where the second inequality uses the fact that $\underline v_x^u\in \tilde K^u$.  Note that the function $t\mapsto \frac{at + b}{ct + d}$ is decreasing in $t$ when $ad-bc<0$, which is the case for the expression above, and thus we can obtain an upper bound by observing that \eqref{eq:vux} gives $|\underline v_{1,x}^u| \geq \frac 12$, so monotonicity gives
\begin{equation}\label{eqn:Q9}
\frac{ |\underline v^{u}_{2,y}|}{   |\underline v_{1,y}^{u}|} \leq  
\frac{(1+Q_{14}\delta^\beta)(1+Q_{14}\delta^\zeta)e^{-\lambda}\omega/2 + Q_{14}^2\delta^\beta}
{(1-Q_{14}\delta^\beta)(1-Q_{14}\delta^\zeta)/2 - Q_{14}^2\delta^\beta}.
\end{equation}
For sufficiently small $\delta>0$, the right-hand side is $<\omega$, which implies that \( \underline v_{y}^{u}  \in K^{u}\)
as required by the first part of \eqref{eq:overlap}.   

For the second part of \eqref{eq:overlap}, we observe that
\[
\frac{\|\underline v_y^u\|}{\|\underline v_x^u\|}
\geq \frac{|\underline v_{1,y}^u|}{\sqrt{|\underline v_{1,x}^u|^2 + |\underline v_{2,x}^u|^2}}
\geq \frac{\big((1-Q_{14}\delta^\beta)(1-Q_{14}\delta^\zeta) - Q_{14}^2\delta^\beta\big) |\underline v_{1,x}^u|}
{\sqrt{1+e^{-2\lambda}\omega^2} |\underline v_{1,x}^u|},
\]
where the second inequality uses Lemma \ref{lem:4terms1} and \eqref{eq:vux} for the numerator and the fact that $\underline v_x^u\in \tilde K^u$ for the denominator.  Recall from \eqref{eqn:omega} that $1/\sqrt{1+e^{-2\lambda}\omega^2} > e^{-\lambda/24}$; thus we can choose $\delta$ small enough that the right-hand side of the above expression is $> e^{-\lambda/24}$, which proves the second half of \eqref{eq:overlap}.
Condition \eqref{eq:overlaps} follows by analogous arguments. 
\end{proof}

\subsection{Overlapping stable and unstable strips}
\label{sec:overlappingstrips}

We now complete the proof of Theorem  \ref{thm:overlapping} by showing that if 
$\gamma \subset \widetilde{\mathcal{N}}_{x, \ell}^{s}$ is a full length strongly stable admissible curve, then it completely crosses $\widetilde{\mathcal N}_{y,\ell}^u$.  The other three required conditions obtained by interchanging $x/y$ and stable/unstable are proved analogously.

We start with a couple of simple Lemmas relating the distance \( d(x,y) \) between two points and the amount of overlap of their regular neighbourhoods. 

\begin{Lemma}\label{lem:overlapping}
There exists a constant \( Q_{15} > 0 \) such that for every \( \ell \geq 1\) and every \( x,y \in \Lambda_{\ell}\) satisfying $d(x,y) \leq Q_{15} e^{-2\epsilon \ell} b_{\ell}$, we have  $x,y\in \mathcal{N}_x^{(\ell)} \cap \mathcal{N}_y^{(\ell)}$.
\end{Lemma}
\begin{proof}
Recall that the regular neighbourhood $\mathcal{N}_x^{(\ell)}$ is given by $\mathcal{N}_x^{(\ell)} = \Psi_x(\mathcal{B}_x^{(\ell)}) = \exp_x(L_x(\mathcal{B}_x^{(\ell)}))$, where $\mathcal{B}_x^{(\ell)} = [-b_\ell,b_\ell]^2 \subset \RR^2$ and $L_x$ is given by \eqref{def:L}. In particular, $L_x(\mathcal{B}_x^{(\ell)})$ is a parallelogram in $T_x M$ centered at $0$ with side lengths $2b_\ell / u(x)$ and $2b_\ell / s(x)$. Using the upper bounds in \eqref{eq:subound} we see that both of these side lengths are at least $2 (Q_0 \wQ)^{-1} e^{-\eps \ell} b_\ell$. Moreover, by condition \eqref{eq:angle} and the definition of \( \Lambda_{\ell} \) in \eqref{lambdal} we have \( \measuredangle (E^{s}_{x}, E^{u}_{x})\geq e^{-\epsilon\ell} \).

Consider a parallelogram whose sides have length $A,B$ and meet at an angle $\theta \in (0,\pi/2]$. The distance between each pair of opposite edges is at least $\min(A,B) \cdot \sin \theta \geq \min(A,B) \cdot 2\theta/\pi$, and thus the parallelogram contains a ball of radius $\min(A,B) \theta/\pi$ around its center. In the setting of the previous paragraph, this shows that $L_x(\mathcal{B}_x^{(\ell)})$ contains a ball in $T_xM$ centered at the origin with radius $2(Q_0\wQ)^{-1} \pi^{-1} e^{-2\eps\ell} b_\ell$. 

Choose $r>0$ such that given any $x\in M$ and $v\in T_xM$ with $\|v\| \geq r$, we have $d(x, \exp_x(v)) \geq r/2$. Then choosing $Q_{15} >0$ sufficiently small that $2 Q_{15} b_0 <r$ and $Q_{15} \leq (Q_0\wQ)^{-1} \pi^{-1}$, we see that $L_x(\mathcal{B}_x^{(\ell)})$ contains a ball in $T_xM$ centered at the origin with radius $Q_{15} 2 Q_{15} e^{-2\eps\ell} b_\ell$, and since $e^{-2\eps\ell} b_\ell \leq b_0$, the image of this ball under $\exp_x$ contains $B(x, Q_{15} e^{-2\eps\ell} b_\ell)$. This shows that $B(x, Q_{15} e^{-2\eps\ell} b_\ell) \subset \exp_x(L_x(\mathcal{B}_x^{(\ell)})) = \Psi_x(\mathcal{B}_x^{(\ell)}) = \mathcal{N}_x^{(\ell)}$, which completes the proof.
\end{proof}

\begin{Lemma} \label{lem:overlapping1}
There exists  \( Q_{16} > 0 \) such that for  \( \ell \geq 1\) and  \( x,y \in \Lambda_{\ell}\),  
\[
\text{ if } y \in \mathcal N_{x}^{(\ell)} 
\quad \text{ then  } \quad 
\|\Psi_{x}^{-1}\circ \Psi_{y}(0)\|\leq Q_{16} e^{2\epsilon \ell} d(x,y). 
\]
\end{Lemma} 

\begin{proof}
If \( y \in \mathcal N_{x}^{(\ell)}\) then 
\(
 \Psi_{x}^{-1}\circ \Psi_{y}(0) = \Psi_{x}^{-1}(y) 
 \) 
 is well defined. Therefore 
 \(
 \| \Psi_{x}^{-1}\circ \Psi_{y}(0)\| = \|\Psi_{x}^{-1}(y) \| = \|\Psi_{x}^{-1}(y) - 0\| = \|\Psi_{x}^{-1}(y) - \Psi_{x}^{-1}(x) \|
 \)
and so we just need to estimate the Lipschitz constant of \( \Psi_{x}^{-1}\). By definition we have \( \Psi_{x}^{-1}= L_{x}^{-1}\circ \exp_{x}^{-1} \) and  the result follows using \eqref{eq:Pesin1} in Lemma \ref{lem:Lbound} and the fact that \( \exp_{x}^{-1} \) is close to an isometry. 
\end{proof}

\begin{figure}[tbp]
\includegraphics[width=\textwidth]{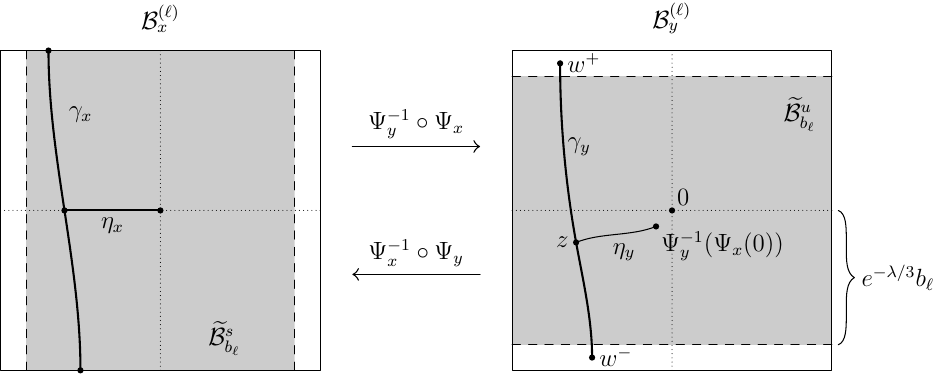}
\caption{Proving Theorem \ref{thm:overlapping}.}
\label{fig:crossing}
\end{figure}

\begin{proof}[Proof of Theorem \ref{thm:overlapping} (stable and unstable strips)]
We can now complete the proof of Theorem  \ref{thm:overlapping}. 
Let \( x, y \in \Lambda_{\ell}\) with $d(x,y)\leq \delta e^{-\lambda \ell}$ as in \eqref{eqn:overlapping}. 
Then since $b_\ell = C e^{-2\eps\ell/\alpha}$ for some constant $C>0$, we have
\begin{equation}\label{eqn:small-in-bell}
d(x,y) e^{2\eps\ell} b_\ell^{-1} 
\leq \delta e^{-\lambda\ell} e^{2\eps\ell} C^{-1} e^{2\eps\ell/\alpha}
= \delta C^{-1} e^{(2\eps(1+\frac 1\alpha) - \lambda)\ell}
\leq \delta C^{-1},
\end{equation}
where the last inequality uses \eqref{eqn:eps}.  By making $\delta$ sufficiently small that $\delta C^{-1} \leq Q_{15}$, we guarantee that the hypothesis of Lemma \ref{lem:overlapping} is satisfied.

Now let $\gamma$ be a full length strongly stable admissible curve in 
$\widetilde{\mathcal{N}}_{x,\ell}^s$, and consider the curves $\gamma_x = \Psi_x^{-1}(\gamma) \subset \mathcal{\widetilde B}_{b_\ell}^s$ and $\gamma_y = \Psi_y^{-1}(\gamma) \subset \mathcal{B}_{y}^{(\ell)}$, as shown in Figure \ref{fig:crossing}.
 Observe that since $\widetilde{\mathcal B}_r^s := [-e^{-\lambda/3}r,e^{-\lambda/3}r]\times [-r,r]$ (see Definition \ref{def:tilde-BN}), $\gamma_x$ intersects the $x$-axis at the point $(t,0)$ for some $|t| \leq e^{-\lambda/3} b_\ell$.
Let $\eta_x$ be the segment of the $x$-axis in $\mathcal{B}_x^{(\ell)}$ that connects this point $(t,0)$ to the origin; the length of $\eta_x$ is at most $e^{-\lambda/3} b_\ell$.  Let $\eta_y = \Psi_y^{-1}\circ \Psi_x(\eta_x)$.  Let $z\in \mathcal{B}_y^{(\ell)}$ be the intersection point of $\eta_y$ and $\gamma_y$, and let $w^{\pm}$ be the endpoints of $\gamma_y$.  Writing $z=z_1 e_1 + z_2 e_2$ and similarly for $w^\pm$, our goal is to show that $|w^\pm_2| \geq e^{-\lambda/3} b_\ell$.

We give the proof for $w^+$; the proof for $w^-$ is similar.  
By Lemma \ref{lem:overlapping1} and \eqref{eqn:small-in-bell}, the point $v := \Psi_y^{-1}(x) = \Psi_y^{-1}(\Psi_x(0))$ has
\begin{equation}\label{eqn:v-small}
\|v\| \leq Q_{16} e^{2\eps\ell} d(x,y) \leq Q_{16} C^{-1} \delta b_\ell.
\end{equation}
By \eqref{eq:overlap}, $\eta_y$ is an unstable admissible curve connecting $z$ and $v$ with length $\leq e^{\lambda/24} e^{-\lambda/3}b_\ell$ (using the fact that $\eta_x$ has length at most $e^{-\lambda/3}b_\ell$).  Since $\gamma_y$ is a stable admissible curve, we have
\[
|w_1^+ - z_1| \leq 2\omega b_\ell,
\]
and thus
\[
|w_1^+| \leq |v_1| + |z_1 - v_1| + |w_1^+ - z_1|
\leq Q_{16} C^{-1} \delta b_\ell + e^{\lambda/24} e^{-\lambda/3} b_\ell + 2\omega b_\ell.
\]
By \eqref{eqn:omega} we can choose $\delta>0$ small enough that 
\begin{equation}\label{eqn:Q-delta}
Q_{16} C^{-1} \delta + e^{\lambda/24} e^{-\lambda/3} + 2\omega < 1,
\end{equation}
and we conclude that $|w_1^+| < b_\ell$, so $w^+$ is not on a vertical boundary of $\mathcal{B}_y^{(\ell)}$.

If $w^+$ is on the top boundary of $\mathcal{B}_y^{(\ell)}$, then there is nothing to prove, so we can assume that the part of $\gamma_y$ running from $z$ to $w^+$ is the image of the top half of $\gamma_x$ under the transition map.  By \eqref{eq:overlap}, this part of $\gamma_y$ is a stable admissible curve with length $\geq e^{-\lambda/24} b_\ell$ (using the fact that the top half of $\gamma_x$ has length at least $b_\ell$).  Since the length of this part of $\gamma_y$ is at most $\sqrt{1+\omega^2}|w_2^+ -z_2|$, we conclude that
\begin{equation}\label{eqn:w2z2}
|w_2^+ - z_2| \geq \frac{e^{-\lambda/24}b_\ell}{\sqrt{1+\omega^2}}.
\end{equation}
As argued above, $\eta_y$ has length $\leq e^{\lambda/24} b_\ell$; on the other hand since it is a stable admissible curve, its length is at least
\[
\sqrt{|z_1-v_1|^2 + |z_2-v_2|^2} \geq \sqrt{\omega^{-2} + 1} |z_2-v_2|,
\]
and we conclude that
\begin{equation}\label{eqn:z2v2}
|z_2 - v_2| \leq \frac{\omega e^{\lambda/24} b_\ell}{\sqrt{1+\omega^2}}.
\end{equation}
Combining \eqref{eqn:v-small}, \eqref{eqn:w2z2}, and \eqref{eqn:z2v2} gives
\begin{align*}
|w_2^+| &\geq |w_2^+ - z_2| - |z_2-v_2| - |v_2|
\geq \frac{e^{-\lambda/24}b_\ell}{\sqrt{1+\omega^2}}
- \frac{\omega e^{\lambda/24} b_\ell}{\sqrt{1+\omega^2}}
- Q_{16} C^{-1} \delta b_\ell \\
&= \Big( \frac{e^{-\lambda/24} - \omega e^{\lambda/24}}{\sqrt{1+\omega^2}} - Q_{16} C^{-1}\delta \Big) b_\ell
\geq (e^{-\lambda/4} - Q_{16} C^{-1}\delta) b_\ell,
\end{align*}
where the last inequality uses \eqref{eqn:omega}.  As long as $\delta$ is small enough that
\begin{equation}\label{eqn:Q20Q9}
e^{-\lambda/4} - Q_{16} C^{-1}\delta > e^{-\lambda/3},
\end{equation}
this gives $|w_2^+| \geq e^{-\lambda/3} b_\ell$, completing the proof.
\end{proof}

\section{Pseudo-orbits, branches, shadowing: Proofs of  Theorems \ref{thm:pseudo-orbit} and \ref{thm:inf-po}}

We are now ready to prove Theorems \ref{thm:pseudo-orbit} and \ref{thm:inf-po}.
\subsection{Regular branches: Proof of Theorem \ref{thm:pseudo-orbit}}\label{sec:prove-po}

The two fundamental ingredients in the proof of Theorem \ref{thm:pseudo-orbit} are  Theorem \ref{thm:overlapping}  and Theorem \ref{thm:existsbranch}, which is essentially the special case of Theorem \ref{thm:pseudo-orbit} where $k=1$ and the pseudo-orbit is in fact a real orbit. 

Fix \( \delta>0 \) sufficiently small so that the conclusions of Theorem \ref{thm:overlapping} hold.  
To prove the first part of Theorem \ref{thm:pseudo-orbit}, we observe that if 
$\barx = (x_0,\dots, x_k)$ is an $(\bar{\ell},\delta,\lambda)$-pseudo-orbit, then by Theorem \ref{thm:overlapping}, the Lyapunov charts $\mathcal{N}_{x_j}^{(\ell_j)}$ and $\mathcal{N}_{f(x_{j-1})}^{(\ell_j)}$ are overlapping for every $1\leq j\leq k$.  
By Theorem \ref{thm:existsbranch},
for each  $0\leq j< k$,
the sets
\begin{align*}
\mathcal{B}_{x_j,(\ell_j,\ell_{j+1})}^{s,1} &= \Psi_{x_{j}}^{-1} (\mathcal{N}_{x_j}^{(\ell_{j})} \cap f^{-1} \mathcal{N}_{f(x_j)}^{\ell_{j+1}}), \\
\mathcal{B}_{f(x_j),(\ell_{j},\ell_{j+1})}^{u,1} &= \Psi_{f(x_j)}^{-1}(f(\mathcal{N}_{x_j}^{(\ell_j)}) \cap \mathcal{N}_{f(x_j)}^{(\ell_{j+1})})
\end{align*}
from \eqref{eq:Bsuell}
are strongly stable and strongly unstable strips in $\widetilde{\mathcal{B}}_{b_{\ell_{j}}}^s$ and
$\widetilde{\mathcal{B}}_{b_{\ell_{j+1}}}^u$, respectively, and $f_{x_{j}}$
is a diffeomorphism between them that satisfies the inclusions and estimates in \eqref{eqn:cones-0}.
Since the Lyapunov charts $\mathcal{N}_{x_{j+1}}^{(\ell_{j+1})}$ and $\mathcal{N}_{f(x_j)}^{(\ell_{j+1})}$ are overlapping, we conclude that $\Psi_{x_{j+1}}^{-1} \circ \Psi_{f(x_{j})} (\mathcal{B}_{f(x_j),(\ell_j,\ell_{j+1})}^{u,1})$ is an unstable strip in $\mathcal{B}_{x_{j+1}}^{(\ell_{j+1})}$, and thus its preimage under $\Psi_{x_{j+1}}^{-1} \circ f \circ \Psi_{x_j}$ is a stable strip in $\mathcal{B}_{x_j}^{(\ell_{j})}$. Thus we have
\[
\mathcal{B}_{x_0}^{(\ell_0)} \xrightarrow{\Psi_{x_0}^{-1} \circ f \circ \Psi_{x_1}} \mathcal{B}_{x_1}^{(\ell_1)}
\xrightarrow{\Psi_{x_1}^{-1} \circ f \circ \Psi_{x_2}}
\mathcal{B}_{x_2}^{(\ell_2)}
\to \cdots \to
\mathcal{B}_{x_{k-1}}^{(\ell_{k-1})}
\xrightarrow{\Psi_{x_{k-1}}^{-1} \circ f \circ \Psi_{x_k}}
\mathcal{B}_{x_k}^{(\ell_k)}
\]
where the maps are not defined on the entirety of the indicated domain, but only on a stable strip, and the corresponding image is an unstable strip.  In particular, taking the composition of all the maps we see that \eqref{eqn:Bjx} with $j=0$ defines a stable strip $\mathcal{B}_{\barx}^0 \subset \mathcal{B}_{x_0}^{(\ell_0)}$ that is mapped to an unstable strip $\mathcal{B}_{\barx} \subset \mathcal{B}_{x_k}^{(\ell_k)}$ by $\Psi_{x_k}^{-1} \circ f^k \circ \Psi_{x_0}$.  This proves the first property in the definition of an $\bar\ell$-regular branch.  For the second, we observe
that by \eqref{eqn:cones-0}, each $f_{x_{j-1}} = \Psi_{f(x_{j-1})}^{-1} \circ f \circ \Psi_{x_{j-1}}$ has a derivative that maps $K^u$ into $\widetilde K^u$, and that the transition map $\Psi_{x_j}^{-1} \circ \Psi_{f(x_{j-1})}$ maps this into $K^u$ by the definition of overlapping charts; moreover, the first map above expands each vector in $K^u$ by a factor of at least $e^{\lambda/2}$, and so after composing with the transition map, the derivative of $\Psi_{x_j}^{-1} \circ f \circ \Psi_{x_{j-1}}$ expands each vector in $K^u$ by a factor of at least $e^{\lambda/2} e^{-\lambda/24}$.  Iterating completes the proof of Theorem \ref{thm:pseudo-orbit}.

\subsection{Shadowing: Proof of Theorem \ref{thm:inf-po}}\label{sec:inf-po}

We prove Theorem \ref{thm:inf-po} using Theorem \ref{thm:pseudo-orbit} and the definition of regular branch in Definition \ref{def:reg-branch}, together with the hyperbolicity estimates from  Proposition \ref{prop:hyp-branch}.
Start by choosing the constants as in the assumption of Theorem \ref{thm:inf-po}.  

First we prove that $V^s := \bigcap_{n\geq 0} f^{-n}(\mathcal{N}_{x_n}^{(\ell_n)})$ is a $C^{1+\text{H\"older}}$ full length local $(\wQ e^{2\eps\ell_0},\lambda/3)$-stable curve.
Given $n\in\NN$, let $\barl^{n,+} := (\ell_0, \dots, \ell_n)$ and $\barx^{n,+} := (x_0,\dots, x_n)$, so that $\barx^{n,+}$ is an $(\barl^{n,+},\delta,\lambda)$-pseudo-orbit, which by Theorem \ref{thm:pseudo-orbit} determines an $\barl^{n,+}$-regular branch.  Let $\mathcal{B}_\barx^{0}(n) \subset \mathcal{B}_{x_0}^{(\ell_0)}$ be the corresponding stable strip.  

The boundary curves of the strips $\{\mathcal{B}_{\barx}^0(n)\}_{n\in\mathbb{N}}$ can be represented as the graphs of a uniformly Lipschitz family of functions $\spn{e_2} \to \spn{e_1}$ (recall Figure \ref{fig:lyap-chart} and Definitions \ref{def:conesize}--\ref{def:strips}). This family contains a sequence that converges uniformly to a Lipschitz functions whose graph is a continuous curve $\tilde V^s \subset \bigcap_{n\geq 0} \mathcal{B}_{\barx}^0(n) \subset \mathcal{B}_x$ with endpoints on $[-b(x),b(x)] \times \{-b(x),b(x)\} \subset \mathcal{B}_x$. Moreover, the expansion estimates in Definition \ref{def:reg-branch} imply that any full length unstable admissible curve in $\mathcal{B}_x$ intersects $\mathcal{B}_\barx^0(n)$ in a curve of length $\leq 2b_0 e^{-\lambda n/3}$, and thus intersects $\bigcap_{n\geq 0} \mathcal{B}_{\barx}^0(n)$ in a single point. It follows that $\tilde V^s = \bigcap_{n\geq 0} \mathcal{B}_{\barx}^0(n)$, and thus
$V^s = \bigcap_{n\geq 0} f^{-n} (\mathcal{N}_{x_n}^{(\ell_n)}) = \Psi_x(\bigcap_{n\geq 0} \mathcal{B}_{\barx}^0(n))$ is a continuous curve satisfying the geometric conditions in Definition \ref{def:localman}; moreover, $V^s$ is full length in $\mathcal{N}_x$. To prove that it is a local $(\wQ e^{2\eps\ell_0},\lambda/3)$-stable curve, we need to prove that it is $C^1$ (in fact we will prove that it is $C^{1+\text{H\"older}}$) and that for every $y,z\in V^s$ and $j\geq 0$ we have
\begin{equation}\label{eqn:V-contract}
d(f^j(y),f^j(z)) \leq C e^{-\lambda j/3} d(y,z)
\quad\text{for }C = \wQ e^{2\eps\ell_0}.
\end{equation}

We start by proving \eqref{eqn:V-contract}. Given $n\geq 0$, let $W_n$ be an arbitrary full length stable admissible curve in the unstable strip $f^n(\Psi_{x_n}(\mathcal{B}_{\barx}^0(n))$, and let $V_n = f^{-n}(W_n)$. Then for each $y,z\in V^s$ there are $y_n,z_n \in V_n$ such that $y_n \to y$ and $z_n\to z$. Moreover, \eqref{eqn:vjsvks} from Proposition \ref{prop:hyp-branch} gives $d(f^j(y_n),f^j(z_n)) \leq C e^{-\lambda j/3} d(y_n,z_n)$ for all $0\leq j\leq n$. Thus for every $j\geq 0$ we can send $n\to\infty$ and deduce that \eqref{eqn:V-contract} holds.

It remains to show that $V^s$ is $C^{1+\text{H\"older}}$.
To this end, let $x\in V^s$ be arbitrary and choose a unit tangent vector $v_n \in (D f_x^n)^{-1} K_{x_n, f^n(x)}^u \subset T_x M$, where the cones are as defined in \eqref{eqn:M-cones}. Then \eqref{eqn:vjsvks} from Proposition \ref{prop:hyp-branch} gives
\begin{equation}\label{eqn:vn-contracts}
\|Df_x^j v_n \| \leq C e^{-\lambda j/3}
\end{equation}
for all $0\leq j\leq n$. By compactness of the unit tangent space there is a subsequence $v_{n_k} \to e_x^s \in T_xM$.  For every $j\in \NN$, \eqref{eqn:vn-contracts} holds for all $v_{n_k}$ with $n_k \geq j$, and thus it holds for $e_x^s$ as well:
\begin{equation}\label{eqn:exs-contracts}
\|Df_x^j e_x^s\| \leq C e^{-\lambda j/3} \text{ for all } j\geq 0.
\end{equation}
We will show that $x\mapsto e_x^s$ is H\"older continuous and that $e_x^s$ is tangent to $V^s$ at $x$; this will complete the proof that $V^s$ is a local $(C, \lambda/3)$-stable curve.

Let $e_x^u$ be an arbitrary unit tangent vector in $K_{x_0,x}^u$. From \eqref{eqn:vjuvku} in Proposition \ref{prop:hyp-branch}, we have
\[
\|Df_x^j e_x^u\| \geq \wQ^{-1} e^{-2\eps\ell_j} e^{\lambda j/3}
\]
for all $j\geq 0$. Observe that
\begin{equation}\label{eqn:ellj-0}
-2\eps\ell_j + \frac\lambda 3 j \geq -2\eps\ell_0 - 2\eps j + \frac\lambda 3 j \geq -2\eps\ell_0 + \frac \lambda 4 j,
\end{equation}
where the last inequality uses the fact that $\eps < \frac \lambda 3 - \frac\lambda 4 = \frac\lambda{12}$, which follows from \eqref{eqn:eps}. We conclude that
\begin{equation}\label{eqn:exu-expands}
\|Df_x^j e_x^u\| \geq \wQ^{-1} e^{-2\eps \ell_0} e^{\lambda j/4} = C^{-1} e^{\lambda j/4}
\text{ for all } j\geq 0.
\end{equation}
Using \eqref{eqn:exs-contracts} and \eqref{eqn:exu-expands}, we can apply Proposition \ref{prop:theta-holder} and conclude that $x\mapsto e_x^s$ is H\"older continuous.

It remains to show that $e_x^s$ is tangent to $V^s$ at $x$. Equivalently, we must show that for any $\omega>0$, the cone
\[
K_\omega := \{0\} \cup \{ a e_x^u + b e_x^s : |a| < \omega |b|\} \subset T_x M
\]
has the following property: there is a neighbourhood $U \subset T_x M$ containing the origin such that $(\exp_x|_U)^{-1}(V^s)$ lies in $K_\omega$. Once we verify this, we can conclude that $V^s$ is differentiable at $x$ and that $e_x^s$ is in its tangent space. To find the neighbourhood $U$, we first observe that given $\omega>0$, every $v\in T_x M \setminus K_\omega$ has the form $v=a e_x^u + be_x^s$ for some $|a| \geq \omega |b|$ with $|a| > 0$.
Thus for all $j\geq 0$ we have
\[
\|Df_x^j v\| \geq |a| \|Df_x^j e_x^u\| - |b| \|Df_x^j e_x^s\|
\geq |a| ( C^{-1} e^{\lambda j/4} - \omega^{-1} C e^{-\lambda j/3}).
\]
Moreover, $\|v\| \leq |a| + |b| \leq |a|(1+\omega^{-1})$, and thus
\begin{equation}\label{eqn:v-expands}
\frac{\|Df_x^j v\|}{\|v\|}
\geq \frac{C^{-1} e^{\lambda j/4} - \omega^{-1} C e^{-\lambda j/3}}{1+\omega^{-1}}.
\end{equation}
Now we fix a choice of $j$ sufficiently large that the right-hand side of \eqref{eqn:v-expands} is at least $3 C e^{-\lambda j/3}$, and let $U \subset T_x M$ be a sufficiently small neighbourhood of $0$ that for all $v\in U$ we have
\[
\frac{d(f^j(\exp_x v), f^j x)}{d(\exp_x v, x)} \geq \frac 12 \frac{\|D_xf^j v\|}{\|v\|}.
\]
Combining this with \eqref{eqn:v-expands} and our choice of $j$, we see that for all $v\in U \setminus K_\omega$ we have
\[
\frac{d(f^j(\exp_x v), f^j x)}{d(\exp_x v, x)} \geq \frac 32 C e^{-\lambda j/3},
\]
and comparing this to \eqref{eqn:V-contract} we conclude that $\exp_x v$ does not lie on the curve $V^s$. It follows that $(\exp_x|_U)^{-1}(V^s) \subset K_\omega$, and since $\omega>0$ was arbitrary, this shows that $V^s$ is differentiable at $x$ and that $e_x^s$ is in its tangent space. This completes the proof of the first part of Theorem \ref{thm:inf-po}. The second part of the theorem follows by an identical argument.

\begin{Remark}\label{rmk:hyp-branches-Vs}
The same proof given above can also be applied to a family of $(C,\kappa)$-hyperbolic branches as in Proposition \ref{prop:intersect-branches}; it suffices to replace $\lambda/3$ and $\lambda/4$ in the above arguments with $\kappa$, and to replace the cone families from \eqref{eqn:M-cones} with the cone families associated to the hyperbolic branches. Thus we have also proved Proposition \ref{prop:intersect-branches}.
\end{Remark}

It remains to prove the third part of Theorem \ref{thm:inf-po}. Existence and uniqueness of the shadowing point follows immediately from the first two parts: indeed, a bi-infinite pseudo-orbit determines a local stable curve $V^s = \bigcap_{n\geq 0} f^{-n}(\mathcal{N}_{x_n}^{(\ell_n)})$ and a local stable curve $V^u = \bigcap_{n\leq 0} f^{-n}(\mathcal{N}_{x_n}^{(\ell_n)})$, and thus $\bigcap_{n\in\ZZ} f^{-n}(\mathcal{N}_{x_n}^{(\ell_n)}) = V^s \cap V^u$. Since the curves determined in the first two parts are both full length, then have a unique point of intersection, which is the unique shadowing point.

Now we establish regularity of $y$.  
Fix $\delta,\lambda>0$ and let $\Lambda'$ be the collection of unique shadowing points for $(\barl,\delta,\lambda)$-pseudo-orbits in $\Lambda$, where $\barl \in \NN^\ZZ$ can be arbitrary. Observe that $\Lambda'$ is invariant since as remarked after the statement of Theorem \ref{thm:inf-po}, if $y$ is the unique shadowing point for $\barx$, then $f(y)$ is the unique shadowing point for $\sigma\barx$. Given $y\in \Lambda'$ we have local stable and unstable curves $V_y^{s/u}$ as in the first two parts of the theorem.  Moreover, these curves satisfy $f(V_y^s) \subset V_{f(y)}^s$ (since shadowing the orbit for $n\geq 0$ implies shadowing for $n\geq 1$) and $f^{-1}(V_y^u) \subset V_{f^{-1}(y)}^u$. Thus the subspaces defined by $E_y^{s/u} = T_y V_z^{s/u}$ produce a measurable $Df$-invariant splitting on $\Lambda'$.

It remains to prove that this invariant splitting satisfies \eqref{eq:slowvar}--\eqref{eq:hypest}. We start by letting $e_y^{s/u}$ be unit vectors in $E_y^{s/u}$ and observing that \eqref{eqn:exs-contracts} and \eqref{eqn:exu-expands} imply that writing $C(y) = \wQ e^{2\eps\ell_0}$, we have
\[
\|Df_y^n e_y^s\| \leq C(y) e^{-\lambda n/3}
\text{ and } \|Df_y^n e_y^u \| \geq C(y)^{-1} e^{\lambda n /4}
\]
for all $n\geq 0$, where we observe that \eqref{eqn:exu-expands} applies to a broader class of unit tangent vectors (anything in the cone $K_{x_0,y}^u$), and in particular applies to our choice of $e_y^u$.  This establishes two of the four inequalities in \eqref{eq:hypest}; the other two follow from the analogues of \eqref{eqn:exs-contracts} and \eqref{eqn:exu-expands} for the backwards pseudo-orbit and corresponding unstable curve.

\begin{Remark}\label{rmk:by-4-or-3}
The fact that \eqref{eqn:exu-expands} contains $\lambda/4$ instead of $\lambda/3$ is due to the fact that we must use \eqref{eqn:ellj-0} to compare the level of regularity at the two ends of the orbit segment. One consequence of this appears later in Proposition \ref{thm:nicehyp} and its proof in \S\ref{sec:sat-rec}, when we build a nice rectangle for which $\lambda/3$ appears in the hyperbolic branch property, but $\lambda/4$ appears in the regularity. Roughly speaking one may say that the hyperbolic branch property controls contraction/expansion from the endpoints to somewhere in the middle of the corresponding orbit segment, but regularity requires us also to control it when going between two points in the middle of the segment, and for this we must weaken $\lambda/3$ to $\lambda/4$.
\end{Remark}

To prove \eqref{eq:slowvar} for $C$, observe that $f(y)$ is the shadowing orbit for $\sigma\barx$ with regularity sequence $\barl$, so that $C(f(y)) = \wQ e^{2\eps \ell_1}$, and thus
\begin{equation}\label{eqn:C-slow}
e^{-2\eps} \leq C(f(y)) / C(y) \leq e^{2\eps}.
\end{equation}
Now we need to estimate $\measuredangle(E_y^s,E_y^u)$.  First observe that $\measuredangle(E_y^s,E_y^u) \geq \|e_y^s - e_y^u\|$ since the angle represents the length of the arc of the unit circle joining the endpoints of $e_y^s$ and $e_y^u$, while the right-hand side is the length of the straight line joining them.  Let $\underline{v}^{u/s} := D_y \Psi_{x_0}^{-1} e_y^{u/s}$ and observe that $\|\underline{v}^{u/s}\| \geq \frac 12$ by the first estimate in \eqref{eqn:norm-Psix}.  Moreover, $\underline{v}^{u/s} \in K^{s/u}$, so the endpoint of $\underline{v}^{u}$ lies in the region of $\RR^2$ given by
\[
\{(x,y) \in \RR^2 : |y| \leq \omega |x| \text{ and } x^2 + y^2 \geq 1/4 \},
\]
while the endpoint of $\underline{v}^u$ lies in the region
\[
\{(x,y) \in \RR^2 : |x| \leq \omega |y| \text{ and } x^2 + y^2 \geq 1/4 \},
\]
Thus, by \eqref{eqn:omega}, $\|\underline{v}^{u} - \underline{v}^{s}\| \geq 1/2$,\footnote{An elementary computation shows that the optimal lower bound is $(1-\omega)/\sqrt{2(1+\omega^2)}$.} 
and we conclude that
\[
\frac 12 \leq \|\underline{v}^{u} - \underline{v}^{s}\|
\leq \|D_y \Psi_{x_0}^{-1}\| \|e_y^u - e_y^s\|
\leq 4Q_0 \wQ e^{2\eps\ell_0} \measuredangle(E_y^s,E_y^u)
\]
using the second inequality in \eqref{eqn:norm-Psix}.  Thus
\[
\measuredangle(E_y^s,E_y^u) \geq \wQ^{-1} e^{-2\eps\ell_0},
\]
so putting $K(y) := \wQ^{-1} e^{-2\eps\ell_0}$ establishes \eqref{eq:angle}, and \eqref{eq:slowvar} for $K$ follows just as it did for $C$.  We conclude that the set $\Lambda'$ of shadowing points is $(\lambda/4,2\eps)$-hyperbolic.  Moreover, to find which $\Lambda'_\ell$ contains $y$, we write
\[
C(y) = \wQ e^{2\eps\ell_0} \leq e^{2\eps \ell}
\]
and find that this holds as soon as $\ell \geq \ell_0 + \frac 1{2\eps} \log \wQ$.  A similar computation with $K(y)$ shows that $y$ is $(\lambda/4,2\eps,\ell_0 + \lceil\frac 1{2\eps}\log\wQ\rceil)$-regular.

\part{Nice Rectangles and Young Towers}
\label{part:Y}

In this third and final part of the paper, we apply the general results stated in  
Theorem~\ref{thm:pseudo-orbit} and \ref{thm:inf-po}, and proved in Part~\ref{part:hyp} above, to our particular setting in order to  prove Theorems \ref{thm:existsY}, \ref{thm:existsnice1} and  \ref{thm:genkatok2}. As mentioned above, Theorem~\ref{thm:existsSRB} follows directly from Theorems \ref{thm:existsY} and  \ref{thm:existsnice1}, and therefore this completes the proofs of  all our results. The sections are organized as follows: 
in \S\ref{sec:almost-return} we prove Theorem \ref{thm:genkatok2}, which is essentially a reformulation of Theorem~\ref{thm:pseudo-orbit} in the setting of almost returns to nice domains, in \S \ref{sec:C->B} we show that Theorem \ref{thm:genkatok2} implies Theorem \ref{thm:existsY}, and in \S \ref{sec:B->A} we  prove Theorem \ref{thm:existsnice1}.

\section{Hyperbolic branches in nice domains: Proof of Theorem \ref{thm:genkatok2}}\label{sec:almost-return}

The proof of Theorem \ref{thm:genkatok2} consists of two parts. 
First  we show that every almost return gives rise to a pseudo-orbit and thus, by Theorem~\ref{thm:pseudo-orbit}, to a regular branch, which satisfies the hyperbolicity estimates given in Proposition \ref{prop:hyp-branch}. Then we show that this  regular branch  can be ``restricted'' to give a hyperbolic branch in the  nice domain \( \Gpq \).  This second part of the proof does not explicitly require Theorem \ref{thm:pseudo-orbit}, it only uses the existence of a regular branch, but does use in an essential way the fact that \( \Gpq \) is a nice domain.

To begin, let  \( C_{\ell} \) be the constant given in Theorem \ref{thm:HP} (without loss of generality we may assume that $C_\ell\geq 1$), \( c_{2}>0 \) as given in \eqref{eqn:bc}, and \( \delta>0 \) as in Theorem~\ref{thm:pseudo-orbit}. Then we let
\begin{equation}\label{def:r}
r:=\frac{\delta e^{-\lambda (\ell+1)}e^{-c_{2}}}{2C_{\ell}}\leq \delta,
\end{equation}
where the inequality holds since $C_\ell\geq 1$ and $\lambda, \ell, c_2 \geq 0$.
For generality we state the following Lemma for almost returns in a slightly more general setting than that of Theorem \ref{thm:genkatok2}, without any explicit references to rectangles or nice domains. 

\begin{Lemma}\label{lem:overlap}
 If   $x,y\in \Lambda_\ell$ and $k\geq 1 $ are such that $f^k(V_x^s) \cap V_y^u \neq\emptyset$, and
  $z\in f^k(V_x^s) \cap V_y^u$ satisfies $d(z,y) < r $ and $d(f^{-k}(z),x) < r $,
then the sequence $\barx = (x_{0},  \dots,  x_{k})$ given by 
\( x_j = 
f^j(x) \) for \(  0\leq j \leq k/2\) and
\( x_j =  f^{j-k}(y) \) for \(  k/2 < j \leq k
\) 
is an $(\barl, \delta, \lambda)$-pseudo-orbit for $\ell_j = \min(\ell + j, \ell + k-j)$.  
\end{Lemma}
\begin{proof}
Write $i=\lfloor k/2\rfloor$. By assumption \( z\in V_{y}^{u} \) and  \( f^{-k}(z) \in V_{x}^{s} \) and therefore by the assumptions of the Lemma and Theorem \ref{thm:HP} we have 
\[
d(f^{i}(x), f^{i-k}(z)) \leq C_\ell e^{-\lambda i} d(x,f^{-k}(z))\leq C_\ell e^{-\lambda i} r \leq \frac 12 \delta e^{-\lambda (\ell + i)},
\]
and 
\[
d(f^{i-k}(z)),f^{i-k}(y)) \leq C_\ell e^{-\lambda(k-i)} d(z,y)\leq C_\ell e^{-\lambda(k-i)}  r \leq \frac 12 \delta e^{-\lambda (\ell + i)},
\]
and thus $d(f(x_i),x_{i+1}) =  d(f^{i}(x), f^{i-k}(y)) \leq \delta e^{-\lambda(\ell + i)}$.  Since $f(x_j) = x_{j+1}$ for all $j\neq i$, this completes the proof.  \end{proof}

Consider now the setting of Theorem \ref{thm:genkatok2}: suppose \( \Gamma\)  is a nice regular set with \( \diam(\Gpq)<r \) and
suppose  \( x \in \Gamma \) has an almost return to \( \Gamma \) at time \( k\in T\mathbb N\). Then the assumptions of Lemma \ref{lem:overlap} are satisfied and there is an \(( \barl, \delta, \lambda) \)-pseudo-orbit $\barx = (x_{0},  \dots,  x_{k})$ as in the lemma starting and ending inside \( \Gpq \). 
Moreover, notice that
\[
  d(f(p), x_{1})=d(f(p), f(x_{0})) \leq e^{c_{2}}d(p, x_{0}) 
 \leq 
 e^{c_{2}} \diam (\Gpq) \leq e^{c_{2}}  r <  \delta e^{-\lambda (\ell+1)},
\]
where the first inequality uses the general fact from \eqref{eqn:bc} that $d(f(x),f(y)) \leq e^{c_2} d(x,y)$ for all $x,y\in M$, the second inequality is immediate since $p,x_0 \in \Gpq$, the third uses our assumption that $\diam(\Gpq) < r$, and the fourth uses \eqref{def:r}.

By a similar calculation, \( d(f(x_{k-1}), p) < \delta e^{-\lambda \ell} \), and therefore the sequence 
\[
\bar p := (p, x_{1}  \dots,  x_{k-1}, p)
\]
 is  also an \(( \barl, \delta, \lambda) \)-pseudo-orbit.
  Considering the  Lyapunov chart $\Psi_p \colon \mathcal{B}_p^{(\ell)} \to \mathcal{N}_p^{(\ell)}$, 
 by Theorem \ref{thm:pseudo-orbit} there is an $\ell$-regular branch   from $\mathcal{B}_p^{(\ell)}$ to itself associated to this pseudo-orbit, and we have the corresponding maps 
  \begin{equation}\label{eq:chartmaps}
  f^{0,k}_{\bar p}\colon \mathcal{B}^0_{\bar p} \to \mathcal{B}^k_{\bar p}
\quad\text{ and } \quad 
 f^{k} \colon \mathcal N^{0}_{\bar p}\to \mathcal N^{k}_{\bar p} 
\end{equation}
at the level of Lyapunov charts and of the manifold respectively, recall \eqref{eq:Bjxmap}.
For this branch 
Proposition \ref{prop:hyp-branch} gives the hyperbolicity estimates required in Definition \ref{def:hypbranch} for a $(\wQ e^{2\eps\ell},\lambda/3)$-hyperbolic branch.
Moreover,   concatenating any finite sequence of such branches gives a new $\ell$-regular branch that is associated to the concatenated pseudo-orbit, and thus has the same hyperbolicity estimates given by Proposition \ref{prop:hyp-branch}.  Thus  the collection of such branches satisfies the concatenation property.  

\begin{Remark}
We emphasize  that these are not yet the hyperbolic branches we require for $\Gpq$ as in Definition \ref{def:hypbranch}. Indeed, these branches are constructed on the scale of the Lyapunov chart which \emph{a priori} may be  significantly  bigger than the scale of the nice domain \( \Gpq \). The strips 
 \( \mathcal N^{0}_{\bar p},  \mathcal N^{k}_{\bar p} \) intersect \( \Gpq \) but may extend across the boundary of \( \Gpq \). 
We therefore need to ``restrict'' these branches to \( \Gpq \) and produce $\Gpq$-strips $\widehat{C}^{s} \subset \mathcal N^{0}_{\bar p}$  and $\widehat{C}^{u} \subset \mathcal N^{k}_{\bar p}$ such that $f^k$ maps $\widehat{C}^s$ onto $\widehat{C}^u$. Since these are subsets of the larger strips \( \mathcal N^{0}_{\bar p},  \mathcal N^{k}_{\bar p} \) and  the cones $K_{p,y}^{s/u} \subset T_y M$ defined in \eqref{eqn:M-cones} give conefields over $\Gpq$ that are adapted to the set \( \Gamma \), the restricted strips will automatically inherit the hyperbolicity and concatenation properties.  
\end{Remark}

The remaining part of the argument is essentially topological, and this is where the niceness assumption plays a crucial role. Indeed, the crucial consequence of niceness is formalized in the following statement. 
  
\begin{Lemma}\label{lem:niceint1}
Let \( \Gamma \) be a nice regular set and suppose that some \( x\in \Gamma \) has an almost return to \( \Gamma \) at a time $k\in \NN T$.  Then $f^k(W_x^s) \subset \Gpq$.
\end{Lemma}

\begin{proof}
Suppose by contradiction that the conclusion does not hold. Then \( f^{k}(W^{s}_{x}) \) must  intersect one of \( W^{u}_{p/q}\),  but this implies that the image under $f^{-k}$ of this intersection point lies in the interior of \( \Gpq \), which is forbidden by niceness.\end{proof}

As shown in Figure \ref{fig:get-branch}, let 
 \[
   \gamma^{u}_{p}:= W^{u}_{p} \cap   \mathcal N^{0}_{\bar{p}}
  \quad\text{ and } \quad 
     \gamma^{u}_{q}:= W^{u}_{q} \cap  \mathcal N^{0}_{\bar{p}}.
  \]

\begin{figure}[tbp]
\includegraphics[width=\textwidth]{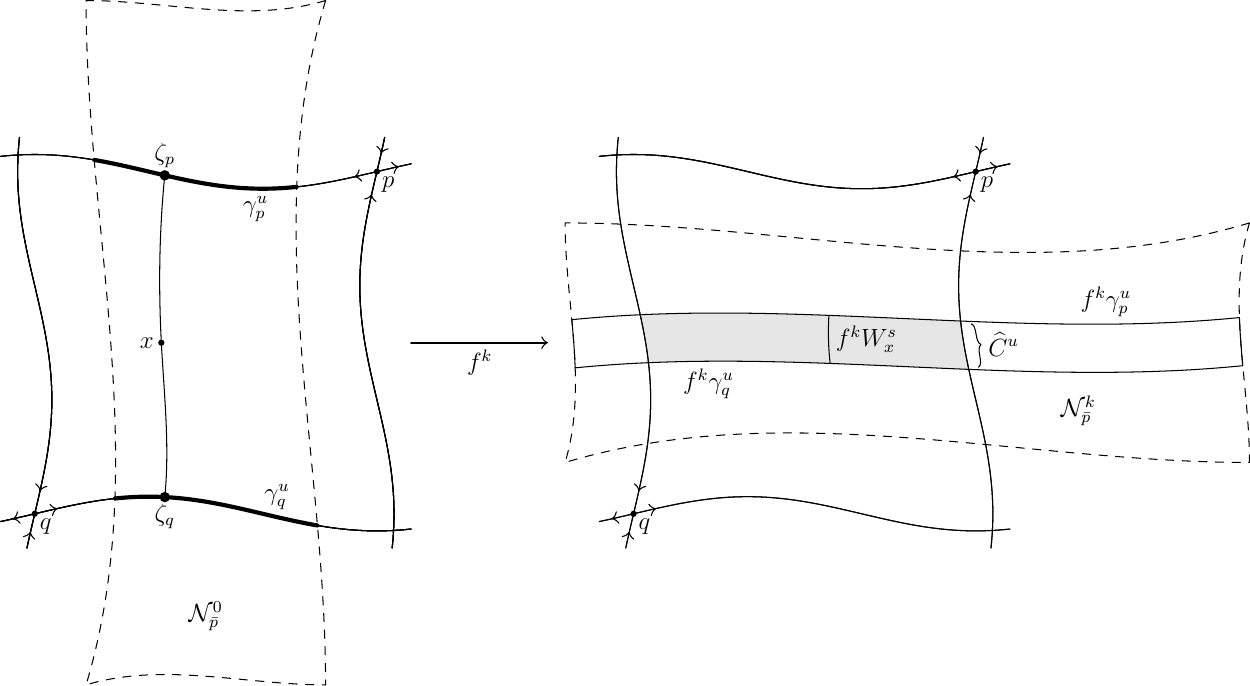}
\caption{Proving Theorem \ref{thm:genkatok2}.}
\label{fig:get-branch}
\end{figure}

\begin{Lemma}\label{lem:images}
The curves $f^{k}(\gamma^{u}_{p})$ and $ f^{k}(\gamma^{u}_{q})$ are full length unstable admissible curves in 
\( \Gpq \).
\end{Lemma}
 
\begin{proof}
 We prove the statement for \( \gamma^{u}_{p} \), the same argument applies to \( \gamma^{u}_{q} \).
 Observe that each endpoint of $\gamma_p^u$ lies on either $W_p^s$, $W_q^s$, or the stable boundary of $\mathcal{N}_{\bar{p}}^0$. Since $p$ and $q$ are fixed by $f^k$, we have $f^k(W_p^s) \subset W_p^s$ and $f^k(W_q^s) \subset W_q^s$; it follows that each endpoint of $f^k(\gamma_p^u)$ lies on either $W_p^s$, $W_q^s$, or the stable boundary of $f^k(\mathcal{N}_{\bar{p}}^0) = \mathcal{N}_{\bar{p}}^k$. Since $f^k(\gamma_p^u)$ must intersect $\Gpq$ by Lemma \ref{lem:niceint1}, these restrictions on its endpoints guarantee that it is a full length unstable admissible curve in $\Gpq$.
\end{proof}
 
\begin{proof}[Proof of Theorem  \ref{thm:genkatok2}] 
By Lemma \ref{lem:images},  \(   f^{k}(\gamma^{u}_{p}),   f^{k}(\gamma^{u}_{q}) \) are full length unstable admissible curves in \( \Gpq \) and therefore  they  define an unstable strip \(  \widehat C^{u}  \) in \(  \Gpq  \), whose preimage \(  \widehat C^{s}:= f^{-k}(\widehat C^{u})  \) is a stable strip in \(  \Gpq  \), thus yielding the desired hyperbolic branch.  The required hyperbolicity estimates are inherited from the regular branch of which this hyperbolic branch is a subset.
\end{proof}

\section{Building a tower out of hyperbolic branches: Proof of Theorem \ref{thm:existsY}}
\label{sec:C->B}

In this section we prove Theorem \ref{thm:existsY}. 
In \S \ref{sec:saturated} we introduce some definitions and notation and reduce the proof of  Theorem \ref{thm:existsY} 
to three Propositions \ref{thm:nicehyp}, \ref{prop:sat-tower}, and \ref{prop:nicefat-tower}. We then prove each Proposition in its own subsection. 

\subsection{Saturation and Young Towers} \label{sec:saturated}

Before formulating the propositions, we need to establish some notation and to introduce the notion of \emph{saturation}.

Let $\Gpq$ be a nice domain and suppose that $A\subset \Gpq$ is such that every point $x\in A$ has full length stable and unstable curves $W_x^{s/u}$.
Suppose moreover that $A$ has the $(C,\kappa)$-hyperbolic branch property for some $C,\kappa>0$ (see Definition \ref{def:hyp-branch-prop}). Let $\mathfrak{C}(A)$ denote the set of hyperbolic branches associated to almost returns to $A$, and let $\mathfrak{C}_0(A) \subset \mathfrak{C}(A)$ be the subset consisting of those branches associated to (true) returns. Let $\mathfrak{C}^*(A)$ denote the set of branches obtained by concatenating finitely many members of $\mathfrak{C}(A)$. Thus we have
\[
\mathfrak{C}_0(A) \subset \mathfrak{C}(A) \subset \mathfrak{C}^*(A).
\]

\begin{Remark}\label{hyp-br}
Both inclusions can be proper. For the first one, $A$ can have almost returns without having any true returns, for example, if the set $A$ consists of one non-periodic point which returns to $\Gpq$, then there is no return to $A$ but almost returns may exist. For the second, $\mathfrak{C}^*(A)$ may even contain branches $f^k\colon \widehat{C}^s\to \widehat{C}^u$ such that $\widehat{C}^s$ and $\widehat{C}^u$ are disjoint from $A$.
This can occur if 
two branches $f^i \colon \widehat{C}_1^s \to \widehat{C}_1^u$ and $f^j \colon \widehat{C}_2^s \to \widehat{C}_2^u$ generated by almost returns have a concatenated branch $f^{i+j} \colon \widehat{C}^s \to \widehat{C}^u$ (recall Definition \ref{def:concat} and Figure \ref{fig:concat}) 
with the property that the part of $A$ in $\widehat{C}_1^s$ lies entirely outside of $\widehat{C}^s$, and similarly for $\widehat{C}_2^u$.
\end{Remark}

Since \(  f  \) is a diffeomorphism of a compact manifold the derivative of \(  f  \)  is bounded and so for each \(  i \geq 1  \) there can be at most a finite number of  hyperbolic branches of order $i$.  
Thus we can index $\mathfrak{C}_0(A)$ by 
\[
I_0(A):=\{ij: i\in T\NN, j\in \{1,...., m_i\}\},
\]
where \( i \) gives the order (return time) of the  hyperbolic branch and \( j \) indexes  the \( m_i \) hyperbolic branches with  order \( i \). Therefore we obtain
\begin{equation}\label{eq:C}
\mathfrak{C}_0(A) =  \{f^{i}\colon \widehat C^s_{ij}\to \widehat C^u_{ij}\}_{ij\in I_0(A)}.
\end{equation}
We consider the particular case when $A$ is a rectangle.

\begin{Definition}[Saturated Rectangle]\label{def:satrect}
Let $\Gamma\subset\Gpq$ be a nice rectangle with the hyperbolic branch property\footnote{In fact one does not need the full strength of the hyperbolic branch property to make this definition; it suffices to have a hyperbolic branch associated to each (true) return.} and let $I_0(\Gamma)$ and 
$\mathfrak{C}_0(\Gamma)$ be as above (see \eqref{eq:C}). The rectangle $\Gamma$ is called 
\emph{saturated} if for all $ij \in I_0(\Gamma)$, we have
\begin{equation}\label{eq:saturation}
 C^{s}_{ij}:=f^{-i}(\widehat C^{u}_{ij}\cap C^{s}) \subset C^{s}
\qand 
 C^{u}_{ij}:=f^{i}(\widehat C^{s}_{ij}\cap C^{u})\subset C^{u},
\end{equation} 
where $C^{s/u} = \bigcup_{x\in \Gamma} W_x^{s/u}$ as in \eqref{eq:rectanglestruct}.
\end{Definition}

See \S\ref{sec:sat-disc} for a general discussion of the saturation property and its connection to the first return property of the Young tower that is eventually constructed, as well as to multiplicity of the associated countable-state Markov coding. That section also includes an example of a non-saturated rectangle.

Given $\chi>\lambda>0$, $0<\eps < \eps_1(f,\chi,\lambda)$, and $\ell\in \NN$, let $r>0$ be given by Theorem \ref{thm:genkatok2}, so that every $(\chi,\eps,\ell,r)$-nice regular set has the $(\wQ e^{2\eps\ell},\lambda/3)$-hyperbolic branch property.

\begin{Proposition} \label{thm:nicehyp} 
With $\chi,\lambda,\eps,\ell,r$ as above, let $A$ be a   \((\chi, \eps, \ell, r)\)-nice  recurrent set. Then there exists  a nice $(\wQ e^{2\eps\ell},\lambda/3)$-rectangle  $\Gamma \subset \Gpq$ such that the following are true.
\begin{enumerate}
    \item\label{contains-A} $A \subset \Gamma \subset \Gpq$.
    \item\label{T-rec} $\Gamma$ is $T$-recurrent.
    \item\label{branch} $\Gamma$ satisfies the $(\wQ e^{2\eps\ell},\lambda/3)$-hyperbolic branch property.
    \item\label{returns} $\mathfrak{C}_0(\Gamma) = \mathfrak{C}(\Gamma) = \mathfrak{C}^*(\Gamma) = \mathfrak{C}^*(A)$.
    \item\label{saturated} $\Gamma$ is saturated.
    \item\label{reg} $\Gamma$ is $(\lambda/4,2\eps,\ell+\ell')$-regular, where $\ell' = \lceil \frac 1{2\eps} \log \wQ \rceil$.
\end{enumerate}
\end{Proposition}

\begin{Remark}\label{rem:Gamma-reg}
See Remark \ref{rmk:by-4-or-3} for a discussion of why $\lambda/3$ appears in the hyperbolic branch property \eqref{branch} and $\lambda/4$ appears in the regularity property \eqref{reg}.
This regularity property is used in the proof of Theorem \ref{thm:existsY}\ref{B2} to establish a H\"older continuity estimate (\eqref{eqn:fa-Gamma} in Proposition \ref{prop:nicefat-tower}) leading to bounded distortion (condition \ref{Y2} in the definition of Young tower), but the construction in Proposition \ref{prop:sat-tower} of the topological Young tower  for Theorem \ref{thm:existsY}\ref{B1} only uses the hyperbolic branch property \eqref{branch}.

At first glance it may appear redundant to state both conclusions \eqref{branch} and \eqref{reg} in Proposition \ref{thm:nicehyp}, since Theorem \ref{thm:genkatok2} guarantees that every nice regular set satisfies the hyperbolic branch property, so one might reasonably expect to deduce \eqref{branch} as a consequence of \eqref{reg}.
The problem is that $A$ is $(\chi,\eps,\ell)$-regular while $\Gamma$ is only $(\lambda/4,2\eps,\ell+\ell')$-regular (which is weaker),
and so in order to apply Theorem \ref{thm:genkatok2} and deduce that $\Gamma$ satisfies the $(C,\kappa)$-hyperbolic branch property for some $\kappa \in (0,\lambda/4)$, we would need to make more careful choices in our original setup, choosing $0 < \eps < \frac 12 \eps_1(f,\lambda/4,\kappa)$ and then $r = r(\lambda/4,\kappa,2\eps)$. This can be done but would create a more complicated set of conditions for $\eps$ and $r$ in our main theorems, which we prefer to avoid; after deducing conclusion \eqref{reg} from Theorem \ref{thm:inf-po} (see the paragraph following Lemma \ref{lem:nicesat}), we deduce conclusion \eqref{branch} by finding an almost return to $A$ itself, to which Theorem \ref{thm:genkatok2} can be applied, see Lemma \ref{lem:any-almost} and the sentence following it.
\end{Remark}

\begin{Proposition}\label{prop:sat-tower}
Let \( \Gamma \subset \Gpq \) be a $T$-recurrent nice \( (C_0,\kappa) \)-rectangle satisfying the \(  (C_0,\kappa)  \)-hyperbolic branch property, and suppose that $\Gamma$ is saturated. Then \( \Gamma \) supports a First $T$-return Topological Young Tower.
\end{Proposition}

Replacing $T$-recurrence with Lebesgue-strong $T$-recurrence
allows us to upgrade the topological Young Tower to a fully fledged Young Tower satisfying the required distortion estimates and having integrable return times.

\begin{Proposition}\label{prop:nicefat-tower}
Let \( \Gamma \subset \Gpq\) be a Lebesgue-strongly $T$-recurrent nice $(C_0,\kappa)$-rectangle satisfying the \(  (C_0,\kappa)  \)-hyperbolic branch property, and suppose that $C_0 e^{-\kappa T} < 1$. Suppose moreover that $\Gamma$ is saturated
and that there are $C_1,\beta_1,\gamma_1>0$ such that $\gamma_1 < \beta_1 \kappa$ and for all $a\in \ZZ$ and $x,y \in f^a(\Gamma)$, we have
\begin{equation}\label{eqn:fa-Gamma}
d(E_x^u,E_y^u) \leq C_1 e^{\gamma_1 |a|} d(x,y)^{\beta_1},
\end{equation}
Then  \(  \Gamma  \)  supports a First $T$-return Young Tower with integrable return times.
\end{Proposition}

 To prove Theorem \ref{thm:existsY}, we will apply Propositions \ref{prop:sat-tower} and \ref{prop:nicefat-tower}
 to the rectangle constructed in Proposition \ref{thm:nicehyp}, with $C_0 = \wQ e^{2\eps\ell}$ and $\kappa = \lambda/3$.
Propositions \ref{thm:nicehyp}, \ref{prop:sat-tower}, and \ref{prop:nicefat-tower} will be proved in \S\ref{sec:sat-rec}, \S\ref{sec:building-tower}, and \S\ref{sec:fatyoung}, respectively.

\begin{Remark}\label{rmk:only-true}
The only hyperbolic branches that are required in the proof of Propositions \ref{prop:sat-tower} and \ref{prop:nicefat-tower} are those associated to true returns (of $\Gamma$); so these results remain true if the hyperbolic branch property from Definition \ref{def:almostreturns} is weakened to only require that the set of true returns produces a collection of hyperbolic branches with the concatenation property, rather than requiring such a collection for the (larger) set of almost returns. Branches associated to almost returns (of $A$) play a crucial role in the proof of Proposition \ref{thm:nicehyp} for establishing the saturation property; roughly speaking, once this property is obtained it is enough to consider true returns.
\end{Remark}

\begin{proof}[Proof of Theorem \ref{thm:existsY}] 
As stated before Proposition \ref{thm:nicehyp}, we choose $r$ depending on $\chi,\lambda,\eps,\ell$ to satisfy the conditions of Theorem \ref{thm:genkatok2}, and thus given a \((\chi, \eps, \ell, r)\)-nice recurrent set $A$ as in the assumptions of Theorem \ref{thm:existsY}, 
Proposition \ref{thm:nicehyp} yields a nice rectangle $\Gamma \subset \Gpq$ that contains $A$, is $T$-recurrent, satisfies the $(\wQ e^{2\eps\ell},\lambda/3)$-hyperbolic branch property, and is saturated. Thus Proposition \ref{prop:sat-tower} applies to $\Gamma$, and we conclude that $\Gamma$ supports a First $T$-return Topological Young Tower, which proves the first part of Theorem \ref{thm:existsY}. For the second part of Theorem \ref{thm:existsY}, we will apply Proposition \ref{prop:nicefat-tower} to $\Gamma$, which requires us to verify Lebesgue-strong recurrence, the $(C_0,\kappa)$-hyperbolic branch property with $C_0 e^{-\kappa T} < 1$, and \eqref{eqn:fa-Gamma} with $\gamma_1 < \beta_1 \kappa$.

To this end, first observe that since $A$ is Lebesgue-strongly $T$-recurrent, so is any set containing $A$, including $\Gamma$. Moreover, with $C_0= \wQ e^{2\eps\ell}$ and $\kappa = \lambda/3$, we have $C_0 e^{-\kappa T} = \wQ e^{2\eps\ell} e^{-\lambda T/3} < 1$ by our choice of $T$ in Definition \ref{def:nice}. Finally, since $\Gamma$ is $(\lambda/4,2\eps,\ell+\ell')$-regular by Proposition \ref{thm:nicehyp}\eqref{reg}, we see that $f^a(\Gamma)$ is $(\lambda/4,2\eps,\ell+\ell'+|a|)$-regular for every $a\in \ZZ$. We can deduce \eqref{eqn:fa-Gamma} from Proposition \ref{prop:theta-holder}, and specifically \eqref{eqn:Esu-holder} (with $s$ replaced by $u$), where we replace $\chi$ with $\lambda/4$. More precisely, regularity guarantees that for all $x,y \in f^a(\Gamma)$ we have \eqref{eqn:Cvu-} and \eqref{eqn:Cvs-} with $C=e^{2\eps(\ell+\ell'+|a|)}$ and $\chi$ replaced by $\lambda/4$, and moreover $\measuredangle(E_z^s,E_z^u) \geq C^{-1}$ for $z=x,y$, so that Proposition \ref{prop:theta-holder} gives $Q_6'>0$ depending only on $M$, $\|Df^{\pm1}\|$, $\alpha$, $|Df^{\pm 1}|_\alpha$, and $\lambda/4$, such that 
\[
d(E_x^u,E_y^u)
\leq Q_6' (e^{2\eps(\ell+\ell'+|a|)})^{3 \cdot 2\gamma'} d(x,y)^{\beta'}
= Q_6' e^{12\eps \gamma'(\ell+\ell')} e^{12\eps\gamma'|a|} d(x,y)^{\beta'},
\]
where $\beta'$ and $\gamma'$ are obtained by replacing $\chi$ by $\lambda/4$ in \eqref{eqn:greek0} to get
\begin{equation}\label{eqn:g'b'}
\gamma' = \frac{\frac\lambda4 - c_1}{2 \frac\lambda 4}
= \frac{\lambda-4c_1}{2\lambda}
\quad\text{and}\quad
\beta' = \frac{2\frac\lambda 4}{c_3 + \frac\lambda4} \alpha = \frac{2\lambda\alpha}{4c_3 + \lambda}.
\end{equation}
It follows that \eqref{eqn:fa-Gamma} holds with $C_1 = Q_6' e^{12\eps \gamma'(\ell+\ell')}$, $\gamma_1 = 12\eps\gamma'$, and $\beta_1 = \beta'$. Thus the condition that $\gamma_1 < \beta_1 \kappa$ is equivalent to $12\eps\gamma' < \beta' \lambda/3$. Using \eqref{eqn:g'b'}, this is equivalent to
\[
6\eps \frac{\lambda - 4c_1}{\lambda} < \frac{2\lambda \alpha}{4c_3 + \lambda} \cdot \frac \lambda 3,
\]
which holds by the fifth inequality in \eqref{eqn:eps}. This confirms that we can apply Proposition \ref{prop:nicefat-tower} to $\Gamma$, and completes the proof of Theorem \ref{thm:existsY}.
\end{proof}

\subsection{Proof of Proposition \ref{thm:nicehyp}}\label{sec:sat-rec}

We will define \( \Gamma \) as the  ``maximal invariant set'' for the dynamics generated by the hyperbolic branches associated to almost returns to $A$ (via Theorem \ref{thm:genkatok2}). 
This can be thought of as a generalization of the standard horseshoe where we define a maximal invariant set as consisting of the points which remain in the strips for all forward and backward iterations. The key difference is that in the horseshoe setting we have at most a finite number of branches with pairwise disjoint stable strips and pairwise disjoint unstable strips, all of which have the same return time.
In our   setting,  a point \( x \) may belong to infinitely many stable strips with varying  return times and therefore we need a more involved construction.
With this in mind, we make the following general definition.
Let 
\[
\mathfrak{C} =  \{f^{i}\colon \widehat C^s_{ij}\to \widehat C^u_{ij}\}_{ij\in I}
\]
be a collection of hyperbolic branches indexed by the set \(I\). Note that so far we do not assume that these come from almost returns to a nice regular set; for now we allow an arbitrary collection of hyperbolic branches.

\begin{Definition}[Hyperbolic sequences]\label{def:hypseq}
A sequence \(  \mathfrak h^{+}=\{i_{m}j_{m}\}_{m=0}^\infty \)  with \( i_mj_m \in I\) is a \emph{forward hyperbolic sequence} for $x\in \Gpq$ if for all $m\geq 0$ we have
\begin{equation}\label{eqn:N-hyperbolic+}
f^{i_0 + i_1 + \cdots + i_{m-1}}(x) \in \widehat C^s_{i_mj_m}.
\end{equation}
Similarly, \( \mathfrak h^{-}=\{i_{m}j_{m}\}_{m=-\infty}^{-1}\) is a \emph{backward hyperbolic sequence} for $x\in \Gpq$ if for all $m< 0$ we have
\begin{equation}\label{eqn:N-hyperbolic}
f^{-(i_m + \cdots + i_{-1})}(x) \in \widehat C^s_{i_m j_m}.
\end{equation}
If \( \mathfrak h^{-}\) and \(\mathfrak h^{+} \) are backward and forward hyperbolic sequences for $x$, their concatenation \( \mathfrak h = \{i_{m}j_{m}\}_{m\in \mathbb Z} \)
is called a \emph{hyperbolic sequence}   for \( x \) (associated to the collection of hyperbolic branches \( \mathfrak C \)).  
\end{Definition}

A point \( x \) may or may not admit a forward or backward hyperbolic sequence and, if it does, these sequences need not be uniquely defined.  Consider the sets
\begin{align*}
C^{+} &:=\{x\in \Gpq \mid x \text{ has a forward hyperbolic sequence } \mathfrak h^{+}\}, \\
C^{-} &:=\{x\in \Gpq \mid x \text{ has a backward hyperbolic sequence } \mathfrak h^{-} \}.
\end{align*}

\begin{Lemma}\label{lem:nicesat}
Given an arbitrary collection $\mathfrak{C}$ of $(C,\kappa)$-hyperbolic branches,
the following set is a nice $(C,\kappa)$-rectangle and is $T$-recurrent:
\begin{equation}\label{eqn:Gamma}
\Gamma:= C^{+}\cap C^{-} = \{x\in \Gpq \mid x \text{ has a hyperbolic sequence } \mathfrak h\}. 
\end{equation}
Moreover, $C^+ = \bigcup_{x\in \Gamma} W_x^s$ and $C^- = \bigcup_{x\in \Gamma} W_x^u$, so that $C^+ = C^s(\Gamma)$ and $C^- = C^u(\Gamma)$ in the notation of \eqref{eq:rectanglestruct}.
\end{Lemma}

We think of \(  \Gamma  \) as the ``maximal invariant set'' of \(  \mathfrak {C}  \). 

\begin{proof}[Proof of Lemma \ref{lem:nicesat}]
To prove that $\Gamma$ is a rectangle, we use arguments similar to those
 in the standard horseshoe setting. Let \( \mathfrak h = \{i_{m}j_{m}\}_{m\in\ZZ} \) denote a hyperbolic sequence for the point \( x \), and $\mathfrak h^\pm$ its forward and backward parts. 
For each $n\geq 0$, the branches $f^{i_m} \colon \widehat{C}^s_{i_m j_m} \to \widehat{C}^u_{i_m j_m}$ with $m=0,1,\dots, n$ can be concatenated as in Definition \ref{def:concatenation} to produce a branch of order $i_0 + i_1 + \cdots + i_n$; by Proposition \ref{prop:intersect-branches}, the intersection (over all $n\geq 0$) of the resulting branches is a local $(C,\kappa)$-stable curve that has full length in $\Gpq$. It can be characterized as
\begin{equation}\label{eqn:Ws-from-h}
W_x^s = \{y\in \Gpq : \mathfrak h^+ \text{ is a forward hyperbolic sequence for } y\}.
\end{equation}
A completely analogous  argument shows that every \( x\in \Gamma \) has a full length local $(C,\kappa)$-unstable curve 
\begin{equation}\label{eqn:Wu-from-h}
W_x^u = \{y\in \Gpq : \mathfrak h^- \text{ is a backward hyperbolic sequence for } y\}.
\end{equation}
This implies  that for any \( x,y\in \Gamma \) the intersection  \( W^{s}_{x}\cap W^{u}_{y} \) consists of a single point \( z \).  Moreover, $x$ has a forward hyperbolic sequence $\mathfrak h_x^+$ and $y$ has a backward hyperbolic sequence $\mathfrak h_y^-$; writing $\mathfrak h$ for the concatenation of these two sequences, it follows from \eqref{eqn:Ws-from-h} and \eqref{eqn:Wu-from-h} that $\mathfrak h$ is a hyperbolic sequence for $z$, so $z\in \Gamma$.  
This shows that $\Gamma$ is a rectangle, as claimed.

To see that $C^+ = C^s(\Gamma)$ it suffices to observe that $C^+$ is a union of curves $W_x^s$ by \eqref{eqn:Ws-from-h}, so $C^+ = \bigcup_{x\in C^+} W_x^s \supset \bigcup_{x\in \Gamma} W_x^s$, while given an arbitrary $x\in C^+$ and $y\in \Gamma$ we have $z = [x,y] \in \Gamma$ and $x\in W_z^s$, so $C^+ \subset \bigcup_{z\in \Gamma} W_z^s$. A similar argument gives $C^- = C^u(\Gamma)$.

Finally, the fact that $\Gamma$ is recurrent follows almost immediately from the definition; if $\mathfrak h = \{i_m j_m\}_{m\in \ZZ}$ is a hyperbolic sequence for $x\in \Gamma$, then $f^{i_0}(x)$ and $f^{-i_{-1}}(x)$ have hyperbolic sequences given by shifting $\mathfrak h$ one index in either direction, hence $i_0$ and $-i_{-1}$ are return times to $\Gamma$.
\end{proof}

Now we restrict ourselves to the case when $\mathfrak{C}$ arises from almost returns.
Given $\chi>\lambda>0$, $0<\eps < \eps_1(f,\chi,\lambda)$, and $\ell\in \NN$, let $r>0$ be given by Theorem \ref{thm:genkatok2}, so that every $(\chi,\eps,\ell,r)$-nice regular set has the $(\wQ e^{2\eps\ell},\lambda/3)$-hyperbolic branch property. Let $A$ be a   \((\chi, \eps, \ell, r)\)-nice  recurrent set and let \( \mathfrak C =\mathfrak{C}(A) \) be the corresponding collection of hyperbolic branches associated to almost returns via Theorem \ref{thm:genkatok2}. Then Lemma \ref{lem:nicesat} shows that \eqref{eqn:Gamma} defines a nice $(\wQ e^{2\eps\ell},\lambda/3)$-rectangle $\Gamma \subset \Gpq$ that is $T$-recurrent, verifying conclusion \eqref{T-rec} of Proposition \ref{thm:nicehyp}. 
Conclusion \eqref{reg} on regularity follows from Theorem \ref{thm:inf-po}.
For conclusion \eqref{contains-A}, we observe that since $A$ is recurrent, given $x\in A$ there are $i_0,i_1,i_2,\dots \in \mathbb{N}$ such that $x_k := f^{i_0 + i_1 + \cdots + i_{k-1}}(x) \in A$ for all $k\geq 0$; then $f^{i_k}(x_k) = x_{k+1}$ is a return to $A$ that produces a corresponding hyperbolic branch in $\mathfrak{C}$, and the sequence of branches obtained this way is a forward hyperbolic sequence for $x$. A backward hyperbolic sequence is produced similarly, and thus $x\in \Gamma$, verifying \eqref{contains-A}.

\begin{Remark}\label{rmk:why-almost}
Everything in the previous paragraph remains true if we replace $\mathfrak{C}(A)$ with the (potentially smaller) collection $\mathfrak{C}_0(A)$ of branches associated to true returns. However, in the remainder of the proof it will be essential that we define $\Gamma$ using the collection of branches corresponding to almost returns.
\end{Remark}

It remains to prove conclusions \eqref{branch}, \eqref{returns}, and \eqref{saturated}.
In order to prove these conclusions, we need the following result about hyperbolic branches. We remark that this is the one and only place in the paper where we use the assumption from Definition \ref{def:nice} that \( T \) is even. 

\begin{Lemma}\label{lem:nestdisj}
Let $\Gpq$ be a nice domain. Then for any hyperbolic branch $f^i\colon \widehat C^s \to \widehat C^u$ we have $\Int(f^k(\widehat C^s)) \cap \partial \Gpq = \emptyset$ for all $k=0,\dots, i$ that are multiples of $T$.
Moreover, if $f^{i'} \colon \widehat C'^s \to \widehat C'^u$ is any other hyperbolic branch, then the corresponding stable (resp.\ unstable) strips are either nested or disjoint.
\end{Lemma} 

\begin{proof}
The first statement is automatic for $k=0,i$. Suppose that there exists some 
\(  k \in \{1,..,i-1\}  \)  such that 
\(  \Int (f^{k}(\widehat C^{s})) \cap \partial \Gpq \neq  \emptyset  \).  Then we must have \( \Int (f^{k}(\widehat C^{s})) \cap (W^{s}_{q}\cup W^{s}_{p}) \neq \emptyset  \) or \( \Int (f^{k}(\widehat C^{s})) \cap (W^{u}_{q}\cup W^{u}_{p}) \neq \emptyset  \) (or both). In the first case, iterating forward by \(  i-k  \) iterates,  this would imply that \(  \Int (\widehat C^{u})   \cap f^{i-k}(W^{s}_{q}\cup W^{s}_{p})  \neq\emptyset \), contradicting the niceness property of $\Gpq$.  Similarly, in the second case, iterating backwards by \(  k  \) iterates, this would imply \(  \Int (\widehat C^{s})  \cap f^{-k}(W^{u}_{q}\cup W^{u}_{p}) \neq\emptyset  \), contradicting niceness.

For the second statement, assume without loss of generality that \( i\leq i' \). 
Suppose by contradiction that the two stable strips \( \widehat  C^{s},  \widehat C'^{s}\neq\emptyset\) are neither nested nor disjoint (the argument for unstable strips is exactly the same). Recall that the stable boundaries of \( \widehat  C^{s},  \widehat C'^{s} \) are pieces of the global stable curves of \( p, q \) (see Remark \ref{rem:boundaries}) and therefore cannot intersect
unless they agree; thus
the intersection \( \widehat  C^{s}\cap \widehat C'^{s} \) is a single stable admissible curve or a non-empty stable strip. 

In the first case, we see that $f^i$ maps this stable admissible curve $\gamma^s$ to the local stable manifold of either $p$ or $q$. Due to our choice of \( t \) even, $D_pf^T$ and $D_qf^T$ have positive eigenvalues, and it follows that since $i' \geq i$ are  multiples of $T$, there is an open region $U\subset \widehat C'^{s}$ adjacent to $\gamma^s$ such that $f^{i'}(U) \cap \Gpq = \emptyset$, contradicting the assumption that $f^{i'} \colon \widehat C'^s \to \widehat C'^u$ is a hyperbolic branch.

In the second case,
each stable strip has one of the components of its stable boundary inside the interior of the other stable strip, see first figure in Figure \ref{fig:nesteddisjoint}. 
It follows that \( f^{i}(\widehat C'^{s}) \) contains a piece of the stable boundary of \( \Gpq \) in its interior, contradicting the first statement proved above.  
\end{proof}

\begin{figure}[tbp]
\includegraphics[width=\textwidth]{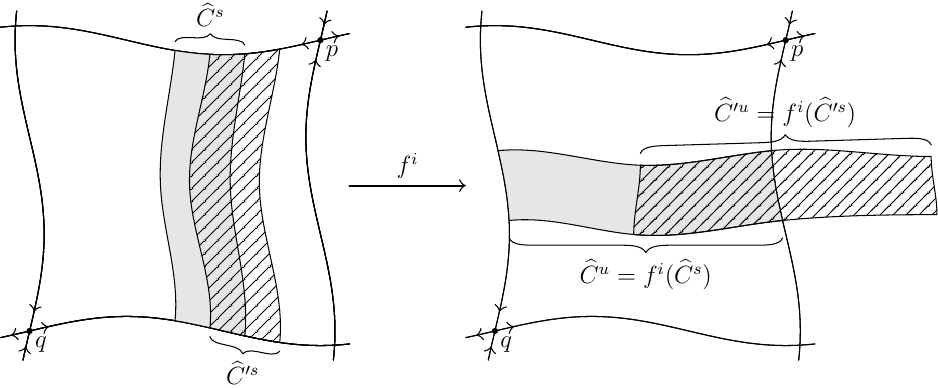}
\caption{Two overlapping branches leads to a contradiction (Lemma \ref{lem:nestdisj}).}
\label{fig:nesteddisjoint}
\end{figure}

To prove conclusion \eqref{branch}, that $\Gamma$ has the $(\wQ e^{2\eps\ell},\lambda/3)$-hyperbolic branch property, we start by proving the following lemma, which is illustrated in Figure \ref{fig:i-tau-i}.

\begin{figure}[tbp]
\includegraphics[width=\textwidth]{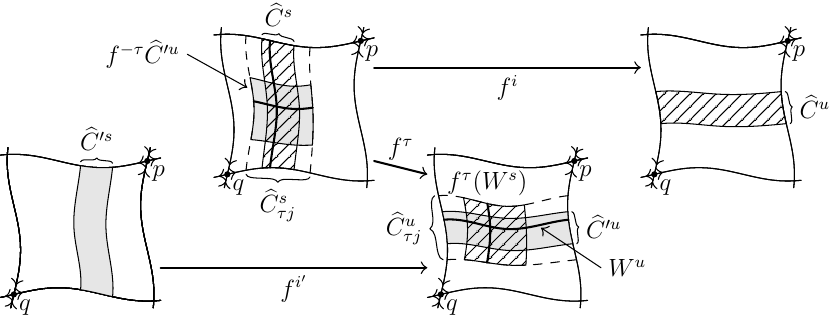}
\caption{Proving Lemma \ref{lem:almostret}.}
\label{fig:i-tau-i}
\end{figure}

\begin{Lemma}\label{lem:almostret}
Suppose that $W^s$ and $W^u$ are full length local stable and unstable curves in $\Gpq$ and that $\tau \in T\mathbb{N}$ is such that $f^\tau(W^s) \cap W^u \neq \emptyset$. Suppose moreover that $f^i \colon \widehat{C}^s \to \widehat{C}^u$ and $f^{i'} \colon \widehat{C}'^s \to \widehat{C}'^u$ are hyperbolic branches with $i,i' \geq \tau$ such that $W^s \subset \widehat{C}^s$ and $W^u \subset \widehat{C}'^u$. Then for any full length local stable and unstable curves $W_*^s \subset \widehat{C}^s$ and $W_*^u \subset \widehat{C}'^u$, we have $f^\tau(W_*^s) \cap W_{*}^u \neq \emptyset$.

In particular, if $A \cap \widehat{C}^s \neq \emptyset$ and $A \cap \widehat{C}'^u \neq \emptyset$, then there exists $\tau j \in I(A)$ such that the hyperbolic branch $f^\tau \colon \widehat{C}^s_{\tau j} \to \widehat{C}^u_{\tau j}$ from the collection $\mathfrak{C}(A)$ has $W^s \subset \widehat{C}^s_{\tau j}$ and $W^u \subset \widehat{C}^u_{\tau j}$.
\end{Lemma}
\begin{proof}
Observe that $f^\tau(\widehat{C}^s) \cap \widehat{C}'^u \supset f^\tau(W^s) \cap W^u \neq \emptyset$. Applying $f^{-\tau}$ gives
\begin{equation}\label{eqn:strip-intersects}
\widehat C^{s}\cap  f^{-\tau}(\widehat C'^{u})  \neq \emptyset.
\end{equation}
This implies in particular that $f^{-\tau}(\widehat C'^u) \cap \Gpq \neq \emptyset$; on the other hand, Lemma \ref{lem:nestdisj} gives $\Int f^{-\tau}(\widehat C'^u) \cap \partial\Gpq = \emptyset$, and it follows that 
\begin{equation}\label{eqn:in-Gpq}
 f^{-\tau}(\widehat C'^u )  \subset \Gpq.
\end{equation}
Now we claim that 
\begin{equation}\label{eqn:strip-crosses}
 f^{-\tau}(\partial^{s}\widehat C'^{u}) \cap  \Int(\widehat C^{s}) =\emptyset,
\end{equation}
as shown in Figure \ref{fig:i-tau-i}.  To see this, observe that 
\begin{align*}
f^i(f^{-\tau}(\partial^{s}\widehat C'^{u}) \cap  \Int(\widehat C^{s}))
&= f^{i-\tau}(\partial^{s}\widehat C'^{u}) \cap \Int(\widehat C^u) \\
&\subset f^{i-\tau}(W_p^s \cap W_q^s) \cap \Int \Gpq = \emptyset
\end{align*}
since $\Gpq$ is a nice domain. Applying $f^{-i}$ gives \eqref{eqn:strip-crosses}, and together with \eqref{eqn:strip-intersects} and \eqref{eqn:in-Gpq} this implies that 
\(
f^{-\tau}(\widehat C'^{u})
\)
fully crosses  \( \widehat C^{s} \) in the unstable direction, as shown in Figure \ref{fig:i-tau-i}.  The conclusion about $W_*^u$ and $W_*^s$ follows; these appear in Figure \ref{fig:i-tau-i} in the same configuration as $W^u$ and $W^s$ do.

For the final claim in the lemma, choose $x\in A\cap \widehat{C}^s$ and $y\in \widehat{C}'^u$; then apply the first part of the lemma to deduce that $f^\tau(W_x^s) \cap W_y^u \neq \emptyset$. By Theorem \ref{thm:genkatok2}, the collection $\mathfrak{C}(A)$ contains a hyperbolic branch $f^\tau \colon \widehat{C}^s_{\tau j} \to \widehat{C}^u_{\tau j}$ associated to this almost return, which has $\widehat{C}^s_{\tau j} \supset \widehat{C}^s \supset W^s$ and $\widehat{C}^u_{\tau j} \supset \widehat{C}'^u \supset W^u$ as in Figure \ref{fig:i-tau-i}.
\end{proof}

Observe that the previous two lemmas did not refer to the rectangle $\Gamma$; they are general facts about nice domains and hyperbolic branches. Now we once again consider the rectangle $\Gamma$ and prove the hyperbolic branch property in conclusion \eqref{branch}.

\begin{Lemma}\label{lem:any-almost}
Let $\Gamma$ be the rectangle constructed in \eqref{eqn:Gamma} using the collection $\mathfrak{C}(A)$ of $(\wQ e^{2\eps\ell},\lambda/3)$-hyperbolic branches associated to almost returns of $A$ by Theorem \ref{thm:genkatok2}.
Let $z,z'\in \Gamma$ and $\tau\in T\NN$ be such that $f^\tau(W_z^s) \cap W_{z'}^u \neq\emptyset$. Then there is a branch $f^\tau \colon \widehat{C}^s \to \widehat{C}^u$ in the collection $\mathfrak{C}^*(A)$ such that $z \in \widehat{C}^s$.
\end{Lemma}
\begin{proof}
A natural idea is to try to apply Lemma \ref{lem:almostret} to $W_z^s$ and $W_{z'}^u$ and thus produce an almost return of $A$ that gives the desired hyperbolic branch. However, the lemma requires us to choose branches for $z$ and $z'$ with order at least $\tau$; we can find such branches in $\mathfrak{C}^*(A)$ but not necessarily in $\mathfrak{C}(A)$, and  there is no guarantee that these branches will contain any elements of $A$ (see Remark \ref{hyp-br}).

Thus before applying Lemma \ref{lem:almostret}, we first observe that $z$ has a forward hyperbolic sequence $\mathfrak{h}^+ = \{i_m j_m\}_{m=0}^\infty$, and let $n\geq 0$ be such that $\sum_{m=0}^{n-1} i_m \leq \tau < \sum_{m=0}^n i_m$. Let $k = \sum_{m=0}^{n-1} i_m$ and observe that by the definition of a forward hyperbolic sequence, we can concatenate the branches $f^{i_m} \colon \widehat{C}_{i_m j_m}^{s} \to \widehat{C}_{i_m j_m}^{u}$ for $0\leq m \leq n-1$ to get a hyperbolic branch $f^k \colon \widehat{C}_z^s \to \widehat{C}_z^u$ from $\mathfrak{C}^*(A)$ such that $z\in \widehat{C}_z^s$.

Similarly, $z'$ has a backward hyperbolic sequence $\mathfrak{h}^- = \{i_{-m} j_{-m}\}_{m=1}^\infty$, and choosing $n'\geq 1$ such that $\sum_{m=1}^{n'-1} i_{-m} \leq \tau - k < \sum_{m=1}^{n'} i_{-m}$, we can write $\ell = \sum_{m=1}^{n'-1} i_{-m}$ and do a similar concatenation to get a hyperbolic branch $f^\ell \colon \widehat{C}_{z'}^s \to \widehat{C}_{z'}^u$ from $\mathfrak{C}^*(A)$ such that $z' \in \widehat{C}_{z'}^u$.

Now we observe that
\[
f^{\tau - k - \ell}(W_{f^k z}^s) \cap W_{f^{-\ell} z'}^u
\supset f^{\tau - \ell} (W_z^s) \cap f^{-\ell} (W_{z'}^u)
= f^{-\ell} (f^\tau (W_z^s) \cap W_{z'}^u) \neq \emptyset
\]
by the invariance properties of local stable and unstable curves, and so the point $f^k(z) \in \Gamma$ has an almost return at time $\tau - k - \ell$ via $f^{-\ell}(z') \in \Gamma$. Moreover, $f^k(z) \in \widehat{C}_{i_n j_n}^s$ and $f^{-\ell}(z') \in \widehat{C}_{i'_{-n'}j'_{-n'}}^u$, where $i_n, i'_{-n'} \geq \tau - k - \ell$ by our choice of $n$ and $n'$. The corresponding branches are in $\mathfrak{C}(A)$, so each of $\widehat{C}_{i_n j_n}^s$ and $\widehat{C}_{i'_{-n'}j'_{-n'}}^u$ contains an element of $A$, and we can apply Lemma \ref{lem:almostret} to get a hyperbolic branch $f^{\tau - k - \ell} \colon \widehat{C}_{(\tau-k-\ell)j}^s \to \widehat{C}_{(\tau-k-\ell)j}^u$ such that $z \in \widehat{C}_{(\tau-k-\ell)j}^s$. Concatenating this with the branches $f^k \colon \widehat{C}_z^s \to \widehat{C}_z^u$ and $f^\ell \colon \widehat{C}_{z'}^s \to \widehat{C}_{z'}^u$ gives a hyperbolic branch from $\mathfrak{C}^*(A)$ whose stable strip contains $z$.
\end{proof}

Lemma \ref{lem:any-almost} shows that every almost return of $\Gamma$ is associated to a hyperbolic branch from the collection $\mathfrak{C}^*(A)$, and thus any concatenation of such branches lies in this collection as well. Since these branches are all $(\wQ e^{2\eps\ell},\lambda/3)$-hyperbolic by Theorem \ref{thm:genkatok2}, we have proved conclusion \eqref{branch}.

In fact, we have also proved that
\[
\mathfrak{C}_0(\Gamma) \subset \mathfrak{C}(\Gamma) \subset \mathfrak{C}^*(\Gamma) \subset \mathfrak{C}^*(A).
\]
To prove conclusion \eqref{returns} it suffices to show that $\mathfrak{C}^*(A) \subset \mathfrak{C}_0(\Gamma)$. For this, observe that any branch in $\mathfrak{C}^*(A)$ is obtained by concatenating a finite sequence of branches from $\mathfrak{C}(A)$, and repeating this finite sequence periodically in both directions produces a bi-infinite hyperbolic sequence. This hyperbolic sequence determines a unique point $x\in \Gamma$, which is periodic, and the hyperbolic branch associated to this periodic return is exactly the branch from $\mathfrak{C}^*(A)$ that we began with. This establishes conclusion \eqref{returns}, and then the proof of Proposition \ref{thm:nicehyp} is completed by the following result, which establishes conclusion \eqref{saturated}.

\begin{Lemma}
\(  \Gamma  \) is saturated. 
\end{Lemma}

\begin{proof}
Given  a branch $f^i \colon \widehat{C}^s_{ij} \to \widehat{C}^u_{ij}$ from $\mathfrak{C}_0(\Gamma) = \mathfrak{C}^*(A)$, for every
\( w\in f^{-i}(\widehat C^{u}_{ij}\cap C^{s}) \) we have  \( f^{i}(w)\in C^{s} \) and therefore \( f^{i}(w) \) has a forward hyperbolic sequence \( \{i_{m}j_{m}\}_{m\geq 0} \).
The branch $f^i \colon \widehat{C}^s_{ij} \to \widehat{C}^u_{ij}$ is a concatenation of finitely many branches from $\mathfrak{C}(A)$; adding these branches to the start of the forward hyperbolic sequence for $f^i(w)$ produces a forward hyperbolic sequence for $w$.
It follows that \( w\in C^{s} \), thus proving the first inclusion in \eqref{eq:saturation}.  A completely analogous argument works for the unstable leaves to show that \(  \Gamma  \) is saturated.
\end{proof}

\subsection{Proof of Proposition \ref{prop:sat-tower}}\label{sec:building-tower}

Let \( \Gamma = C^s\cap C^u\) be a nice T-recurrent saturated rectangle satisfying the \( (C, \kappa) \)-hyperbolic branch property and  let
$\mathfrak{C}_0 = \mathfrak{C}_0(\Gamma)$
denote the collection of hyperbolic branches associated to returns to~\( \Gamma \), indexed by \( I_0 \).
Recall the definition of $s$-subsets and $u$-subsets from Definition \ref{def:su-sets}.
To each branch $f^i \colon \widehat C^s_{ij} \to \widehat C^u_{ij}$ in $\mathfrak C_0$ we will associate an $s$-subset $\Gamma^s_{ij} \subset \widehat C^s_{ij} \cap \Gamma$ and a $u$-subset $\Gamma^u_{ij} \subset \widehat C^u_{ij} \cap \Gamma$ such that $f^i\colon \Gamma^s_{ij} \to \Gamma^u_{ij}$ is a bijection.  We stress that both inclusions are in general proper; roughly speaking, the reason for this is that there may be some $x\in \widehat C^s_{ij} \cap \Gamma$ for which $f^i(x)\notin \Gamma$, and such points must be excluded from $\Gamma^s_{ij}$; similarly for $x\in \widehat C^u_{ij}$ and $f^{-i}(x)$.
To define $\Gamma_{ij}^{s,u}$, first recall from \eqref{eq:saturation} that for $ij\in I_0$ we write
\begin{equation}\label{eqn:Cij}
\begin{aligned}
C_{ij}^s &= f^{-i}(\widehat C_{ij}^u \cap C^s)
= \widehat C_{ij}^s \cap f^{-i}(C^s), \\
C_{ij}^u &= f^i(\widehat C_{ij}^s \cap C^u)
= \widehat C_{ij}^u \cap f^i(C^u);
\end{aligned}
\end{equation}
then let
\begin{equation}\label{Gij}
\Gamma^{s}_{ij}:=  C^{s}_{ij}\cap C^{u},
\qquad \qquad
 \Gamma^{u}_{ij}:= C^{u}_{ij}\cap C^{s}.
\end{equation}
Notice that  \( C^{s}_{ij}, C^{u}_{ij}  \) are collections of stable and unstable leaves respectively, whereas  \( \Gamma^{s}_{ij}, \Gamma^{u}_{ij} \) may be  Cantor sets. 

\begin{Lemma}\label{prop:MY} 
For every \(  ij\in I_0  \), \(  \Gamma^s_{ij}, \Gamma^u_{ij}  \) are \(  s  \)-subsets and \(  u  \)-subsets respectively of \(  \Gamma  \) and 
\(
 f^{i}(\Gamma^{s}_{ij}) = \Gamma^{u}_{ij}.
 \)
Moreover, if $x\in \Gamma$ and $i\in T\NN$ are such that $f^i(x)\in \Gamma$, then $x\in \Gamma_{ij}^s$ for some $ij\in I_0$; in particular this implies that $\Gamma = \bigcup_{ij\in I_0} \Gamma_{ij}^s = \bigcup_{ij\in I_0}\Gamma_{ij}^u$.
\end{Lemma}

\begin{proof}
By the saturation assumption,  \(   C^{s}_{ij} \subseteq C^s,   C^{u}_{ij} \subseteq C^u \)   and therefore 
 \( 
 \Gamma^s_{ij}:=  C^{s}_{ij}\cap C^{u} \subseteq C^s\cap C^u = \Gamma\) and \(  \Gamma^u_{ij}:=  C^{u}_{ij}\cap C^{s} \subseteq C^u\cap C^s = \Gamma
  \) 
and so  \( \Gamma^s_{ij}, \Gamma^u_{ij} \subseteq \Gamma \). Since
\(  C^{s}_{ij}  \) is a union of stable leaves and \(  C^{u}_{ij}  \) is a union of unstable leaves, the sets \(  \Gamma^s_{ij}:=  C^{s}_{ij}\cap C^{u}   \) and \( \Gamma^u_{ij}:=  C^{u}_{ij}\cap C^{s} \) are \(s\)-subsets and \(  u  \)-subsets, respectively, of \(\Gamma\). 
Moreover, directly from the definitions we have   
\[
f^{i}(\Gamma^{s}_{ij})=f^{i}(\widehat C^{s}_{ij})\cap C^{s}\cap f^{i}(C^{u})
=f^{i}(\widehat C^{s}_{ij}\cap C^{u})\cap C^{s}= \Gamma^{u}_{ij}.
\]
For the second statement, let $x\in \Gamma$ and $i\in T\NN$ be such that $f^i(x)\in \Gamma$.  Since $\Gamma$ 
has the $(C,\kappa)$-hyperbolic branch property,
 there is a hyperbolic  branch \(  f^{i}\colon \widehat C^{s}_{ij}\to  \widehat C^{u}_{ij} \) in \( \mathfrak C_0 \) such that 
\(  x\in\widehat C^s_{ij}\cap \Gamma \) and \(  f^i(x)\in
\widehat C^u_{ij}\cap\Gamma.
\)
Since \(  \Gamma\subset C^{s}  \), the definition of \(  C^{s}_{ij}  \) gives
 \[
  x = f^{-i}(f^i(x))\in f^{-i}( \widehat C^u_{ij} \cap \Gamma )\subset  f^{-i}( \widehat C^u_{ij} \cap C^s )= C^s_{ij}.
   \]
 Since we also have \(  x\in \Gamma = C^s\cap C^u\subseteq C^{u}  \) it follows that \(  x\in C^s_{ij}\cap C^u = \Gamma^s_{ij}  \). The final assertion follows because $\Gamma$ is recurrent and so every $x\in \Gamma$ has some $i,i'\in T\NN$ such that $f^i(x), f^{-i'}(x) \in \Gamma$.
\end{proof}

Lemma \ref{prop:MY} shows that there exists a \emph{cover} of \(  \Gamma  \) by \(  s  \)-subsets and another by \(  u  \)-subsets satisfying the Markov property \ref{eq:youngtower} required in the definition of topological Young tower.
It does not however claim that this cover is a \emph{partition} of \(  \Gamma  \) as required by the definition, i.e.\ that the \(  s  \)-subsets of the cover are disjoint which in fact they may not be. The following Lemma shows however that they are pairwise either nested or disjoint, and this will then allow us to choose a sub-cover made up of pairwise disjoint sets (see the proof of Proposition  \ref{prop:sat-tower} below).

\begin{Lemma}\label{prop:nested}
 Let 
 \(  k\ell\in I_0 \) and  suppose there exists \(  x\in \Gamma^s_{k\ell}\) and \( 0<    i<k  \) such that \(  f^i(x) \in\Gamma \). Then there exist \(  ij\in I_0 \)  such that 
\(
  \Gamma^s_{k\ell}\subseteq \Gamma^s_{ij}. 
 \) 
In particular all \(  \{\Gamma^{s}_{ij}\}_{ij\in I_0}  \) are pairwise either nested or disjoint.
\end{Lemma}

We will prove this lemma momentarily.

\begin{Remark}
Notice that the last statement in  Lemma \ref{prop:nested} does not follow directly from  Lemma \ref{lem:nestdisj}. Indeed, the fact that two stable strips \( \widehat C^s_{ij},  \widehat C^s_{k\ell}\) are nested does not \emph{a priori} imply that the corresponding sets \( C^s_{ij},   C^s_{k\ell}\) are either disjoint or nested, recall \eqref{eq:saturation},  and therefore also does not \emph{a priori} imply that \(
  \Gamma^s_{ij}, \Gamma^s_{k\ell}
 \) are either disjoint or nested, recall \eqref{Gij}. 
\end{Remark}

\begin{figure}[tbp]
\includegraphics[width=\textwidth]{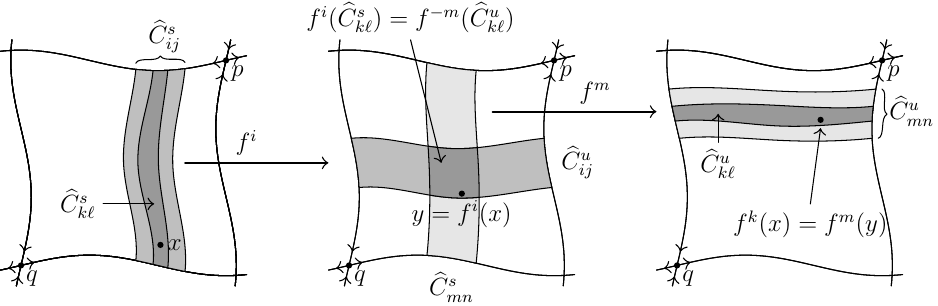}
\caption{Proof of Sublemma \ref{sublem:nested4}.}
\label{fig:ijklmn}
\end{figure}

\begin{Sublemma} \label{sublem:nested4}
In the setting of Lemma \ref{prop:nested},
letting \(  m=k-i  \), there exist \(  ij, mn \in I_0 \) such that  
\(  \widehat C^s_{k\ell} \subseteq \widehat C^s_{ij}  \), 
\(  \widehat C^u_{k\ell} \subseteq \widehat C^u_{mn}  \),
and such that 
\( 
f^i(\widehat C^s_{k\ell}) = \widehat C^u_{ij}\cap \widehat C^s_{mn} = f^{-m}(\widehat C^u_{k\ell}). 
\) 
\end{Sublemma}

\begin{proof}
From   Lemma  \ref{prop:MY} we have \(  x\in \Gamma^s_{ij}  \) for some \(ij \in I_0 \). Therefore \(  x\in \Gamma^s_{ij}\cap\Gamma^s_{k\ell}  \) and so 
\(  x\in \widehat C^s_{ij}\cap \widehat C^s_{k\ell}  \) and  in particular 
\(
\widehat C^s_{ij}\cap \widehat C^s_{k\ell}  \neq\emptyset
\) and therefore, by Lemma \ref{lem:nestdisj},  \(   \widehat C^s_{k\ell} \subseteq \widehat C^s_{ij}  \), as shown in Figure \ref{fig:ijklmn}. 
Also,  from Lemma  \ref{prop:MY},  \(  x\in \Gamma^{s}_{k\ell}  \) implies  \(  x\in\Gamma  \) and 
\(  f^{k}(x)\in \Gamma  \) and therefore, 
letting \(  y=f^{i}(x)\in\Gamma  \) we have  \(  f^{m}(y)=f^{m}(f^{i}(x))=f^{k}(x) \in \Gamma  \). Thus  there exists \(  mn\in I_0  \) such that \(  y\in \widehat C^{s}_{mn}  \).  We therefore have \(  f^{m}(y) \in f^{m}(\widehat C^{s}_{mn}) = \widehat C^{u}_{mn}  \) and also  \(  f^{m}(y)=f^{k}(x)\in \Gamma^{u}_{k\ell}\subseteq \widehat C^{u}_{k\ell}  \), and therefore
\(  \widehat C^{u}_{mn}\cap \widehat C^{u}_{k\ell} \neq\emptyset \) and thus, since \(  m< k  \), \( \widehat C^{u}_{k\ell} \subseteq \widehat C^{u}_{mn} \). 
Then, since \(  \widehat C^s_{k\ell} \subseteq \widehat C^s_{ij}  \) are both full height vertical (stable) strips and \(  \widehat C^u_{k\ell} \subseteq \widehat C^u_{mn}  \) are both full length horizontal (unstable) strips, it follows  that \(  f^i(\widehat C^s_{k\ell})  \) is ``full height''  relative to the horizontal strip \(  \widehat C^u_{ij} \) and   \(  f^{-m}(\widehat C^u_{k\ell})  \) is ``full width'' relative to the vertical strip \(  \widehat C^s_{mn} \). 
Since \(   f^i(\widehat C^s_{k\ell})  =   f^{-m}(\widehat C^u_{k\ell})  \) we complete the proof.
\end{proof}

\begin{proof}[Proof of Lemma \ref{prop:nested}]
Let $m=k-i$.  Directly from the definitions,
\begin{equation}\label{eq:nested3}
C^s_{k\ell}=f^{-k}(\widehat C^u_{k\ell} \cap C^s)  
= f^{-i}(f^{-m}(\widehat C^u_{k\ell} \cap C^s)).
\end{equation}
From Sublemma \ref{sublem:nested4} we have \(  f^{-m}(\widehat C^u_{k\ell})  =  \widehat C^u_{ij}\cap \widehat C^s_{mn} \) and thus \eqref{eqn:Cij} gives
\[
f^{-m}(\widehat C^u_{k\ell}\cap C^s)=\widehat C^u_{ij}\cap\widehat C^s_{mn}\cap f^{-m}(C^s)=\widehat C^u_{ij} \cap C^s_{mn}.
\] 
Substituting this into \eqref{eq:nested3} and using the saturation condition which implies \(   C^s_{mn}\subseteq C^s   \), we get 
\(
C^s_{k\ell} = f^{-i}(\widehat C^u_{ij} \cap C^s_{mn})\subseteq f^{-i}( \widehat C^u_{ij} \cap C^s) =: C^s_{ij} 
\),  which implies the statement
since $\Gamma_{k\ell}^s = C_{k\ell}^s \cap C^u \subseteq C_{ij}^s \cap C^u = \Gamma_{ij}^s$ by \eqref{Gij}.
\end{proof}

\begin{proof}[Proof of Proposition  \ref{prop:sat-tower}]
Since the family of sets \(  \{\Gamma^{s}_{ij}\}_{ij\in I_0}  \)  are pairwise either nested or disjoint,   they are partially ordered by inclusion.  We can therefore define the set
\(  I'_0\subset I_0  \) of indices \(  ij  \) which are \emph{maximal} with respect to this partial order. We then let 
\(
\mathcal P:= \{\Gamma^s_{ij}\}_{ij\in I'_0}.
\)
By  Lemma  \ref{prop:MY}, every point \(  x\in \Gamma^{s}  \) belongs to some  \(  \Gamma^s_{ij}  \) for some \(  ij\in I_0  \) and therefore must also belong to some maximal element \(  \Gamma^s_{ij}  \) for some \(  ij\in I'_0  \). Thus \(  \mathcal P  \) is a partition of \(  \Gamma  \) into \emph{pairwise disjoint} \(  s  \)-subsets whose images are \(  u  \)-subsets. This gives the Markov--Young structure. 
To see that it is a First Return Topological Young Tower we suppose by contradiction that there exists some \(  k\ell\in I'_0  \),  \(  x\in \Gamma^s_{k\ell}  \) and  \(  0< i < k  \) such that \(  f^i(x)\in \Gamma  \). Then   
Lemma \ref{prop:nested} implies that there exists some \(  ij\in I_0  \) such that \(  \Gamma^s_{k\ell}\subset \Gamma^s_{ij}  \), contradicting the maximality of \(  \Gamma^s_{k\ell}  \).  
\end{proof}

\subsection{Proof of Proposition \ref{prop:nicefat-tower}}\label{sec:fatyoung}
We split the proof into two independent parts, one to prove the hyperbolicity and distortion conditions \ref{Y1}--\ref{Y2} 
using the hyperbolic branch property and H\"older estimate \eqref{eqn:fa-Gamma}, and the second to prove the integrability of the return times, which follows from the Lebesgue-strong T-return property. 

\subsubsection{Hyperbolicity and distortion properties of the tower}

We will verify Conditions \ref{Y1} and \ref{Y2} in Definition \ref{def:young-tower}. 
Fix $i\in T\NN$, 
$j\in\{1,\dots,m_i\}$, and $x\in\Gamma^s_{ij}$. Note that the map $F=f^{i}\colon\Gamma^s_{ij}\to\Gamma^u_{ij}$ has the $(C,\kappa)$-hyperbolic branch property with constant $C>0$ and $\kappa>0$ independent of $x$. 
Since the number $T$ is large enough, (see Condition 3 in Definition \ref{def:nice}) it ensures that $Ce^{-\kappa T}<1$ and Condition \ref{Y2a} follows with $\kappa_1 = C e^{-\kappa T}$. Condition \ref{Y2b} can be shown by a similar argument.

We now prove Condition \ref{Y3a}, the proof of Condition \ref{Y3b} is similar. It suffices to show that there are constants $c>0$ and $\alpha_1>0$ such that for any $z\in\Gamma$ and $w\in V^s_z$ we have
\begin{equation}\label{bounded-dist}
\Big |\log\frac{\text{Jac}^uF(z)}{\text{Jac}^uF(w)}\Big |\le cd(z,w)^{\alpha_1}.
\end{equation}
Indeed, setting $z=F^n(x)$ and $w=F^n(y)$, the desired bounded distortion estimate follows from \eqref{bounded-dist} and Condition \ref{Y2a}. 

To show \eqref{bounded-dist} notice that $z\in\Gamma^s_{ij}$ for some 
$i\in T\NN$ and $j\in\{1,\dots,m_i\}$ and hence,
\begin{equation}\label{equat1}
\Big |\log\frac{\text{Jac}^uF(z)}{\text{Jac}^uF(w)}\Big |
=\Big |\sum_{a=0}^{i-1}\log\frac{\text{Jac}^uf(f^a(z))}{\text{Jac}^uf(f^a(w))}\Big |.
\end{equation}
Since $\Gamma$ is a $(C_0,\kappa)$-rectangle, we see that for
any $z\in \Gamma$, $w\in V_z^s$, and $0\leq a\leq i-1$, we have $f^a(w)\in V^s_{f^a(z)}$ and 
\begin{equation}\label{equat2}
d(f^a(z),f^a(w))\le 
C_0 e^{-\kappa a} d(z,w).
\end{equation} 
Moreover, using the assumption that $\Gamma$ satisfies \eqref{eqn:fa-Gamma}, we see that
\begin{equation}\label{eqn:dEu}
\begin{aligned}
d(E_{f^a(z)}^u, &E_{f^a(w)}^u) \leq C_1 e^{\gamma_1 a} d(f^a(z),f^a(w))^{\beta_1} \\
&\leq C_1 e^{\gamma_1 a} C_0^{\beta_1} e^{-\kappa a \beta_1} d(z,w)^{\beta_1}
= C_1 C_0^{\beta_1} e^{(\gamma_1 - \kappa\beta_1) a} d(z,w)^{\beta_1}.
\end{aligned}
\end{equation}
Since $f$ is $C^{1+\alpha}$, it follows that there is $C_2>0$ such that
\begin{equation}\label{eqn:Jacu}
\begin{aligned}
\Big |\frac{\text{Jac}^uf(f^a(z))}{\text{Jac}^uf(f^a(w))}-1\Big |
&=\Big |\frac{\text{Jac}^uf(f^a(z))-\text{Jac}^uf(f^a(w))}{\text{Jac}^uf(f^a(w))}\Big |\\
&\leq C_2 d(E_{f^a(z)}^u,E_{f^a(w)}^u)^\alpha \\
&\leq C_2 C_1^\alpha C_0^{\beta_1\alpha} e^{\alpha(\gamma_1 - \kappa \beta_1) a} d(z,w)^{\beta_1 \alpha}.
\end{aligned}
\end{equation}
If $\text{Jac}^uf(f^a(z)) \geq \text{Jac}^uf(f^a(w))$ this gives
\[
\log\frac{\text{Jac}^uf(f^a(z))}{\text{Jac}^uf(f^a(w))}
\leq \frac{\text{Jac}^uf(f^a(z))}{\text{Jac}^uf(f^a(w))} - 1 \leq C_2 C_1^\alpha C_0^{\beta_1\alpha} e^{\alpha(\gamma_1 - \kappa \beta_1) a} d(z,w)^{\beta_1 \alpha}.
\]
The same bound holds with $z$ and $w$ exchanged, by exchanging their roles in \eqref{eqn:Jacu}. Then by \eqref{equat1} we get
\begin{equation}\label{eqn:Jacu-2}
\Big |\log\frac{\text{Jac}^uF(z)}{\text{Jac}^uF(w)}\Big |
\leq \sum_{a=0}^{i-1} C_2 C_1^\alpha C_0^{\beta_1\alpha} e^{\alpha(\gamma_1 - \kappa \beta_1) a} d(z,w)^{\beta_1 \alpha}.
\end{equation}
Since we assumed $\gamma_1 < \beta_1 \kappa$, we see that $\gamma_1 - \kappa \beta_1 < 0$, and thus \eqref{eqn:Jacu-2} implies \eqref{bounded-dist}, which proves Condition \ref{Y3a}.

\subsubsection{Integrability of the return times}
Let \( F\colon \Gamma \to \Gamma \) be the induced map to the base of the Topological First T-Return Young Tower where \( F(x)=f^{\tau(x)}(x) \) and  \( \tau(x) \) is the first return time to \( \Gamma \) which is a multiple of \( T \). 

For every \( n\geq 0 \) let 
\[
R_{n}(x) := \sum_{j=0}^{n-1} \tau (F^{j}(x)) \quad \text{ and } \quad v_{n}(x) := \#\{0 < i \leq  n/T: f^{iT}(x)\in \Gamma\} 
\]
where we define \( R_{0}(x)=0 \) by convention. 
Notice that  \( R_{n}, v_{n} \) are  on quite different time scales, the index \( n \) in \( v_{n} \)   refers to the iterates of the original map \( f \),  whereas in \( R_{n} \)  it refers to  the iterates of the induced map \( F \). The following relation between the two quantities is not surprising but neither is it completely trivial, we thank Vilton Pinheiro for explaining it to us. 

\begin{Lemma}[Pinheiro \cite{Pin19}] 
\label{lem:pin}
Let  \( x\in \Gamma \) and suppose 
\( \displaystyle{\lim_{n\to\infty} {R_{n}(x)}/{n}}  \) exists. Then 
\[
\lim_{n\to\infty} \frac{v_{n}(x)}{n}
= \left(\lim_{n\to\infty} \frac {R_{n}(x)}{n} \right)^{-1}.
\]
In particular \( \displaystyle{\lim_{n\to\infty}{v_{n}(x)}/{n} }\) exists and is equal to 0 if \( \displaystyle{\lim_{n\to\infty} {R_{n}(x)}/{n}}=\infty \).
\end{Lemma}

\begin{proof}[Proof of Lemma \ref{lem:pin}]
For any \( x\in \Gamma \), by definition of \( v_{n}(x) \) we have 
\[
v_{n}(x) :=
 \#\{0 < i \leq  n/T: f^{iT}(x)\in \Gamma\}  =\max \Big\{k\geq 0: \sum_{j=0}^{k-1} \tau (F^{j}(x)) \leq  n  \Big\}
\]
and so, for every \( n\geq  \tau(x) \), so that \( v_{n}(x)\geq 1 \),  we have 
\[
R_{v_{n}(x)}(x):=\sum_{j=0}^{v_{n}(x)-1} \tau (F^{j}(x)) 
\leq  n < 
\sum_{j=0}^{v_{n}(x)} \tau (F^{j}(x))=:R_{v_{n}(x)+1}(x).
\]
Dividing through by \( v_{n}(x) \) this gives 
\begin{equation}\label{eq:limits}
\frac{R_{v_{n}(x)}(x)}{v_{n}(x)}
\leq  \frac{n}{v_{n}(x)} < 
\frac{v_{n}(x)+1}{v_{n}(x)} \frac{R_{v_{n}(x)+1}(x)}{v_{n}(x)+1} 
\end{equation}
Since \(({v_{n}(x)+1})/{v_{n}(x)} \to 1 \) as \( n\to \infty \) and the limit of the sequence \( R_{n}(x)/n \) exists, the subsequences on the left and right hand side of  \eqref{eq:limits} also converge to the same limit. It follows that the \( n/v_{n}(x) \) converges and therefore also \( v_{n}(x)/n \) to a limit as in the statement. 
\end{proof}

Lemma \ref{lem:pin}  leads to the integrability of the return times stated in Definition \ref{def:integrability}. Indeed, by the results in \cite{lY98} the induced map \( F\colon \Gamma \to \Gamma \) admits an SRB measure \( \hat\mu \) whose conditional measures \( \hat\mu_{z} \) on unstable  curves  of points of \( \Gamma \) are equivalent to the Lebesgue measure \( m_{V_z^u} \) on these same curves. It is therefore sufficient to show the integrability with respect to one of these conditional measures. 
For $\hat\mu$-a.e.\ $z\in \Gamma$ 
there is $E_z \subset V_z^u$ with $\hat\mu_z(E_z^c)=0$ such that
\begin{equation}\label{eqn:Ez}
\lim_{n\to\infty} \frac{R_n(x)}n = \int_\Gamma \tau \,d\hat\mu = \int_\Gamma \int \tau \,d\hat\mu_z \,d\hat\mu(z)
\text{ for all } x \in E_z,
\end{equation}
where a priori these quantities may be infinite. However, 
since $\Gamma$ is Lebesgue-strongly $T$-recurrent, there is a local unstable curve $V^u$ and a set $E \subset V^u \cap \Gamma$ of positive 1-dimensional Lebesgue measure such that for every $x\in E$ we have $\limsup_{n\to\infty} v_n(x) / n > 0$ (see \eqref{eqn:pos-freq}). Now Lemma \ref{lem:pin} implies that if $x\in E$ is such that $\lim_{n\to\infty} R_n(x)/n$ exists, then this limit is finite.

Observe that for each $n\geq 1$ and $y\in \Gamma$, the function $x\mapsto R_n(x) / n$ is constant on $V_y^s \cap \Gamma$. In particular, fixing $z\in \Gamma$ and writing $E_z' = \bigcup_{y\in E} V_y^s \cap V_z^u$, we see that $E_z' \subset V_z^u$ has $\hat\mu_z(E_z')>0$
and that $\lim_{n\to\infty} R_n(x)/n$ is finite for every $x\in E_z'$ such that the limit exists.

Choosing $z\in \Gamma$ such that $\hat\mu_z(E_z^c)=0$, we get $\hat\mu_z(E_z' \cap E_z) > 0$. 
For any point $x \in E_z' \cap E_z$, \eqref{eqn:Ez} gives $\int_\Gamma \tau \,d\hat\mu < \infty$, and it follows that $\int \tau \,d\hat\mu_z < \infty$ for $\hat\mu$-a.e.\ $z$.
This completes the proof of the integrability of the return times and thus the proof of Proposition \ref{prop:nicefat-tower}.

\subsection{Discussion of the saturation condition}\label{sec:sat-disc}

The saturation condition in Definition \ref{def:satrect} is crucial to our construction of a first $T$-return Young tower, but may appear mysterious at first glance, so we give some general discussion here of the role it plays; this section is not part of the proofs. 

The strategy in the proof of Theorem \ref{thm:existsY} was to use a collection of hyperbolic branches $\mathfrak{C}$ to produce a rectangle $\Gamma$ (Proposition \ref{thm:nicehyp}), and then to build a Young tower that has $\Gamma$ as its base (Propositions \ref{prop:sat-tower} and \ref{prop:nicefat-tower}). In \S\ref{sec:sat-rec}, Proposition \ref{thm:nicehyp} was proved by taking $\mathfrak{C}$ to be the collection of branches associated to almost returns of some nice regular set $A \subset \Gpq$, but Lemma \ref{lem:nicesat} shows that the construction of a nice rectangle $\Gamma$ that is recurrent works if we begin with \emph{any} collection $\mathfrak{C}$ of hyperbolic branches.

The proof of conclusions \eqref{branch}, \eqref{returns}, and \eqref{saturated} in Proposition \ref{thm:nicehyp}, however, required us to work with the specific choice of $\mathfrak{C}$ given by the collection of almost returns. To illustrate why this was important, consider the following example, which shows how \eqref{returns} and \eqref{saturated} can fail for a general collection $\mathfrak{C}$, where the collection of hyperbolic branches associated to returns of $\Gamma$ may contain some branches that do not appear as concatenations of branches from the generating set $\mathfrak{C}$. 

\begin{Example}\label{eg:notsat}
Suppose that $p,q$ are fixed points and that $\Gpq$ contains two hyperbolic branches $f\colon \widehat{C}_0^s \to \widehat{C}_0^u$ and $f\colon \widehat{C}_1^s \to \widehat{C}_1^u$. Define a map $\pi\colon \{0,1\}^\ZZ \to \Gpq$ by $\pi(x) = \bigcap_{n\in\ZZ} f^{-n}(\widehat{C}_{x_n}^s)$. To each finite word $w = w_0 w_1 \cdots w_{n-1} \in \{0,1\}^n$ we can associate the cylinder $[w] = \{ x \in \{0,1\}^\ZZ : x_j = w_j$ for all $0\leq j < n\}$, and $\pi[w]$ is the stable strip for a hyperbolic branch of order $n$ (recall Definition \ref{def:hypbranch}) associated to the word $w$.

Let $\mathfrak{C}$ be the collection of two hyperbolic branches associated to the words $v=01$ and $w=10$. Applying Lemma \ref{lem:nicesat} to $\mathfrak{C}$ produces the rectangle $\Gamma$ consisting of all points $\pi(x)$ for which $x_{2n}x_{2n+1} \in \{v,w\}$ for all $n\in\ZZ$. Observe that for $x = \overline{10}.\overline{10} \in \{0,1\}^\ZZ$ (where the overline denotes infinite repetition to either the left or the right), we have $\pi(x) \in \Gamma$ and $f(\pi(x)) \in\Gamma$ (this is a period-2 orbit), but the hyperbolic branch corresponding to this return is $f \colon \widehat{C}^s_1 \to \widehat{C}_1^u$, which is not a concatenation of branches from $\mathfrak{C}$, so conclusion \eqref{returns} fails. Moreover, the rectangle $\Gamma$ is not saturated, because $z= \pi(\overline{01}.\overline{10}) \in \Gamma$ has
\[
f^{-1}(\widehat{C}^u_1 \cap W_z^s) = \{ \pi(y) : y_0y_1y_2 \cdots = 1\overline{10}\},
\]
which is a full-length local stable curve that is disjoint from $\Gamma$ and thus does not lie in $C^s$, violating \eqref{eq:saturation}. Then the construction of a topological Young tower in \S\ref{sec:building-tower} fails: Lemma \ref{prop:MY} does not go through because the sets defined in \eqref{Gij} are not contained in $\Gamma$.
\end{Example}

In fact one can still build a ``not first return'' topological Young tower over $\Gamma$ by writing $\Gamma_v^s = \{ \pi(y) : y_0 y_1 = v\}$ and $\Gamma_w^s = \{ \pi(y) : y_0 y_1 = w\}$, then setting the inducing time $\tau$ to be equal to $2$ everywhere so that $\Gamma_v^u = f^2(\Gamma_v^s) = \{ \pi(y) : y_{-2} y_{-1} = v\}$, and similarly for $\Gamma_w^{u}$.
This gives the Markov structure in \ref{eq:youngtower}, but is not a first return tower because some points, such as the point $\pi(x)$ given in Example \ref{eg:notsat}, return before time $2$.

One might reasonably object that the rectangle in Example \ref{eg:notsat} is indeed saturated if we take $T=2$ and only consider returns at even times, since then every $T$-return is associated to a concatenation of the branches in $\mathfrak{C}$, and thus $\Gamma$ supports a first $2$-return topological Young tower. This phenomenon occurs if we start with any finite collection of branches $\mathfrak{C}$: taking $T$ to be any common multiple of their lengths will guarantee that the rectangle is saturated (for $T$-returns) and produces a first $T$-return topological Young tower. This simple solution fails when $\mathfrak{C}$ is infinite, as in the next example.

\begin{Example}\label{eg:notsat-2}
Given three hyperbolic branches $\{f\colon \widehat{C}^s_j \to \widehat{C}^u_j\}_{j\in \{0,1,2\}}$, define $\pi \colon \{0,1,2\}^\ZZ \to \Gpq$ as in Example \ref{eg:notsat}, and let $\mathfrak{C}$ be the infinite collection of hyperbolic branches associated to words of the form $210^n1$ and $10^n12$ for $n\geq 0$. Then given any $i\geq 4$, we observe that $x = \overline{210^{i-4}1}.\overline{210^{i-4}1}$ has the property that $\pi(x),f^i(\pi(x)) \in \Gamma$, so for $w = 210^{i-4}12 \in \{0,1,2\}^i$, the branch $f^i \colon \widehat{C}^s_w = \widehat{C}^u_w$ is associated to a return to $\Gamma$, but is not a concatenation of branches from $\mathfrak{C}$. Putting $z = \pi(\overline{10^{i-4}12}.\overline{210^{i-4}1})$ one can check that $f^{-i}(\widehat{C}^u_w \cap W_z^s) \not\subset C^s$, so \eqref{eq:saturation} fails and $\Gamma$ is not saturated for any choice of $T$. Because the stable strips associated to the branches in $\mathfrak{C}$ are disjoint, one can still build a topological Young tower (with infinitely many branches) as in the paragraph following Example \ref{eg:notsat}, but it will not be a first $T$-return tower for any choice of $T$.
\end{Example}

The reason we can avoid these problems in Proposition \ref{thm:nicehyp} is that the collection of branches we use there is not arbitrary: rather, it is the collection of \emph{all} hyperbolic branches associated to almost returns to a particular nice regular set $A$. For such a collection, Lemmas \ref{lem:almostret} and \ref{lem:any-almost} guarantee that any $T$-return for the rectangle $\Gamma$ corresponds to an almost return for the original set $A$, and thus to a branch from the original collection. Thus one may interpret the saturation condition as the requirement that the original collection of branches be ``large enough'' that no new branches are created by ``accidental returns''.

For a discussion of a related issue in the setting of coded shift spaces, see \cite[\S3]{vC18}, and specifically condition [III*] of that paper, which plays the same role there as saturation does here, guaranteeing that a certain Markov coding is 1-1. 
Example 3.6 of that paper could be translated into the setting of Examples \ref{eg:notsat} and \ref{eg:notsat-2} by letting $\mathfrak{C}$ be the collection of four hyperbolic branches associated to the words $\{0, 10, 01, 101\}$; then the associated rectangle is not saturated, and moreover the procedure described after Example \ref{eg:notsat} does not produce a topological Young tower at all since the $s$-sets $\Gamma_w^s$ are not disjoint. One can still obtain a Markov coding but it is uncountable-to-1. Together with the previous two examples, this illustrates how important the saturation condition is for our results.

\section{Hyperbolic measures have nice regular sets: Proof of Theorem \ref{thm:existsnice1}}\label{sec:B->A}

In this section we prove Theorem \ref{thm:existsnice1}. The non-trivial part of the proof is to show that we can find arbitrarily small  domains \(  \Gpq  \) with \(  \mu(\Gpq\cap \Lambda_{\ell})>0  \) where \(  p,q\in \Lambda_{\ell}  \) are periodic points. Then letting \(  T >0  \) be any common multiple of the periods of \(  p  \) and \(  q  \),  it follows that \(  p,q  \) are fixed points for \(  f^{T} \) and therefore \(  \Gpq  \) is a nice domain with \(  T(\Gpq)=T  \). Moreover, \(  \mu  \) is also \(  f^{T}  \)-invariant and therefore \(  \mu  \)-a.e.\ \(  x\in \Gpq\cap \Lambda_{\ell}  \) returns to \(  \Gpq  \cap \Lambda_\ell\) with positive frequency for iterates which are multiples of \(  T  \), in both forward and backward time.  
The set $A$ of such points is $T$-recurrent, and if $\mu$ is an SRB measure it is Lebesgue-strongly $T$-recurrent.
Thus   Theorem \ref{thm:existsnice1}  follows from the statement below.

\begin{Proposition}\label{prop:get-pq}
Let $f$ be a $C^{1+\alpha}$ diffeomorphism, $\mu$ an ergodic non-atomic $\chi$-hyperbolic measure, and $\Lambda$ a $\chi$-hyperbolic set.  
Fix $\lambda\in (0,\chi)$ and $\eps \in (0,\eps_1(f,\chi,\lambda))$.
Let $U\subset M$ be an open set and $\ell\in \NN$ such that $\mu(U\cap \Lambda_\ell)>0$.  Then with $\ell'$ as in Theorem \ref{thm:inf-po},
there are  $(\lambda/4,2\eps,\ell+\ell')$-regular periodic points $p,q$ such that $\Gpq$ is defined, contained in $U$, and satisfies $\mu(\Gpq \cap \Lambda_\ell) > 0$.  Since $\Gpq \subset U$, $\diam\Gpq$ can be made arbitrarily small.
\end{Proposition}

Before proving Proposition \ref{prop:get-pq}, we use Theorem \ref{thm:inf-po} to establish a result reminiscent of the Katok Closing Lemma.
Say that $y\in B(x,\delta) \cap \Lambda_\ell$ is \emph{$\Lambda_\ell$-nonwandering} if there is a sequence $n_k\to\infty$ and $y_k\in \Lambda_\ell \cap f^{-n_k} \Lambda_\ell$ such that $y_k,f^{n_k}(y_k) \xrightarrow{k\to\infty} y$.  Observe that by Poincar\'e Recurrence, every point in $\supp(\mu|\Lambda_\ell)$ is $\Lambda_\ell$-nonwandering.  

\begin{Lemma}\label{lem:periodic-bracket}
Given $\delta>0$ as in Theorem \ref{thm:pseudo-orbit},  $\ell'\in\NN$ as in Theorem \ref{thm:inf-po}, and any $\ell\in\NN$, for all $\Lambda_\ell$-nonwandering points $y,z\in B(x,\delta e^{-\lambda\ell}/3)\cap\Lambda_\ell$ there is a sequence of $(\lambda/4,2\eps,\ell+\ell')$-regular periodic points $p_k\xrightarrow{k\to\infty} V_y^s \cap V_z^u$.
\end{Lemma}
\begin{proof}
Suppose $y,z$ are as in the hypothesis.
Choose $n_k\to\infty$ and $y_k \in \Lambda_\ell \cap f^{-n_k} \Lambda_\ell$ such that $y_k, f^{n_k}(y_k) \to y$.  Choose $m_k,z_k$ similarly for $z$.  For suitably large $k$, we have
\[
y_k, f^{n_k}(y_k), z_k, f^{m_k}(z_k) \in B(x,\delta e^{-\lambda \ell}/2)
\]
and thus in particular
\[
d(f^{n_k}(y_k),z_k) \leq \delta e^{-\lambda\ell}
\quad\text{and}\quad
d(f^{m_k}(z_k),y_k) \leq \delta e^{-\lambda\ell}.
\]
It follows that $y_k, f(y_k), \dots, f^{n_k-1}(y_k), z_k, f(z_k), \dots, f^{m_k-1}(z_k), y_k$
is a $(\bar\ell,\delta,\lambda)$-pseudo-orbit with 
\[
\ell_i = \ell + \begin{cases}
\min (i, n_k - i) & 0\leq i\leq n_k, \\
\min (i - n_k, n_k + m_k - i) & n_k \leq i \leq n_k + m_k.
\end{cases}
\]
Repeating this finite pseudo-orbit $\barx$ periodically gives a periodic bi-infinite pseudo-orbit to which we can apply Theorem \ref{thm:inf-po} and obtain a $(\lambda/4,2\eps,\ell+\ell')$-regular periodic shadowing point $p_k$.  Note that $p_k \in \mathcal{N}_{\barx}^0 \cap \mathcal{N}_\barx^{n_k+m_k}$, and that the intersections converge to $V_y^s \cap V_z^u$ as $k\to\infty$ because $y_k \in \mathcal{N}_{\barx}^0$ and $f^{m_k} z_k \in \mathcal{N}_\barx^{n_k+m_k}$.  Thus $p_k \to V_y^s \cap V_z^u$.
\end{proof}

\begin{figure}[tbp]
\includegraphics[width=.9\textwidth]{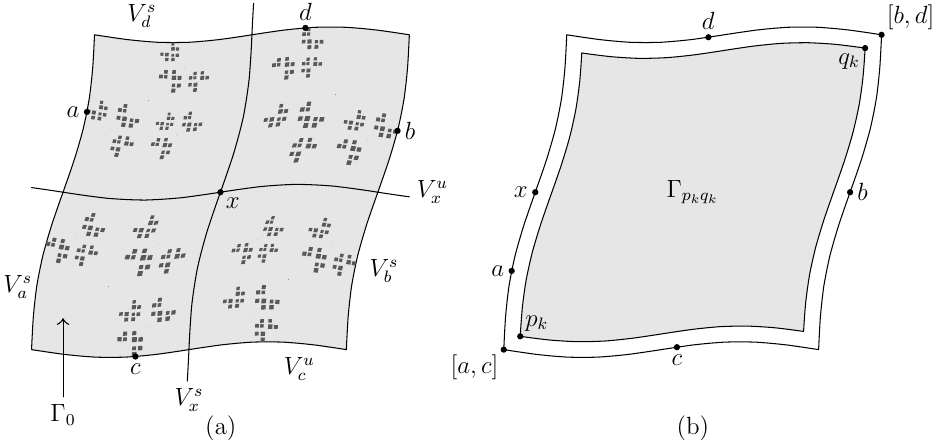}
\caption{Proving Proposition \ref{prop:get-pq}.}
\label{fig:abcd}
\end{figure}

\begin{proof}[Proof of Proposition \ref{prop:get-pq}]
Fix $x\in U\cap \supp(\mu|\Lambda_\ell)$.  Since $\Lambda_\ell$ is closed, we have $\supp(\mu|\Lambda_\ell) \subset \Lambda_\ell$.  Choose $\delta',\delta>0$ sufficiently small  that $B(x,\delta') \subset U$, and such that for every $y,z\in \overline{B(x,\delta)} \cap \Lambda_\ell$, the intersection $V_y^s \cap V_z^u$ is a single point and lies in $B(x,\delta') \cap \Lambda_{\ell'}$.  
Assume also that $\delta$ is chosen small enough and $\ell'$ large enough to satisfy Lemma \ref{lem:periodic-bracket}.

Let $Z := \overline{B(x,\delta)} \cap \supp(\mu|\Lambda_\ell)$.  Observe that $Z$ is compact, and that $\mu(Z)>0$ by our choice of $x$.  
Let $\pi^s \colon Z \to V_x^u$ and $\pi^u\colon Z\to V_x^s$ be projection along local stable and unstable leaves, respectively.  Since $V_x^{s,u}$ are one-dimensional we can equip each with a total order, and by compactness we can choose $a,b,c,d\in Z$ such that
\begin{gather*}
\pi^s(a)=\inf \pi^s(Z),
\quad
\pi^s(b)=\sup \pi^s(Z), \\
\pi^u(c)=\inf \pi^u(Z),
\quad
\pi^u(d) = \sup\pi^u(Z).
\end{gather*}
Let $\Gamma_0$ be the region bounded by $V_a^s$, $V_b^s$, $V_c^u$, and $V_d^u$, as shown in Figure \ref{fig:abcd}.  Observe that $\Gamma_0 \supset Z$ and thus $\mu(\Gamma_0 \cap \Lambda_\ell)>0$.  By Lemma \ref{lem:periodic-bracket} there are periodic points $p_k,q_k\in \Lambda_{\ell'}$ such that $p_k \to V_a^s \cap V_c^u$ and $q_k \to V_b^s \cap V_d^u$.  It is possible that none of the domains $\Gamma_{p_kq_k}$ contains $x$ (this can occur, for example, if $a\in V_x^u$, as in Figure \ref{fig:abcd}(b)); on the other hand, the union $\bigcup_n \Gamma_{p_nq_n}$ covers all of $Z$ except possibly for $Z \cap V_a^s \cup V_b^s \cup V_c^u \cup V_d^u$.  Since $\mu$ is non-atomic, a single local stable or unstable curve always has zero measure, thus this subset is $\mu$-null.  Using the fact that $\mu(Z)>0$, we conclude that there is some $n$ such that $\mu(\Gamma_{p_nq_n} \cap Z) > 0$.  This completes the proof of the proposition.
\end{proof}

\addtocontents{toc}{\vspace{\normalbaselineskip}}
\appendix
\section{List of terminology and notation}

\newcommand{\fullref}[1]{\ref{#1} on page~\pageref{#1}} 

\begin{enumerate}
\item almost returns, Definition \fullref{def:almostreturns}
\item brackets, Definition \fullref{def:brackets}
\item branch
\begin{enumerate}
\item[--] $(C,\kappa)$-hyperbolic, Definition \fullref{def:hypbranch}
\item[--] $\hat{\ell}$-regular, Definition \fullref{def:reg-branch}
\item[--] $(C,\kappa)$-hyperbolic branch property, Definition \fullref{def:hyp-branch-prop}
\end{enumerate} 
\item overlapping charts, Definition \fullref{def:overlapping} 
\item concatenation property, Definition \fullref{def:concatenation}
\item cones, Definition \fullref{def:cones1}
\begin{enumerate}
\item[--] in regular neighbourhoods,  Definition \fullref{def:cones}
\end{enumerate}
\item conefield, Definition \fullref{def:conefield}
\begin{enumerate}
\item[--] adapted, Definition \fullref{def:conefields}
\end{enumerate} 
\item curves
\begin{enumerate}
\item[--] local $(C,\lambda)$-stable (unstable), Definition \fullref{def:localman}
\item[--] $\mathcal{K}$-admissible, Definition \fullref{def:kadmissible}
\item[--] stable and unstable admissible, Definition \fullref{def:nice-adm} 
 \item[--] in regular neighbourhoods, Definition  \fullref{def:localmanreg} 
\item[--] 
full length stable and unstable admissible, Definition \ref{def:nice-adm} and Definition \ref{def:localmanreg}
\end{enumerate}
\item $(\chi,\varepsilon,\ell,r)$-nice domain, Definition \fullref{def:nice}
\item measure
\begin{enumerate}
\item[--] hyperbolic, Definition \fullref{def:nonzero}
\item[--] physical,  Definition \fullref{def:physical}
\item[--] SRB, Definition \fullref{srb-mes}
\end{enumerate}
\item nice
\begin{enumerate}
\item[--] domain, Definition  \fullref{def:nice}
\item[--] regular set, Definition  \fullref{def:nice-set}
\item[--] rectangle, Definition  \fullref{def:nicerectangle}
\end{enumerate}
\item pseudo-orbit
\begin{enumerate}
\item[--] finite $(\hat{\ell},\delta,\lambda)$-, Definition \fullref{def:po}
\item[--] bi-infinite $(\hat{\ell},\delta,\lambda)$-, Definition \fullref{def:infpo}
\end{enumerate}
\item rectangle
\begin{enumerate}
\item[--] $(C,\lambda)$, Definition  \fullref{def:rectangle}
\item[--] nice, Definition  \fullref{def:nicerectangle}
\item[--] saturated, Definition \fullref{def:satrect}
\end{enumerate}
\item recurrence (recurrent)
\begin{enumerate}
\item[--] recurrent and Lebesgue-strongly recurrent, Definition \fullref{def:Treturn}
\item[--] almost recurrent, Definition \fullref{def:almostreturns}
\end{enumerate}
\item regular
\begin{enumerate}
\item[--] level sets, Definition \fullref{def:reg}
\item[--] $(\chi,\varepsilon,\ell)$-regular set, Definition \fullref{def:chi-reg}
\item[--] $(\chi,\varepsilon,\ell,r)$-nice regular set (nice regular), Definition \fullref{def:nice-set}
\item[--] $\hat{\ell}$-regular branch, Definition \fullref{def:reg-branch}
\end{enumerate}
\item sequences, hyperbolic, Definition \fullref{def:hypseq}
\item set
\begin{enumerate}
\item[--] fat, Definition \fullref{def:fat}
\item[--] $(\chi,\varepsilon)$-hyperbolic, Definition  \fullref{def:hypset}
\item[--] regular level, Definition  \fullref{def:reg}
\item[--] $(\chi,\varepsilon,\ell)$-regular, Definition  \fullref{def:chi-reg}
\item[--] $s/u$-subsets, Definition \fullref{def:su-sets}
\item[--] $(\chi,\varepsilon,\ell,r)$-nice regular (nice regular), Definition  \fullref{def:nice-set}
\end{enumerate}
\item stable and unstable strips, 
\begin{enumerate}
\item[--] in a nice domain, Definition \fullref{def:nice-strips} 
\item[--] in regular neighbourhoods, Definition \fullref{def:strips}
\end{enumerate}

\item $T$-returns time, Definition \fullref{def:firstTret}
\item tower
\begin{enumerate}
\item[--] topological Young, Definition \fullref{def:FRYT}
\item[--] Young, Definition  \fullref{def:young-tower}
\end{enumerate}

\end{enumerate}

\bibliographystyle{amsplain}
\bibliography{CLP}

\end{document}